\documentclass[11pt, reqno]{amsart}
\usepackage{amssymb}
\usepackage{eucal}
\usepackage{amsmath}
\usepackage{amscd}
\usepackage[dvips]{color}
\usepackage{multicol}
\usepackage[all]{xy}           %xypic macro for latex2.09
\usepackage{graphicx}
\usepackage{color}
\usepackage{colordvi}
\usepackage{xspace}
\usepackage{ulem}
\usepackage{cancel}
\usepackage{tikz}
\usepackage{longtable}
\usepackage{multirow}
\usepackage{ifpdf}
\ifpdf
    \usepackage[colorlinks,final,hyperindex]{hyperref} \else
    \usepackage[colorlinks,final,hyperindex,hypertex]{hyperref}
\fi

\makeatletter
\@namedef{subjclassname@2020}{\textup{2020} Mathematics Subject Classification}
\makeatother

\begin{document}
\newcommand {\emptycomment}[1]{} %to remove paragraphs

\newcommand{\tabincell}[2]{\begin{tabular}{@{}#1@{}}#2\end{tabular}}

\newcommand{\nc}{\newcommand}
\newcommand{\delete}[1]{}
\nc{\mfootnote}[1]{\footnote{#1}} % Use this to show footnotes
\nc{\todo}[1]{\tred{To do:} #1}

\delete{
\nc{\mlabel}[1]{\label{#1}}  % Use this to suppress names
\nc{\mcite}[1]{\cite{#1}}  % Use this to suppress names
\nc{\mref}[1]{\ref{#1}}  % Use this to suppress names
\nc{\meqref}[1]{\eqref{#1}} % Use this to suppress names
\nc{\bibitem}[1]{\bibitem{#1}} % Use this to show number
}

%\delete{
\nc{\mlabel}[1]{\label{#1}  % Use the next two lines to show names
{\hfill \hspace{1cm}{\bf{{\ }\hfill(#1)}}}}
\nc{\mcite}[1]{\cite{#1}{{\bf{{\ }(#1)}}}}  % Use this lines to show names
\nc{\mref}[1]{\ref{#1}{{\bf{{\ }(#1)}}}}  % Use this lines to show names
\nc{\meqref}[1]{\eqref{#1}{{\bf{{\ }(#1)}}}} % Use this lines to show names
\nc{\mbibitem}[1]{\bibitem[\bf #1]{#1}} % Use this to show name
%}

%%%%%%%%%%%%%%%%%%%%%%%% Statements
\newtheorem{thm}{Theorem}[section]
\newtheorem{lem}[thm]{Lemma}
\newtheorem{cor}[thm]{Corollary}
\newtheorem{pro}[thm]{Proposition}
\newtheorem{conj}[thm]{Conjecture}
\theoremstyle{definition}
\newtheorem{defi}[thm]{Definition}
\newtheorem{ex}[thm]{Example}
\newtheorem{rmk}[thm]{Remark}
\newtheorem{pdef}[thm]{Proposition-Definition}
\newtheorem{condition}[thm]{Condition}

%new commands

%%%%%%%%%%%%%%%%%%%%%%%%%%%%%%%%%%%%%%%%%%%%%%%%%%%%%%%%%%%%%%%%%%

\title[The Levi-Civita products of Leibniz algebras]%yemei %Anti-pre-Leibniz algebras
{The Levi-Civita products of Leibniz algebras with nondegenerate skew-symmetric $2$-cocycles}%first page

    \author{Quan Zhao}
    \address{Chern Institute of Mathematics \& LPMC, Nankai University, Tianjin 300071, China}
    \email{zhaoquan@mail.nankai.edu.cn}

    \author{Guilai Liu}
    \address{Chern Institute of Mathematics \& LPMC, Nankai University, Tianjin 300071, China}
    \email{liugl@mail.nankai.edu.cn}

\date{\today}%

\begin{abstract}
	This paper studies the associated Levi-Civita products of a Leibniz algebra with a nondegenerate skew-symmetric $2$-cocycle.
	Such products form into the notion of an anti-pre-Leibniz algebra, which is characterized as a Leibniz-admissible algebra which renders a representation of the sub-adjacent Leibniz algebra through the negative multiplication operators.
	Such a characterization serves as the converse side of the role that pre-Liebniz algebras play in the splitting theory of Leibniz algebras, which justifies the name of anti-pre-Leibniz algebras.
	There is a compatible anti-pre-Leibniz algebra structure on a Leibniz algebra if and only if there is an invertible anti-$\mathcal{O}$-operator of the Leibniz algebra.
Another  important role that anti-pre-Leibniz algebras play is that they give a new characterization of Novikov dialgebras, that is, a Novikov dialgebra is interpreted as
a transformed pre-Leibniz algebra which gives rise to an anti-pre-Leibniz algebra structure through specific combinations of multiplications.
The properties of Novikov dialgebras are also further investigated.
\end{abstract}

\subjclass[2020]{
    17A36,  %Automorphisms, derivations, other operators (nonassociative rings and algebras)
    17A40,  %Ternary compositions
    %18M70,  %Algebraic operads, cooperads, and Koszul duality
    %%%%17B:Lie algebras and Lie superalgebras
    17B10, %Representations, algebraic theory
    % 17B40, %Automorphisms,derivations,other operators for Lie algebras and super algebras
    %    17B60, %Lie (super)algebras associated with other structures (associative, Jordan, ect.)
    %    17B62, %Lie bialgebras; Lie coalgebras
    %17B63,  %Poisson algebras
    %17D25,  %Lie-admissible algebras
    %37J39,   %Relations of finite-dimensional Hamiltonian and Lagrangian systems with topology, geometry and differential geometry (symplectic geometry, Poisson geometry, etc.
    %53D17  %Poisson manifolds; Poisson groupoids and algebroids
    18M70.  %Algebraic operads, cooperads, and Koszul duality
}

\keywords{Leibniz algebras, Levi-Civita products, skew-symmetric $2$-cocycles, anti-pre-Leibniz algebras, anti-$\mathcal{O}$-operators, Novikov dialgebras }%anti-$\mathcal{O}$-operator
\maketitle

\vspace{-1.3cm}

\tableofcontents

\vspace{-.5cm}

\allowdisplaybreaks

\section{Introduction}

A Leibniz algebra \cite{Blo1,Lod2} is a vector space $A$ together
with a multiplication $\circ_{A}:A\otimes A\rightarrow A$ such that the following equation holds:
\begin{equation}\label{eq:Leibniz}
x\circ_{A}(y\circ_{A}z)=(x\circ_{A}y)\circ_{A}z
+y\circ_{A}(x\circ_{A}z),\;\forall x,y,z\in A.
\end{equation}
%Leibniz algebras were first discovered by Bloh under the name of D-algebras , and  were rediscovered and studied by Loday \cite{Lod2}.
As noncommutative analogues of Lie algebras, Leibniz algebras form an important class of non-associative algebras due to their broad connections and various applications in mathematics and physics.
%The (co)homology and homotopy theories of Leibniz algebras were established in \cite{ammardefiLeibnizalgebra,Gao,Lodder,Lod2,Pir}.
Recently,
Leibniz algebras are studied in different aspects such as integration \cite{Bor,Cov,JP}, deformation quantization \cite{Dhe}, (co)homology and homotopy theories \cite{ammardefiLeibnizalgebra,Gao,Lodder,Lod2,Pir},  higher-order differential operators \cite{Akman} and higher gauge theories \cite{Kot, Str}. From the operadic viewpoint, the operad of Leibniz algebras is the duplicator of the operad of  Lie algebras \cite{Pei}.

\delete{
\begin{defi}\cite{Chap2005}
A \textbf{(skew-symmetric) quadratic Leibniz algebra} $(A,\circ_{A},\omega)$ is a Leibniz algebra $(A,\circ_{A})$ equipped with a nondegenerate skew-symmetric bilinear form $\omega$ satisfying the following \textbf{invariant} condition:
\begin{equation}
\omega(x,y\circ_{A}z)=\omega(x\circ_{A}z+z\circ_{A}x,y),\;\forall x,y,z\in A.
\end{equation}
\end{defi}}

In \cite{Chap2005}, Chapoton used the operad theory to properly determine the notion of an invariant bilinear form on a Leibniz algebra. Later the notion of a {\bf  (skew-symmetric) quadratic Leibniz algebra} was introduced in \cite{ST}, defined to be a Leibniz algebra $(A,\circ_{A})$ equipped with a nondegenerate skew-symmetric bilinear form $\omega$ which is invariant in the following sense:
\begin{equation}\label{eq:226}
\omega(x,y\circ_{A}z)=\omega(x\circ_{A}z+z\circ_{A}x,y),\;\forall x,y,z\in A.
\end{equation}
With such kind of bilinear forms applied to Manin triples of Leibniz algebras, the corresponding bialgebra theory for Leibniz algebras was established in \cite{ST, BLST}.

On the other hand, the Levi-Civita products of Leibniz algebras with nondegenerate skew-symmetric bilinear forms are determined as follows.
\begin{defi}\cite{TXS}\label{defi:232}
	Suppose that $(A,\circ_{A},\omega)$ is a \textbf{pseudo-Riemannian Leibniz algebra}, that is, $(A,\circ_{A})$ is a Leibniz algebra and $\omega$ is a nondegenerate skew-symmetric bilinear form on $A$.  The associated \textbf{Levi-Civita products} $\lozenge_{A},\blacklozenge_{A}:A\otimes A\rightarrow A$ are defined by the following equations:
\begin{eqnarray}
&&2\omega(x\lozenge_{A}y,z)=\omega(x\circ_{A}y,z)
+\omega(y\circ_{A}z,x)+\omega(z\circ_{A}y,x)+\omega(x\circ_{A}z,y),\label{eq:LV1}\\
&&2\omega(x\blacklozenge_{A}y,z)=\omega(x\circ_{A}y,z)
-\omega(y\circ_{A}z,x)-\omega(z\circ_{A}y,x)-\omega(x\circ_{A}z,y),\;\forall x,y,z\in A.\label{eq:LV2}
\end{eqnarray}
\end{defi}

Straightforwardly, we have $ \lozenge_{A}+\blacklozenge_{A}=\circ_{A} $  in Definition \ref{defi:232}, that is, $(A, \lozenge_{A} , \blacklozenge_{A} )$   is a Leibniz-admissible algebra. In particular,   if \eqref{eq:226} holds such that $(A,\circ_{A},\omega)$ is a quadratic Leibniz algebra, then the associated Levi-Civita products are $\lozenge_{A}=\blacklozenge_{A}=\frac{1}{2}\circ_{A}$. It is quite natural to consider the Levi-Civita products on Leibniz algebras with other types of nondegenerate skew-symmetric bilinear forms.

On the other hand, a bilinear form $\omega$ on a Leibniz algebra is called a {\bf $2$-cocycle} if the following equation holds:
\begin{equation}\label{eq:2-cocyle}
	\omega(z,x\circ_{A}y)=\omega(x,y\circ_{A}z
	+z\circ_{A}y)-\omega(y,x\circ_{A}z),\;\forall x,y,z\in A.
\end{equation}
When $\omega$ is moreover nondegenerate and symmetric, then $(A,\circ_{A},\omega)$ is called a {\bf symplectic Leibniz algebra}  \cite{TXS}
 which leads to a compatible pre-Leibniz algebra (originally named as a Leibniz dendriform algebra in \cite{ST}) on the underlying vector space.

In this paper, we study nondegenerate skew-symmetric $2$-cocycles on Leibniz algebras, which serve as the skew-symmetric version of the bilinear forms on symplectic Leibniz algebras.
Such bilinear form arises in the study of quadratic perm algebras
\cite{LZB} with certain linear operators.
As an analogue of the well-known fact that a commutative associative algebra with a derivation (or an averaging operator) gives rise to a Lie algebra, we show that a perm algebra \cite{Chap2001} with a derivation (or an averaging operator) gives rise to a Leibniz algebra.
In particular, if the perm algebra is quadratic, that is, the perm algebra admits a nondegenerate skew-symmetric invariant bilinear form (see \cite{Chap2005}), then such bilinear form is a $2$-cocycle on the corresponding Leibniz algebra.

%It is quite natural to consider the Levi-Civita products of Leibniz algebras with nondegenerate skew-symmetric $2$-cocycles.
Unlike the Levi-Civita products of quadratic Leibniz algebras which are still determined by the underlying Leibniz algebras,
the Levi-Civita products of Leibniz algebras with nondegenerate skew-symmetric $2$-cocycles turn to be a new kind of algebra structure, namely the anti-pre-Leibniz algebras.
An anti-pre-Leibniz algebra is characterized as a Leibniz-admissible algebra $(A,\succ_{A},\prec_{A})$ such that the negative left and right multiplication operators of the two operations give a representation $(-\mathcal{L}_{\succ_{A}},-\mathcal{R}_{\prec_{A}},A)$ of the sub-adjacent Leibniz algebra
 $(A,\circ_{A}=\succ_{A}+\prec_{A})$.
Such a characterization serves as the converse side of pre-Leibniz algebras  as the classical splitting theory of Leibniz algebras, since a pre-Leibniz algebra is a Leibniz-admissible algebra $(A,\rhd_{A},\lhd_{A})$ such that  $( \mathcal{L}_{\rhd_{A}}, \mathcal{R}_{\lhd_{A}},A)$ is a representation of the sub-adjacent Leibniz algebra
$(A,\circ_{A}=\rhd_{A}+\lhd_{A})$. Thus the name of anti-pre-Leibniz algebras is justified.
The property of this new algebra structure is further studied. In particular, we
introduced the notion of an anti-$\mathcal{O}$-operator on a Leibniz algebra, and show that
there is a  compatible anti-pre-Leibniz algebra structure on a Leibniz algebra if and only if there is an invertible anti-$\mathcal{O}$-operator of the Leibniz algebra.

The notion of Novikov dialgebras was introduced in \cite{Kol}, and it is shown therein that the operad of Novikov dialgebras is the duplicator of the operad of Novikov algebras \cite{Bal, Gel}
and every Novikov dialgebra can be embedded into a perm algebra with a derivation.
Moreover, it is shown in \cite{ZhHo} that Novikov dialgebras are in one-to-one correspondence with a class of Leibniz conformal algebras, and the bialgebra theory for Novikov dialgebras was studied in \cite{XuBHo}.

This paper gives a new characterization of Novikov dialgebras in terms of anti-pre-Leibniz algebras.
That is, a Novikov dialgebra is a transformed pre-Leibniz algebra which gives rise to an anti-pre-Leibniz algebra structure through specific combinations of multiplications.
More explicitly, there is a one-to-one correspondence between Novikov dialgebras and a subclass of anti-pre-Leibniz algebras, namely admissible Novikov dialgebras.
The properties of (admissible) Novikov dialgebras are further studied.
Note that an anti-pre-Leibniz algebra  $(A,\succ_{A},\prec_{A})$ renders two types of Leibniz algebra structures on the double space $A\oplus A^{*}$. We show that they are compatible Leibniz algebras if and only if $(A,\succ_{A},\prec_{A})$ is moreover an admissible Novikov dialgebra. There is a parallel result on Novikov dialgebras.
On the other hand, we introduce the notion of Gel'fand-Dorfman dialgebras, which include Novikov dialgebras as a subclass.
Then we establish the one-to-one correspondence between a class of infinite-dimensional Leibniz algebras and Gel'fand-Dorfman dialgebras through the process of affinization.

This paper is organized as follows.
In Section \ref{sec:2}, we introduce the
notion of anti-pre-Leibniz algebras, and show that there is a compatible anti-pre-Leibniz algebra of a Leibniz algebra if and only if there exists an invertible anti-$\mathcal{O}$-operator of the Leibniz algebra.
Anti-pre-Leibniz algebras serve as the Levi-Civita products of Leibniz algebras with nondegenerate skew-symmetric $2$-cocycles, and conversely, there is a natural construction of nondegenerate skew-symmetric $2$-cocycles on the semi-direct product Leibniz algebras induced from anti-pre-Leibniz algebras.
In Section \ref{sec:3}, we show that there is a one-to-one correspondence between Novikov dialgebras and a subclass of anti-pre-Leibniz algebras, namely admissible Novikov dialgebras.
The relations among (admissible) Novikov dialgebras, Gel'fand-Dorfman dialgebras and (compatible) Leibniz algebras are also further studied.

Throughout this paper, unless otherwise specified, all the vector
spaces and algebras are finite-dimensional over an algebraically
closed field $\mathbb {K}$ of characteristic zero, although many
results and notions remain valid in the infinite-dimensional case. For a vector space $A$ with a multiplication  $\circ_A:A\otimes A\rightarrow A$, the linear maps ${\mathcal L}_{\circ_A}, {\mathcal R}_{\circ_A}:A\rightarrow {\rm End} (A)$ are
respectively defined by
\begin{eqnarray*}
    {\mathcal L}_{\circ_A}(x)y:=x\circ_A y,\;\; {\mathcal
        R}_{\circ_A}(x)y:=y\circ_A x, \;\;\;\forall x, y\in A.
\end{eqnarray*}

\section{Anti-pre-Leibniz algebras and nondegenerate skew-symmetric 2-cocycles on Leibniz algebras}\label{sec:2}\
In this section,
we introduce the notion of an anti-pre-Leibniz algebra. We show that there is a compatible anti-pre-Leibniz algebra of a Leibniz algebra if and only if there exists an invertible anti-$\mathcal{O}$-operator of the Leibniz algebra.
Moreover,
anti-pre-Leibniz algebras serve as the Levi-Civita products of Leibniz algebras with nondegenerate skew-symmetric $2$-cocycles.
%\subsection{Anti-pre-Leibniz algebras}\label{sec2.1}\

Recall the representation theory of Leibniz algebras \cite{FMi}.
\begin{defi}
	A \textbf{representation} of a Leibniz algebra $(A,\circ_{A})$ is a triple $(l_{\circ_{A}},r_{\circ_{A}},V)$, such that $V$ is a vector space and $l_{\circ_{A}},r_{\circ_{A}}:A\rightarrow\mathrm{End}(V)$ are linear maps satisfying
	\begin{eqnarray}
		&&l_{\circ_{A}}(x\circ_{A}y)v=l_{\circ_{A}}(x)l_{\circ_{A}}(y)v
		-l_{\circ_{A}}(y)l_{\circ_{A}}(x)v,\label{eq:rep1}\\
		&&r_{\circ_{A}}(x\circ_{A}y)v=l_{\circ_{A}}(x)r_{\circ_{A}}(y)v
		-r_{\circ_{A}}(y)l_{\circ_{A}}(x)v,\label{eq:rep2}\\
		&&r_{\circ_{A}}(y)l_{\circ_{A}}(x)v
		=-r_{\circ_{A}}(y)r_{\circ_{A}}(x)v,\;\forall x,y\in A, v\in V.\label{eq:rep3}
	\end{eqnarray}
Two representations $(l_{\circ_{A}},r_{\circ_{A}},V)$ and $(l'_{\circ_{A}},r'_{\circ_{A}},V')$ of $(A,\circ_{A})$ are called \textbf{equivalent} if there exists a linear isomorphism $\phi:V\rightarrow V'$ such that the following equations hold:
\begin{equation}\label{eq:eq Leibniz rep}
	\phi l_{\circ_{A}}(x)=l'_{\circ_{A}}(x)\phi,\; \phi r_{\circ_{A}}(x)=r'_{\circ_{A}}(x)\phi,\;\forall x\in A.
\end{equation}
\end{defi}

\begin{pro}\label{pro:327}
Let $(A,\circ_{A})$ be a Leibniz algebra, $V$ be a vector space and $l_{\circ_{A}},r_{\circ_{A}}:A\rightarrow \mathrm{End}(V)$ be linear maps. Then the triple $(l_{\circ_{A}},r_{\circ_{A}},V)$ is a representation of the  $(A,\circ_{A})$ if and only if there is a  Leibniz algebra structure on the
direct sum $A\oplus V$ of vector spaces given by
\begin{equation}\label{eq:sd Leibniz} (x+u)\circ_{d}(y+v)=x\circ_{A}y+l_{\circ_{A}}(x)v+r_{\circ_{A}}(y)u,
\;\forall x,y\in A, u,v\in V.
\end{equation}
We denote the Leibniz algebra  $(A\oplus V,\circ_{d})$ by $A\ltimes_{l_{\circ_{A}},r_{\circ_{A}}}V$.
\end{pro}

\begin{ex}
Let $(A,\circ_{A})$ be a Leibniz algebra. Then $(\mathcal{L}_{\circ_{A}},\mathcal{R}_{\circ_{A}},A)$ is a representation of $(A,\circ_{A})$, which is called the \textbf{adjoint representation} of $(A,\circ_{A})$.
\end{ex}

%Let us recall the representation theory of Leibniz algebras.%in advance

\delete{
Let $A$ be a vector space with a multiplication $\circ_{A}:A\otimes A\rightarrow A$. Set linear maps $\mathcal{L}_{\circ_{A}},\mathcal{R}
_{\circ_{A}}:A\rightarrow \mathrm{End}(A)$ by
\begin{equation*}
\mathcal{L}_{\circ_{A}}(x)y=x\circ_{A}y=\mathcal{R}_{\circ_{A}}(y)x,\;\;\forall x,y\in A.
\end{equation*}}

Let $A$ and $V$ be vector spaces and $f:A\rightarrow\mathrm{End}(V)$ be a linear map.
We denote $f^{*}:A\rightarrow\mathrm{End}(V^{*})$ by
\begin{equation}
	\langle f^{*}(x)u^{*},v\rangle=-\langle u^{*},f(x)v\rangle,\;\forall x\in A, u^{*}\in V^{*},v\in V.
\end{equation}

\begin{lem}\label{ex:rep Leibniz}\cite{ST}
Let $(l_{\circ_{A}},r_{\circ_{A}},V)$ be a representation of a Leibniz algebra $(A,\circ_{A})$. Then  $(l^{*}_{\circ_{A}}, -l^{*}_{\circ_{A}}$
$-r^{*}_{\circ_{A}},V^{*})$ is also a representation of $(A,\circ_{A})$. In particular, $(\mathcal{L}^{*}_{\circ_{A}}, -\mathcal{L}^{*}_{\circ_{A}}-\mathcal{R}^{*}_{\circ_{A}}, A^{*})$ is a representation of $(A,\circ_{A})$.
\end{lem}

Now we introduce the notion of an anti-pre-Leibniz algebra.

\begin{defi}
Let $A$ be a vector space with multiplications $\succ_{A},\prec_{A}:A\otimes A\rightarrow A$.
If the following equations hold:
\begin{eqnarray}
\delete{
&&x\succ_{A}(y\succ_{A} z+y\prec_{A} z)+x\prec_{A}(y\succ_{A} z+y\prec_{A} z)=(x\succ_{A} y+x\prec_{A} y)\succ_{A}z\nonumber\\
&&\ \
+(x\succ_{A} y+x\prec_{A} y)\prec_{A}z+y\succ_{A}(x\succ_{A} z+x\prec_{A} z)+y\prec_{A}(x\succ_{A} z+x\prec_{A} z),\\}
&&(x\circ_{A}y)\prec_{A}z=x\succ_{A}
(y\circ_{A}z)-y\succ_{A}(x\circ_{A}z),\label{eq:anLei1}\\
&&(x\circ_{A}y)\succ_{A}z=y\succ_{A}(x\succ_{A} z)-x\succ_{A}(y\succ_{A} z),\label{eq:anLei2}\\
&&x\prec_{A}(y\circ_{A}z)=(y\succ_{A}x)\prec_{A}z-y\succ_{A}(x\prec_{A}z),\label{eq:anLei3}\\
&&(x\succ_{A}y)\prec_{A}z=-(y\prec_{A}x)\prec_{A}z,\;\forall x,y,z\in A,\label{eq:anLei4}
\end{eqnarray}
where $\circ_{A}:A\otimes A\rightarrow A$ is also a multiplication defined by
\begin{equation}\label{eq:split}
x\circ_{A}y=x\succ_{A}y+x\prec_{A}y,\;\forall x,y\in A,
\end{equation}
then we say $(A,\succ_{A},\prec_{A})$ is an \textbf{anti-pre-Leibniz algebra}.
\end{defi}

\begin{pro}\label{pro:defi}
Let $A$ be a vector space with multiplications $\succ_{A},\prec_{A}:A\otimes A\rightarrow A$. Define a multiplication $\circ_{A}$ by \eqref{eq:split}.
Then the following statements are equivalent:
\begin{enumerate}
  \item\label{it:1} $(A,\succ_{A},\prec_{A})$ is an anti-pre-Leibniz algebra.
  \item\label{it:2} For $(A,\succ_{A},\prec_{A})$, \eqref{eq:anLei2}-\eqref{eq:anLei4} and the following equation hold:
  \begin{equation}\label{eq:anLei equivalent}
  (x\circ_{A}y)\succ_{A}z=x\prec_{A}
(y\circ_{A}z)-y\prec_{A}(x\circ_{A}z),\;\forall x,y,z\in A.
  \end{equation}
  \item \label{it:3} $(A,\circ_{A})$ is a Leibniz algebra with a representation
      $(-\mathcal{L}_{\succ_{A}},-\mathcal{R}_{\prec_{A}},A)$.
  \item \label{it:4} $(A,\circ_{A})$ is a Leibniz algebra with a representation
      $(-\mathcal{L}^{*}_{\succ_{A}},\mathcal{L}^{*}_{\succ_{A}}+
      \mathcal{R}^{*}_{\prec_{A}},A^{*})$.
\end{enumerate}
\end{pro}
\begin{proof}
(\ref{it:1}) $\Longleftrightarrow$ (\ref{it:2}).
Let $x,y,z\in A$. Suppose that \eqref{eq:anLei2}-\eqref{eq:anLei4} hold.
Then we have
\begin{eqnarray*}
&&(x\circ_{A}y)\prec_{A}z-x\succ_{A}
(y\circ_{A}z)+y\succ_{A}(x\circ_{A}z)\\
&&\overset{\eqref{eq:split}}{=}
(x\succ_{A} y+x\prec_{A} y)\prec_{A}z-x\succ_{A}(y\succ_{A} z+y\prec_{A} z)+y\succ_{A}(x\succ_{A} z+x\prec_{A} z)\\
%&&\overset{\eqref{eq:anLei2},\eqref{eq:anLei3}}{=}(x\prec_{A} y)\prec_{A}z+y\succ_{A}(x\prec_{A} z)+(x\succ_{A} y+x\prec_{A} y)\succ_{A} z+y\prec_{A}(x\succ_{A} z+x\prec_{A} z)\\
&&\overset{\eqref{eq:anLei2}-\eqref{eq:anLei4}}{=}
(x\circ_{A}y)\succ_{A}z+y\prec_{A}(x\circ_{A}z)-x\prec_{A}
(y\circ_{A}z).
\end{eqnarray*}
Thus \eqref{eq:anLei1} holds if and only if \eqref{eq:anLei equivalent} holds.

(\ref{it:1}) $\Longleftrightarrow$ (\ref{it:3}).
Let $x,y,z\in A$. Suppose that $(-\mathcal{L}_{\succ_{A}},-\mathcal{R}_{\prec_{A}},A)$
is a representation of a Leibniz algebra $(A,\circ_{A})$.
Then the triple $(-\mathcal{L}_{\succ_{A}},$
$-\mathcal{R}_{\prec_{A}},A)$
satisfies \eqref{eq:rep1}-\eqref{eq:rep3}, that is,
\eqref{eq:anLei2}-\eqref{eq:anLei4} hold respectively.
Furthermore, since $(A,\circ_{A})$ is a Leibniz algebra, we have
\begin{align*}
&x\circ_{A}(y\circ_{A}z)-(x\circ_{A}y)\circ_{A}z-
y\circ_{A}(x\circ_{A}z)\\
&=x\succ_{A}(y\succ_{A} z+y\prec_{A} z)+x\prec_{A}(y\succ_{A} z+y\prec_{A} z)-(x\succ_{A} y+x\prec_{A} y)\succ_{A}z\nonumber\\
&\ \
-(x\succ_{A} y+x\prec_{A} y)\prec_{A}z-y\succ_{A}(x\succ_{A} z+x\prec_{A} z)-y\prec_{A}(x\succ_{A} z+x\prec_{A} z)\\
&\overset{\eqref{eq:anLei2}-\eqref{eq:anLei4}}{=}
(x\circ_{A}y)\prec_{A}z-x\succ_{A}
(y\circ_{A}z)+y\succ_{A}(x\circ_{A}z).
\end{align*}
\delete{
\begin{align*}
-\mathcal{L}_{\succ_{A}}(x\circ_{A}y)z=
-\mathcal{L}_{\succ_{A}}(x)(-\mathcal{L}_{\succ_{A}})(y)z-
(-\mathcal{L}_{\succ_{A}})(y)(-\mathcal{L}_{\succ_{A}})(x)z\ \ &\Longleftrightarrow \ \ (x\circ_{A}y)\succ_{A}z=y\succ_{A}(x\succ_{A} z)-x\succ_{A}(y\succ_{A} z), \\
-\mathcal{R}_{\prec_{A}}(x\circ_{A}y)z=
-\mathcal{L}_{\succ_{A}}(x)(-\mathcal{R}_{\prec_{A}})(y)z-
(-\mathcal{R}_{\prec_{A}})(y)(-\mathcal{L}_{\succ_{A}})(x)z\ \ &\Longleftrightarrow \ \ x\prec_{A}(y\circ_{A}z)
=(y\succ_{A}x)\prec_{A}z-y\succ_{A}(x\prec_{A}z), \\
-\mathcal{R}_{\prec_{A}}(y)(-\mathcal{L}_{\succ_{A}})(x)z=
-(-\mathcal{R}_{\prec_{A}})(y)(-\mathcal{R}_{\prec_{A}})(x)z\ \ &\Longleftrightarrow \ \ (x\succ_{A}y)\prec_{A}z=-(y\prec_{A}x)\prec_{A}z.
\end{align*}}
Hence $(A,\succ_{A},\prec_{A})$ is an anti-pre-Leibniz algebra.
Conversely, a similar argument shows that, if $(A,\succ_{A},\prec_{A})$ is an anti-pre-Leibniz algebra,
then $(A,\circ_{A})$ is a Leibniz algebra with
a representation
$(-\mathcal{L}_{\succ_{A}},-\mathcal{R}_{\prec_{A}},A)$.

(\ref{it:3}) $\Longleftrightarrow$ (\ref{it:4}).
It follows from Lemma \ref{ex:rep Leibniz}.
\end{proof}

\begin{rmk}
	\begin{enumerate}
		\item Let $(A,\succ_{A},\prec_{A})$ be an anti-pre-Leibniz algebra and
		$(A,\circ_{A})$ be the Leibniz algebra with $\circ_{A}$ given by \eqref{eq:split}. Sometimes in order to clarify the relationship between algebras, we would like to say $(A,\circ_{A})$ is the \textbf{sub-adjacent Leibniz algebra} of $(A,\succ_{A}$,
		$\prec_{A})$, and $(A,\succ_{A},\prec_{A})$ is a \textbf{compatible anti-pre-Leibniz algebra of $(A,\circ_{A})$}.
		\item The classical splitting of Leibniz algebras concerns with the notion of a pre-Leibniz algebra, characterized as  a vector space $A$ together with multiplications $\triangleright_{A},\triangleleft_{A}:A\otimes A\rightarrow A$ such that $(A,\circ_{A}=\triangleright_{A}+\triangleleft_{A})$ is a Leibniz algebra (that is, $(A,\triangleright_{A},\triangleleft_{A})$ is a Leibniz-admissible algebra) and $(\mathcal{L}_{\triangleright_{A}},\mathcal{R}_{\triangleleft_{A}},A)$ is a representation of $(A,\circ_{A})$.
Hence the name of anti-pre-Leibniz algebras is justified due to Proposition \ref{pro:defi} (\ref{it:3}).
	  \item An anti-pre-Leibniz algebra $(A,\succ_{A},\prec_{A})$ such that $x\succ_{A}y=-y\prec_{A}x$ for all $x,y\in A$ is exactly an anti-pre-Lie algebra $(A,\succ_{A})$ (see \cite{LB2022}).
		Hence anti-pre-Leibniz algebras are natural generalizations of anti-pre-Lie algebras.
		Moreover, the role that anti-pre-Leibniz algebras play in the splitting theory of Leibniz algebras is an analogue of the role that anti-pre-Lie algebras play in the splitting theory of Lie algebras, as well as of the role that anti-dendriform algebras play in the splitting theory of associative algebras (see \cite{GLB}).
	\end{enumerate}

\end{rmk}

Recall that
an $\mathcal{O}$-operator \cite{ST} of a Leibniz algebra $(A,\circ_{A})$
with respect to a representation $(l_{\circ_{A}}, r_{\circ_{A}},V)$ is a linear map $T:V\rightarrow A$ satisfying the following equation:
\begin{eqnarray*}
	(Tu)\circ_{A}(Tv)= T\big(l_{\circ_{A}}(Tu)v+r_{\circ_{A}}(Tv)u\big),\;\forall u,v\in V.
\end{eqnarray*}
Now we introduce the notion of anti-$\mathcal{O}$-operators.

\begin{defi}
Let $(l_{\circ_{A}},r_{\circ_{A}},V)$ be a representation of a Leibniz algebra
$(A,\circ_{A})$. If a linear map $T:V\rightarrow A$ satisfies
the following equation
\begin{eqnarray}\label{eq:10}
(Tu)\circ_{A}(Tv)=-T\big(l_{\circ_{A}}(Tu)v+r_{\circ_{A}}(Tv)u\big),\;\forall u,v\in V,
\end{eqnarray}
then we say $T$ is an {\bf anti-$\mathcal{O}$-operator of $(A,\circ_{A})$ associated to $(l_{\circ_{A}},r_{\circ_{A}},V)$}. In particular,
an anti-$\mathcal{O}$-operator $T$ is called {\bf strong} if $T$
satisfies
\begin{eqnarray}\label{eq:strong}
l_{\circ_{A}}\big((Tu)\circ_{A}(Tv)\big)w+r_{\circ_{A}}\big((Tu)\circ_{A}(Tw)\big)v
-r_{\circ_{A}}\big((Tv)\circ_{A}(Tw)\big)u=0,\;\forall u,v,w\in V.
\end{eqnarray}
\delete{
$(V,\circ_{V})$ given by
\begin{equation}\label{eq:strong}
u\circ_{V}v=-l_{\circ_{A}}(Tu)v-r_{\circ_{A}}(Tv)u,\;\forall u,v\in V
\end{equation}
is a Leibniz algebra.}
\end{defi}

\begin{pro}\label{pro:1}
Let $T:V\rightarrow A$ be an anti-$\mathcal{O}$-operator of a Leibniz algebra $(A,\circ_{A})$ associated to $(l_{\circ_{A}},r_{\circ_{A}},V)$. Define the multiplications $\succ_{V},\prec_{V}:V\otimes V\rightarrow V$ by
\begin{eqnarray}\label{eq:111}
u\succ_{V} v=-l_{\circ_{A}}(Tu)v,\; u\prec_{V} v=-r_{\circ_{A}}(Tv)u,\;\forall u,v\in V.
\end{eqnarray}
Then $(V,\succ_{V},\prec_{V})$ satisfies \eqref{eq:anLei2}-\eqref{eq:anLei4}. Moreover, $(V,\circ_{V})$ is a Leibniz algebra such that $(V,\succ_{V},\prec_{V})$
is an anti-pre-Leibniz algebra if and only if $T$ is strong. In this case, $T$ is a homomorphism of Leibniz algebras from the
sub-adjacent Leibniz algebra $(V,\circ_{V})$ to $(A,\circ_{A})$.
Furthermore, there is an induced anti-pre-Leibniz algebra structure
on $T(V)=\{T(u)~|~u\in V\}\subseteq A$ given by
\begin{align}\label{induce}
(Tu)\succ_{A} (Tv)=T(u\succ_{V} v),\quad
(Tu)\prec_{A} (Tv)=T(u\prec_{V} v),\quad \forall
u,v\in V,
\end{align}
and $T$ is a homomorphism of anti-pre-Leibniz algebras.
\end{pro}
\begin{proof}
Let $u,v,w\in V$. Then we have
\begin{eqnarray*}
&&(u\circ_{V} v)\succ_{V}w
%=(-l(Tu)v-r(Tv)u)\succ_{V}w
=-l_{\circ_{A}}T\big(-l_{\circ_{A}}(Tu)v-r_{\circ_{A}}(Tv)u\big)w
=-l_{\circ_{A}}\big((Tu)\circ_{A} (Tv)\big)w,\\
&&v\succ_{V}(u\succ_{V}w)=l_{\circ_{A}}(Tv)l_{\circ_{A}}(Tu)w,\\
&&-u\succ_{V}(v\succ_{V}w)=-l_{\circ_{A}}(Tu)l_{\circ_{A}}(Tv)w,\\
&&u\prec_{V}(v\circ_{V} w)=-r_{\circ_{A}}T\big(-l_{\circ_{A}}(Tv)w-r_{\circ_{A}}(Tw)v\big)u
=-r_{\circ_{A}}\big((Tv)\circ_{A} (Tw)\big)u,\\
&&(v\succ_{V}u)\prec_{V}w=r_{\circ_{A}}(Tw)l_{\circ_{A}}(Tv)u,\\
&&-v\succ_{V}(u\prec_{V}w)=-l_{\circ_{A}}(Tv)r_{\circ_{A}}(Tw)u,\\
&&(u\succ_{V}v)\prec_{V}w=r_{\circ_{A}}(Tw)l_{\circ_{A}}(Tu)v,\\
&&-(v\prec_{V}u)\prec_{V}w=-r_{\circ_{A}}(Tw)r_{\circ_{A}}(Tu)v.
\end{eqnarray*}
By \eqref{eq:rep1}-\eqref{eq:rep3}, we have
\begin{eqnarray*}
&&(u\circ_{V}v)\succ_{V}w
=v\succ_{V}(u\succ_{V}w)-u\succ_{V}(v\succ_{V}w),\\
&&u\prec_{V}(v\circ_{V} w)
=(v\succ_{V}u)\prec_{V}w-v\succ_{V}(u\prec_{V}w),\\
&&(u\succ_{V}v)\prec_{V}w=-(v\prec_{V}u)\prec_{V}w.
\end{eqnarray*}
Thus \eqref{eq:anLei2}-\eqref{eq:anLei4} hold on $(V,\succ_{V},\prec_{V})$. Moreover, $(V,\circ_{V})$ is a Leibniz algebra if and only if \eqref{eq:anLei1} holds on $(V,\succ_{V},\prec_{V})$, that is,
\begin{eqnarray*}
&&(u\circ_{A}v)\prec_{A}w-u\succ_{A}(v\circ_{A}w)
+v\succ_{A}(u\circ_{A}w)\\
&&=l_{\circ_{A}}\big((Tu)\circ_{A}(Tv)\big)w+r_{\circ_{A}}\big((Tu)\circ_{A}(Tw)\big)v
-r_{\circ_{A}}\big((Tv)\circ_{A}(Tw)\big)u=0.
\end{eqnarray*}
The other results follow from a straightforward checking.
\end{proof}

\begin{pro}\label{pro:strong}
An invertible anti-$\mathcal{O}$-operator of a Leibniz algebra is automatically strong.
\end{pro}
\begin{proof}
	Let $T:V\rightarrow A$ be an invertible anti-$\mathcal{O}$-operator of a Leibniz algebra $(A,\circ_{A})$ associated to $(l_{\circ_{A}},r_{\circ_{A}},V)$. Define the multiplications $\succ_{V},\prec_{V}:V\otimes V\rightarrow V$ by \eqref{eq:111}. Then we have
	%\begin{small}
	\begin{eqnarray*} T^{-1}(Tu\circ_{A}Tv)&=&-T^{-1}\Big(T\big(l_{\circ_{A}}(Tu)v+r_{\circ_{A}}(Tv)u\big)\Big)=-l_{\circ_{A}}(Tu)v-r_{\circ_{A}}(Tv)u\\ &=&u\succ_{V}v+u\prec_{V}v=u\circ_{V}v,
	\end{eqnarray*}
	%\end{small}
	for all $u,v\in V$. Therefore for all $u,v,w\in V$, we have
\begin{eqnarray*}
&&u\circ_{V}(v\circ_{V}w)=T^{-1}\big(Tu\circ_{A}T(v\circ_{V}w)\big)
=T^{-1}\big(Tu\circ_{A}(Tv\circ_{A}Tw)\big),\\
&&(u\circ_{V}v)\circ_{V}w=T^{-1}\big(T(u\circ_{V}v)\circ_{A}Tw\big)
=T^{-1}\big((Tu\circ_{A}Tv)\circ_{A}Tw\big),\\
&&v\circ_{V}(u\circ_{V}w)=T^{-1}\big(Tv\circ_{A}T(u\circ_{V}w)\big)
=T^{-1}\big(Tv\circ_{A}(Tu\circ_{A}Tw)\big).
\end{eqnarray*}
Thus $(V,\circ_{V})$ is a Leibniz algebra. Then by Proposition \ref{pro:1}, $T$ is strong.
\end{proof}

\begin{thm}\label{thm:2}
Let $(A,\circ_{A})$ be a Leibniz algebra. Then there is a compatible anti-pre-Leibniz algebra $(A,\succ_{A},\prec_{A})$ of $(A,\circ_{A})$ if and only if there is an invertible anti-$\mathcal{O}$-operator $T$ of $(A,\circ_{A})$ associated to $(l_{\circ_{A}},r_{\circ_{A}},V)$. In this case, the multiplications $\succ_{A},\prec_{A}$ are given by
\begin{equation}\label{eq:pro:2.14}
x\succ_{A}y=-T\big(l_{\circ_{A}}(x)T^{-1}(y)\big),
\;x\prec_{A}y=-T\big(r_{\circ_{A}}(y)T^{-1}(x)\big),\;\forall x,y\in A.
\end{equation}
\end{thm}
\begin{proof}
Let $T:V\rightarrow A$ be an invertible anti-$\mathcal{O}$-operator of $(A,\circ_{A})$ associated to $(l_{\circ_{A}},r_{\circ_{A}},V)$.
Then by Proposition \ref{pro:1} and Proposition \ref{pro:strong}, there is an induced anti-pre-Leibniz algebra structure $(V,\succ_{V},\prec_{V})$ on $V$ with $\succ_{A},\prec_{A}$ given by \eqref{eq:111}. The linear isomorphism $T$ gives an anti-pre-Leibniz algebra structure $(A,\succ_{T},\prec_{T})$ on $A$ by
\begin{eqnarray*}
&&x\succ_{T}y=T\big(T^{-1}(x)\succ_{V}T^{-1}(y)\big)
\overset{\eqref{eq:111}}{=}-T\big(l_{\circ_{A}}(x)T^{-1}(y)\big)=x\succ_{A}y,\\
&&x\prec_{T}y=T\big(T^{-1}(x)\prec_{V}T^{-1}(y)\big)
\overset{\eqref{eq:111}}{=}-T\big(r_{\circ_{A}}(y)T^{-1}(x)\big)=x\prec_{A}y,\;\forall x,y\in A.
\end{eqnarray*}
Moreover, obviously, we have $x\succ_{A}y+x\prec_{A}y=x\circ_{A}y$ for
all $x,y\in A$.
That is, $(A,\succ_{A},\prec_{A})$ is a compatible anti-pre-Leibniz algebra of $(A,\circ_{A})$.

Conversely, suppose that $(A,\succ_{A},\prec_{A})$ is an anti-pre-Leibniz algebra. Then $T=\mathrm{id}$ is an invertible anti-$\mathcal{O}$-operator of $(A,\circ_{A})$ associated to $(-\mathcal{L}_{\succ_{A}},-\mathcal{R}_{\prec_{A}},A)$.
\end{proof}

 Recall the notion of (quadratic) perm algebras.
\begin{defi}
\begin{enumerate}
\item
A perm algebra \cite{Chap2001}
is a vector space $A$ with a
multiplication $\star_{A}:A\otimes A\rightarrow A$ such that the following identity holds:
\begin{equation}\label{eq:perm}
x\star_{A}(y\star_{A}z)=(x\star_{A}y)\star_{A}z=(y\star _{A}x)\star_{A}z,\;\forall x,y,z\in A.
\end{equation}
\item
A quadratic perm algebra \cite{Chap2005} $(A,\star_{A},\omega)$ is a perm algebra $(A,\star_{A})$ with a nondegenerate skew-symmetric bilinear form $\omega$ satisfying the following \textbf{invariant} condition:
\begin{equation}
\omega(x\star_{A}y,z)=\omega(x,y\star_{A}z-z\star_{A}y),\;\forall x,y,z\in A.
\end{equation}
\end{enumerate}
\end{defi}

Let $A$ be a vector space with multiplications $\circ_{i_{A}}:A\otimes A\rightarrow A$, where $i\in\{1,\cdots,n\}$ and $P:A\rightarrow A$ be a linear map.
We say $P$ is an {\bf averaging operator} of the collection $(A,\circ_{1_{A}},\cdots,\circ_{n_{A}})$ if the following equation holds:
\begin{equation}\label{eq:Ao}
P(x)\circ_{i_{A}}P(y)=P\big(P(x)\circ_{i_{A}}y\big)=P\big(x\circ_{i_{A}} P(y)\big), \;\forall x,y\in A, 1\leq i\leq n.
\end{equation}
We say $P$ is a {\bf derivation} of  $(A,\circ_{1_{A}},\cdots,\circ_{n_{A}})$ if the following equation holds:
\begin{equation}\label{eq:der}
P(x \circ_{i_{A}} y)=P(x)\circ_{i_{A}}y+x\circ_{i_{A}} P(y),\;\forall x,y\in A, 1\leq i\leq n.
\end{equation}

\delete{
Let $\mathcal{P}$ be a binary operad with one binary operation and $(A,\circ_{A})$ be a $\mathcal{P}$-algebra. A linear map $P:A\rightarrow A$ is called an averaging operator or a derivation on $(A,\circ_{A})$ if
\begin{equation}\label{eq:Ao}
P(x)\circ_{A}P(y)=P\big(P(x)\circ_{A}y\big)=P\big(x\circ_{A} P(y)\big),\;\text{or}\;P(x\circ_{A}y)=P(x)\circ_{A}y+x\circ_{A} P(y),\;\forall x,y\in A.
\end{equation}
\delete{
A linear map $P:A\rightarrow A$ is called an averaging operator or a derivation on a perm algebra $(A,\star_{A})$ if
\begin{equation}\label{eq:Ao}
P(x)\star_{A}P(y)=P\big(P(x)\star_{A}y\big)=P\big(x\star_{A} P(y)\big),\;\text{or}\;P(x\star_{A}y)=P(x)\star_{A}y+x\star_{A} P(y),\;\forall x,y\in A.
\end{equation}}}

\begin{pro}\label{pro:perm and Leibniz}
Let $(A,\star_{A},\omega)$ be a quadratic perm algebra. Suppose that $P:A\rightarrow A$ is an averaging operator or a derivation on $(A,\star_{A})$. Then there is a  Leibniz algebra $(A,\circ_{A})$ given by
\begin{eqnarray}
x\circ_{A}y=P(x)\star_{A}y-x\star_{A}P(y),\;\forall x,y\in A,
\end{eqnarray}
and $\omega$ is a nondegenerate skew-symmetric $2$-cocycle on $(A,\circ_{A})$.
\end{pro}

\begin{proof}
Let $P:A\rightarrow A$ be an averaging operator on $(A,\star_{A})$. Then for all $x,y,z\in A$, we have
\begin{align*}
x\circ_{A}(y\circ_{A}z)
&=P(x)\star_{A}\big(P(y)\star_{A}z-y\star_{A}P(z)\big)
-x\star_{A}P\big(P(y)\star_{A}z-y\star_{A}P(z)\big)\\
&\overset{\eqref{eq:Ao}}{=}P(x)\star_{A}\big(P(y)\star_{A}z\big)-P(x)\star_{A}\big(y\star_{A}P(z)\big),\\
(x\circ_{A}y)\circ_{A}z
&=P\big(P(x)\star_{A}y-x\star_{A}P(y)\big)\star_{A}z
-\big(P(x)\star_{A}y-x\star_{A}P(y)\big)\star_{A}P(z)\\
&\overset{\eqref{eq:Ao}}{=}-\big(P(x)\star_{A}y\big)\star_{A}P(z)
+\big(x\star_{A}P(y)\big)\star_{A}P(z),\\
y\circ_{A}(x\circ_{A}z)
%&=P(y)\star_{A}\big(P(x)\star_{A}z-x\star_{A}P(z)\big)
%-y\star_{A}P\big(P(x)\star_{A}z-x\star_{A}P(z)\big)\\
&\overset{\eqref{eq:Ao}}{=}P(y)\star_{A}\big(P(x)\star_{A}z\big)-P(y)\star_{A}\big(x\star_{A}P(z)\big).
\end{align*}
Thus $(A,\circ_{A})$ is a Leibniz algebra. Similarly, if $P:A\rightarrow A$ is a derivation on $(A,\star_{A})$, then  $(A,\circ_{A})$ is also a Leibniz algebra. Moreover, we have
\begin{align*}
\omega(z,x\circ_{A}y)&=-\omega\big(P(x)\star_{A}y-x\star_{A}P(y),z\big)\\
&=\omega\big(x,P(y)\star_{A}z-z\star_{A}P(y)\big)
-\omega\big(P(x),y\star_{A}z-z\star_{A}y\big),\\
-\omega(y,x\circ_{A}z)
%&=\omega\big(P(x)\star_{A}z-x\star_{A}P(z),y\big)\\
&=\omega\big(P(x),z\star_{A}y-y\star_{A}z\big)
-\omega\big(x,P(z)\star_{A}y-y\star_{A}P(z)\big),\\
\omega(x,y\circ_{A}z+z\circ_{A}y)
&=\omega\big(x,P(y)\star_{A}z-z\star_{A}P(y)\big)+\omega\big(x,P(z)\star_{A}y-y\star_{A}P(z)\big).
\end{align*}
Then \eqref{eq:2-cocyle} holds and thus the conclusion follows.
\end{proof}

\begin{lem}
	The Levi-Civita products $\succ_{A},\prec_{A}:A\otimes A\rightarrow A$ of a Leibniz algebra $(A,\circ_{A})$ with a nondegenerate skew-symmetric $2$-cocycle $\omega$ are given by
\begin{eqnarray}
	\omega(x\succ_{A}y,z)&=&\omega(y,x\circ_{A}z),\label{eq:cor3}\\
	\omega(x\prec_{A}y,z)&=&-\omega(x,y\circ_{A}z+z\circ_{A}y),\;\forall x,y,z\in A.\label{eq:cor4}
\end{eqnarray}
\end{lem}
\begin{proof}
	By Definition \ref{defi:232}, the Levi-Civita products $\succ_{A},\prec_{A}:A\otimes A\rightarrow A$  are given by
	\begin{align*}
		2\omega(x\succ_{A}y,z)&=\omega(x\circ_{A}y,z)
		-\omega(y\circ_{A}z,x)-\omega(z\circ_{A}y,x)-\omega(x\circ_{A}z,y)\\
		&\overset{\eqref{eq:2-cocyle}}{=}\omega(y,x\circ_{A}z)-\omega(x,y\circ_{A}z+z\circ_{A}y)
		-\omega(y\circ_{A}z,x)-\omega(z\circ_{A}y,x)-\omega(x\circ_{A}z,y)\\
		&=2\omega(y,x\circ_{A}z),\\
		2\omega(x\prec_{A}y,z)&=\omega(x\circ_{A}y,z)
		+\omega(y\circ_{A}z,x)+\omega(z\circ_{A}y,x)+\omega(x\circ_{A}z,y)\\
		&\overset{\eqref{eq:2-cocyle}}{=}\omega(y,x\circ_{A}z)-\omega(x,y\circ_{A}z+z\circ_{A}y)
		+\omega(y\circ_{A}z,x)+\omega(z\circ_{A}y,x)+\omega(x\circ_{A}z,y)\\
		&=-2\omega(x,y\circ_{A}z+z\circ_{A}y),
	\end{align*}
for all $x,y,z\in A$. Hence the conclusion follows.
\end{proof}

Let $V$ be a vector space. Then the isomorphism ${\rm
Hom}_{\mathbb K}(V\otimes V,\mathbb K)\cong {\rm Hom}_{\mathbb
K}(V, V^*)$ identifies a bilinear form  $\omega:V\otimes
V\rightarrow \mathbb K$ on V as a linear map $\omega
^\natural:V\rightarrow V^*$ by
$$\omega(u,v)=\langle \omega^\natural (u),
v\rangle,\;\;\forall u,v\in V.$$ Moreover, $\omega$ is
nondegenerate if and only if $\omega^\natural$ is invertible.

\begin{thm}\label{thm:420}
Let $\omega$ be a nondegenerate skew-symmetric $2$-cocycle on a Leibniz algebra $(A,\circ_{A})$.
Then $(A,\succ_{A},\prec_{A})$ with multiplications $\succ_{A},\prec_{A}:A\otimes A\rightarrow A$ given by \eqref{eq:cor3} and  \eqref{eq:cor4} is an anti-pre-Leibniz algebra.
Moreover, $(-\mathcal{L}_{\succ_{A}},-\mathcal{R}_{\prec_{A}},$
$A)$ and $(\mathcal{L}^{*}_{\circ_{A}},-\mathcal{L}^{*}_{\circ_{A}}-\mathcal{R}^{*}_{\circ_{A}},A^{*})$ are equivalent as representations
of $(A,\circ_{A})$, or equivalently,    $(\mathcal{L}_{\circ_{A}},\mathcal{R}_{\circ_{A}},A)$ and  $(-\mathcal{L}^{*}_{\succ_{A}},$
$\mathcal{L}^{*}_{\succ_{A}}+\mathcal{R}^{*}_{\prec_{A}},A^{*})$ are equivalent as representations of $(A,\circ_{A})$.
\end{thm}
\begin{proof}
	Let $x,y,z\in A$ and $a^{*}=\omega^{\natural}(x), b^{*}=\omega^{\natural}(y)$. Then we have
\begin{align*}
&\langle \omega^{\natural} \big(\omega^{\natural^{-1}}(a^{*})\circ_{A} \omega^{\natural^{-1}}(b^{*})\big),z\rangle=\langle \omega^{\natural}(x\circ_{A}y),z\rangle=\omega(x\circ_{A}y,z)\\	&=\omega(y,x\circ_{A}z)-\omega(x,y\circ_{A}z+z\circ_{A}y)
%=\langle \omega^{\natural}(y),x\circ_{A}z\rangle-\langle \omega^{\natural}(x),y\circ_{A}z+z\circ_{A}y\rangle
=\langle(\mathcal{L}^{*}_{\circ_{A}}+\mathcal{R}^{*}_{\circ_{A}})
\big(\omega^{\natural^{-1}}(b^{*})\big)a^{*},z\rangle-\langle \mathcal{L}^{*}_{\circ_{A}}\big(\omega^{\natural^{-1}}(a^{*})\big)b^{*},z\rangle.
\end{align*}
Hence we have
\begin{align*}
\omega^{\natural^{-1}}(a^{*})\circ_{A} \omega^{\natural^{-1}}(b^{*})=\omega^{\natural^{-1}} \Big((\mathcal{L}^{*}_{\circ_{A}}+\mathcal{R}^{*}_{\circ_{A}})
\big(\omega^{\natural^{-1}}(b^{*})\big)a^{*},z\rangle-\langle \mathcal{L}^{*}_{\circ_{A}}\big(\omega^{\natural^{-1}}(a^{*})\big)b^{*}\Big).
\end{align*}
That is, $\omega^{\natural^{-1}}$ is an invertible
anti-$\mathcal{O}$-operator of $(A,\circ_{A})$ associated to the triple $(\mathcal{L}^{*}_{\circ_{A}},-\mathcal{L}^{*}_{\circ_{A}}-\mathcal{R}^{*}_{\circ_{A}},A^{*})$.
	Hence by Theorem \ref{thm:2}, there are multiplications $\succ_{A},\prec_{A}$ given by \eqref{eq:pro:2.14} for
\begin{equation*}
T=\omega^{\natural^{-1}},\; (l_{\circ_{A}},r_{\circ_{A}},V)=(\mathcal{L}^{*}_{A},-\mathcal{L}^{*}_{\circ_{A}}-\mathcal{R}^{*}_{\circ_{A}},A^{*})
\end{equation*}
such that $(A,\succ_{A},\prec_{A})$ is an anti-pre-Leibniz algebra.
	Explicitly, we have
\begin{eqnarray*} x\succ_{A}y=-\omega^{\natural^{-1}}\big(\mathcal{L}^{*}_{\circ_{A}}(x)\mathcal{B}^{\natural}(y) \big),\;\; x\prec_{A}y=\omega^{\natural^{-1}}\big((\mathcal{L}^{*}_{\circ_{A}}+\mathcal{R}^{*}_{\circ_{A}}) (y)\mathcal{B}^{\natural}(x) \big),
\end{eqnarray*}
and in other words
\begin{align*}
&\omega(x\succ_{A}y,z)=-\langle \mathcal{L}^{*}_{\circ_{A}}(x)\omega^{\natural}(y), z\rangle=\langle \omega^{\natural}(y),x\circ_{A}z\rangle=\omega(y,x\circ_{A}z),\\	&\omega(x\prec_{A}y,z)=\langle(\mathcal{L}^{*}_{\circ_{A}}+\mathcal{R}^{*}_{\circ_{A}}) (y)\omega^{\natural}(x) ,z\rangle=-\langle \omega^{\natural}(x),y\circ_{A}z+z\circ_{A}y\rangle
=-\omega(x,y\circ_{A}z+z\circ_{A}y).
\end{align*}
Hence the first part of the theorem holds. To prove the second half part,
 we have
\begin{align*}
&\langle \omega^{\natural}\big(-\mathcal{L}_{\succ_{A}}(x)y\big),z\rangle=-\omega(x\succ_{A}y ,z)=-\omega(y,x\circ_{A}z)=-\langle \omega^{\natural}(y),x\circ_{A}z\rangle=\langle \mathcal{L}^{*}_{\circ_{A}}(x)\omega^{\natural}(y),z\rangle,\\
&\langle \omega^{\natural}\big(-\mathcal{R}_{\prec_{A}}(y)x\big),z\rangle
=-\omega(x\prec_{A}y,z)=\omega(x,y\circ_{A}z+z\circ_{A}y)=
\langle (-\mathcal{L}^{*}_{\circ_{A}}-\mathcal{R}^{*}_{\circ_{A}})(x)\omega^{\natural}(y),z\rangle.
\end{align*}
Thus we have
\begin{align*}
\omega^{\natural}\big(-\mathcal{L}_{\succ_{A}}(x)y\big)=\mathcal{L}^{*}_{\circ_{A}}(x)\omega^{\natural}(y),\;
\omega^{\natural}\big(-\mathcal{R}_{\prec_{A}}(y)x\big)=(-\mathcal{L}^{*}_{\circ_{A}}-\mathcal{R}^{*}_{\circ_{A}})(x)\omega^{\natural}(y).
\end{align*}
Therefore, the bijection $\omega^{\natural}:A\rightarrow A^{*}$ gives the equivalence between  $(-\mathcal{L}_{\succ_{A}},-\mathcal{R}_{\prec_{A}},A)$ and $(\mathcal{L}^{*}_{\circ_{A}},$
$-\mathcal{L}^{*}_{\circ_{A}}-\mathcal{R}^{*}_{\circ_{A}},A^{*})$
as representations of $(A,\circ_{A})$. Again by \eqref{eq:cor3} and \eqref{eq:cor4}, we have
\begin{eqnarray}\label{eq:bf3}
\omega(x\succ_{A}y+y\prec_{A}x,z)=-\omega(y,z\circ_{A}x).%\;\forall x,y,z\in A.
\end{eqnarray}
Similarly by \eqref{eq:cor3} and \eqref{eq:bf3}, we have
\begin{equation}\label{eq:bf qua}
(\omega^{\natural})^{*}\big(\mathcal{L}_{\circ_{A}}(x)z\big)=-\mathcal{L}^{*}_{\succ_{A}}(x)(\omega^{\natural})^{*}(z),\;
(\omega^{\natural})^{*}\big(\mathcal{R}_{\circ_{A}}(x)z\big)=(\mathcal{L}^{*}_{\succ_{A}}+\mathcal{R}^{*}_{\prec_{A}})(x)(\omega^{\natural})^{*}(z).
\end{equation}
Therefore, the bijection $(\omega^{\natural})^{*}:A\rightarrow A^{*}$ gives the equivalence between $(\mathcal{L}_{\circ_{A}}, \mathcal{R}_{\circ_{A}},A)$ and $(-\mathcal{L}^{*}_{\succ_{A}},$
$\mathcal{L}^{*}_{\succ_{A}}+\mathcal{R}^{*}_{\prec_{A}},A^{*})$ as representations of $(A,\circ_{A})$.
\end{proof}

\delete{
The converse of Theorem \ref{thm:420} cannot hold, and thus we reach a weaker conclusion as follows.

\begin{cor}
Let $(A,\succ_{A},\prec_{A})$ be an anti-pre-Leibniz algebra and $(A,\circ_{A})$ be the sub-adjacent Leibniz algebra. If $(\mathcal{L} _{\circ_{A}},\mathcal{R} _{\circ_{A}},A)$ and $(-\mathcal{L}^{*}_{\succ_{A}},\mathcal{L}^{*}_{\succ_{A}}+\mathcal{R}^{*}_{\prec_{A}},A^{*})$
are equivalent as representations of $(A,\circ_{A})$, then there exists a nondegenerate bilinear form $\omega$ on $(A,\circ_{A})$ satisfying %\eqref{eq:cor3} and \eqref{eq:cor4}.
\begin{equation}\label{eq:various}
\omega(x\circ_{A}y,z)=\omega(y,x\circ_{A}z)-\omega(x,y\circ_{A}z
+z\circ_{A}y),\;\forall x,y,z\in A.
\end{equation}
\end{cor}

\begin{proof}
Let $\phi:A\rightarrow A^{*}$ be a bijection which gives the equivalence between $(\mathcal{L} _{\circ_{A}},\mathcal{R} _{\circ_{A}},A)$ and $(-\mathcal{L}^{*}_{\succ_{A}},\mathcal{L}^{*}_{\succ_{A}}+\mathcal{R}^{*}_{\prec_{A}},A^{*})$ as representations of $(A,\circ_{A})$.
Then there is a nondegenerate bilinear form $\omega$ on $A$ given by
\begin{equation}\label{eq:phi2}
\omega(x,y)=\langle x,\phi(y)\rangle,\;\forall x,y\in A,
\end{equation}
that is, we have $\phi=(\omega^{\natural})^{*}$, and hence \eqref{eq:bf qua} holds. Moreover, \eqref{eq:bf qua} equivalently gives rise to \eqref{eq:cor3} and \eqref{eq:bf3}, which render \eqref{eq:cor4}. Again by \eqref{eq:cor3} and \eqref{eq:cor4}, we obtain \eqref{eq:various}. This completes the proof.
\end{proof}
}

\begin{pro}\label{pro:antipreLeibniz}
Let $(A,\succ_{A},\prec_{A})$ be an anti-pre-Leibniz algebra and $(A,\circ_{A})$ be the sub-adjacent Leibniz algebra.
Then there is a Leibniz algebra structure on $A\oplus A^{*}$ given by
\begin{equation}\label{eq:left inv2}	(x+a^{*})\circ_{d}(y+b^{*})=x\circ_{A}y-\mathcal{L}^{*}_{\succ_{A}}(x)b^{*}
+(\mathcal{L}^{*}_{\succ_{A}}
+\mathcal{R}^{*}_{\prec_{A}})(y)a^{*},\;\forall x,y\in A,a^{*},b^{*}\in A^{*}.
\end{equation}
Moreover, the natural nondegenerate skew-symmetric bilinear form $\omega_{p}$ on $A\oplus A^{*}$ given by
\begin{equation}\label{eq:antibf}
	\omega_p(x+a^{*},y+b^{*})=\langle x,b^{*}\rangle-\langle
	a^{*},y\rangle,\;\;\forall x,y\in A, a^{*},b^{*}\in A^{*}
\end{equation}
is a nondegenerate skew-symmetric $2$-cocycle on
$A\ltimes_{-\mathcal{L}^{*}_{\succ_{A}},\mathcal{L}^{*}_{\succ_{A}}
+\mathcal{R}^{*}_{\prec_{A}}}A^*$.

Conversely, let $(l,r,A^*)$ be a representation of a Leibniz algebra $(A,\circ_{A})$. Suppose that the
bilinear form $\omega_p$ given by \eqref{eq:antibf} is a nondegenerate skew-symmetric $2$-cocycle on
$A\ltimes_{l,r}A^*$. Then there is a compatible anti-pre-Leibniz
algebra $(A,\succ_{A},\prec_{A})$ of
$(A,\circ_{A})$ such that $l=-\mathcal{L}^{*}_{\succ_{A}},
r=\mathcal{L}^{*}_{\succ_{A}}+\mathcal{R}^{*}_{\prec_{A}}$.
\end{pro}

\begin{proof}
By Proposition \ref{pro:defi}, $(-\mathcal{L}^{*}_{\succ_{A}},\mathcal{L}^{*}_{\succ_{A}}+\mathcal{R}^{*}_{\prec_{A}},A^{*})$ is a representation of $(A,\circ_{A})$. Then by Proposition \ref{pro:327},
$A\ltimes_{-\mathcal{L}^{*}_{\succ_{A}},
\mathcal{L}^{*}_{\succ_{A}}+\mathcal{R}^{*}_{\prec_{A}}} A^{*}$ defined by \eqref{eq:left inv2} is a Leibniz algebra. For all $x,y,z\in A,a^{*},b^{*},c^{*}\in A^{*}$, we have
\begin{align*}
\omega_{p}\big(z+c^{*},(x+a^{*})\circ_{d}(y+b^{*})\big)&=
\omega_{p}\big(z+c^{*},x\circ_{A}y-\mathcal{L}^{*}_{\succ_{A}}(x)b^{*}
+(\mathcal{L}^{*}_{\succ_{A}}+\mathcal{R}^{*}_{\prec_{A}})(y)a^{*}\big)\\
&=\langle b^{*},x\succ_{A}z\rangle
-\langle a^{*},y\succ_{A}z+z\prec_{A}y\rangle-\langle c^{*},x\circ_{A}y\rangle.
\end{align*}Similarly, we have
\begin{eqnarray*}
&&-\omega_{p}\big(y+b^{*},(x+a^{*})\circ_{d}(z+c^{*})\big)\\
&=&\langle b^{*},x\circ_{A}z\rangle
-\langle c^{*},x\succ_{A}y\rangle
+\langle a^{*},z\succ_{A}y+y\prec_{A}z\rangle,\\
&&\omega_{p}\big(x+a^{*},(y+b^{*})\circ_{d}(z+c^{*})
+(z+c^{*})\circ_{d}(y+b^{*})\big)\\
&=&-\langle c^{*},x\prec_{A}y\rangle-\langle b^{*},x\prec_{A}z\rangle-\langle a^{*},y\circ_{A}z+z\circ_{A}y\rangle.
\end{eqnarray*}
Thus $\omega_{p}$ is a nondegenerate skew-symmetric $2$-cocycle on $A\ltimes_{-\mathcal{L}^{*}_{\succ_{A}},\mathcal{L}^{*}_{\succ_{A}}+\mathcal{R}^{*}_{\prec_{A}}} A^{*}$.
\delete{
It is straightforward to show that $\omega_p$ is a
nondegenerate skew-symmetric $2$-cocycle on
$A\ltimes_{-\mathcal{L}^{*}_{\succ_{A}},
\mathcal{L}^{*}_{\succ_{A}}+\mathcal{R}^{*}_{\prec_{A}}}A^*$.}

Conversely, by Theorem \ref{thm:420}, there is a
compatible anti-pre-Leibniz algebra $(A\oplus A^*,\succ_{d},$
$\prec_{d})$ with $\succ_{d},\prec_{d}$ given by \eqref{eq:cor3} and \eqref{eq:cor4} respectively.
In particular, we have
\begin{equation*}
\omega(x\succ_{d}y,z)=\omega(y,x\circ_{A}z)=0,\;
\omega(x\prec_{d}y,z)=-\omega(x,y\circ_{A}z+z\circ_{A}y)=0,\;\forall x,y,z\in A.
\end{equation*}
Thus $x\succ_{d}y,x\prec_{d}y\in A$ for all
$x,y\in A$ and hence $(A,\succ_{A}=\succ_{d}|_{A},\prec_{A}=\prec_{d}|_{A})$ is
an anti-pre-Leibniz algebra. Furthermore, for all $x,y\in A, a^*\in A^*,$ we have
\begin{align*}
&\langle -\mathcal{L}^{*}_{\succ_{A}}(x)a^*, y\rangle
=\langle a^*, x\succ_{A}y\rangle
=-\omega_{p}(a^{*},x\succ_{A}y)
=\omega_p(y,x\circ_{d}a^{*})
=\omega_p\big(y,l(x)a^*\big)=\langle l(x) a^*, y\rangle,\\
&\langle (\mathcal{L}^{*}_{\succ_{A}}+\mathcal{R}^{*}_{\prec_{A}})(x)a^*, y\rangle
=-\langle a^*, x\succ_{A}y+y\prec_{A}x\rangle
=\omega_{p}( a^*, x\succ_{A}y+y\prec_{A}x )\\
&
=-\omega_p( y,a^{*}\circ_{d}x)
=-\omega_p\big(y, r(x) a^*\big)
=\langle r(x) a^*,y\rangle.
\end{align*}
So $l=-\mathcal{L}^{*}_{\succ_{A}},
r=\mathcal{L}^{*}_{\succ_{A}}+\mathcal{R}^{*}_{\prec_{A}}$.
Hence the conclusion follows.
\end{proof}

\delete{
\begin{defi}\cite{TXS}
A \textbf{pseudo-Riemannian Leibniz algebra} is a Leibniz algebra $(A,\circ_{A})$ endowed with a nondegenerate skew-symmetric bilinear form $\omega$. The associated \textbf{Levi-Civita products} $\lozenge_{A},\blacklozenge_{A}:A\otimes A\rightarrow A$ are defined by the following equations:
\begin{eqnarray}
&&2\omega(x\lozenge_{A}y,z)=\omega(x\circ_{A}y,z)
+\omega(y\circ_{A}z,x)+\omega(z\circ_{A}y,x)+\omega(x\circ_{A}z,y),\label{eq:LV1}\\
&&2\omega(x\blacklozenge_{A}y,z)=\omega(x\circ_{A}y,z)
-\omega(y\circ_{A}z,x)-\omega(z\circ_{A}y,x)-\omega(x\circ_{A}z,y),\;\forall x,y,z\in A.\label{eq:LV2}
\end{eqnarray}
\end{defi}}

\begin{defi}\label{defi:anLei rep}
Let $(A,\succ_{A},\prec_{A})$ be an anti-pre-Leibniz algebra and $(A,\circ_{A})$ be the sub-adjacent Leibniz algebra. Let $V$ be a vector space and
$l_{\succ_{A}},r_{\succ_{A}},l_{\prec_{A}},r_{\prec_{A}}:A\rightarrow\mathrm{End}_{\mathbb
K}(V)$ be linear maps. Set
\begin{equation}\label{eq:sum linear}
l_{\circ_{A}}=l_{\succ_{A}}+l_{\prec_{A}},\;\;
r_{\circ_{A}}=r_{\succ_{A}}+r_{\prec_{A}}.
\end{equation}
%$(l_{\circ_{A}},r_{\circ_{A}},V)$ is a representation of $(A,\circ_{A})$ and
If the following equations hold:
\begin{eqnarray}
&&r_{\prec_{A}}(x)r_{\circ_{A}}(y)v=r_{\succ_{A}}(y\circ_{A}x)v
-l_{\succ_{A}}(y)r_{\circ_{A}}(x)v,\label{eq:anLei rep1}\\
&&r_{\succ_{A}}(x)r_{\circ_{A}}(y)v=l_{\succ_{A}}(y)r_{\succ_{A}}(x)v
-r_{\succ_{A}}(y\succ_{A}x)v,\label{eq:anLei rep2}\\
&&r_{\prec_{A}}(x\circ_{A}y)v=r_{\prec_{A}}(y)l_{\succ_{A}}(x)v
-l_{\succ_{A}}(x)r_{\prec_{A}}(y)v,\label{eq:anLei rep3}\\
&&r_{\prec_{A}}(x)r_{\succ_{A}}(y)v=-r_{\prec_{A}}(x)l_{\prec_{A}}(y)v,
\label{eq:anLei rep4}\\
&&r_{\succ_{A}}(x)l_{\circ_{A}}(y)v=r_{\succ_{A}}(y\succ_{A}x)v
-l_{\succ_{A}}(y)r_{\succ_{A}}(x)v,\label{eq:anLei rep5}\\
&&l_{\prec_{A}}(x)r_{\circ_{A}}(y)v=r_{\prec_{A}}(y)r_{\succ_{A}}(x)v
-r_{\succ_{A}}(x\prec_{A}y)v,\label{eq:anLei rep6}\\
&&r_{\prec_{A}}(x)l_{\succ_{A}}(y)v=-r_{\prec_{A}}(x)r_{\prec_{A}}(y)v,
\label{eq:anLei rep7}\\
&&l_{\prec_{A}}(x\circ_{A}y)v=l_{\succ_{A}}(x)l_{\circ_{A}}(y)v
-l_{\succ_{A}}(y)l_{\circ_{A}}(x)v,\label{eq:anLei rep8}\\
&&l_{\succ_{A}}(x\circ_{A}y)v=l_{\succ_{A}}(y)l_{\succ_{A}}(x)v
-l_{\succ_{A}}(x)l_{\succ_{A}}(y)v,\label{eq:anLei rep9}\\
&&l_{\prec_{A}}(x)l_{\circ_{A}}(y)v=l_{\prec_{A}}(y\succ_{A}x)v
-l_{\succ_{A}}(y)l_{\prec_{A}}(x)v,\label{eq:anLei rep10}\\
&&l_{\prec_{A}}(x\succ_{A}y)v=-l_{\prec_{A}}(y\prec_{A}x)v
,\;\forall x,y\in A, v\in V,\label{eq:anLei rep11}
\end{eqnarray}
then we say $(l_{\succ_{A}},r_{\succ_{A}},l_{\prec_{A}},r_{\prec_{A}},V)$ is a {\bf representation} of $(A,\succ_{A},\prec_{A})$.
Two representations $(l_{\succ_{A}},r_{\succ_{A}},$
$l_{\prec_{A}},r_{\prec_{A}},V)$ and $(l'_{\succ_{A}},r'_{\succ_{A}},l'_{\prec_{A}},r'_{\prec_{A}},V')$ of $(A,\succ_{A},\prec_{A})$ are called
\textbf{equivalent} if there exists a linear isomorphism $\phi:V \rightarrow V'$ such that the following equations hold:
\begin{eqnarray*}
\phi l_{\succ_{A}}(x)=l'_{\succ_{A}}(x)\phi,\;\phi r_{\succ_{A}}(x)=r'_{\succ_{A}}(x)\phi,\;\phi l_{\prec_{A}} (x)=l'_{\prec_{A}}(x)\phi,\;\phi r_{\prec_{A}} (x)=r'_{\prec_{A}}(x)\phi,\;\forall x\in A.
\end{eqnarray*}
\end{defi}

 \begin{pro}
 Let $(A,\succ_{A},\prec_{A})$ be an anti-pre-Leibniz algebra, $V$ be a vector space and
 $l_{\succ_{A}},r_{\succ_{A}}$,
 $l_{\prec_{A}},r_{\prec_{A}}:A\rightarrow\mathrm{End}_{\mathbb
 	K}(V)$ be linear maps.
 Then $(l_{\succ_{A}},r_{\succ_{A}},$
 $l_{\prec_{A}},r_{\prec_{A}},V)$ is a representation of $(A,\succ_{A},\prec_{A})$ if and only if there is an anti-pre-Leibniz algebra structure on $A\oplus V$ given by
 \begin{eqnarray*}
 &&	(x+u)\succ_{d}(y+v)=x\succ_{A}y+l_{\succ_{A}}(x)v+r _{\succ_{A}}(y)u,\\
 &&(x+u)\prec_{d}(y+v)=x\prec_{A}y+l_{\prec_{A}}(x)v+r _{\prec_{A}}(y)u,\;\forall x,y\in A, u,v\in V.
 \end{eqnarray*}
Moreover, $(l_{\circ_{A}}=l_{\succ_{A}}+l_{\prec_{A}},
r_{\circ_{A}}=r_{\succ_{A}}+r_{\prec_{A}},V)$ is a representation of the sub-adjacent Leibniz algebra $(A,\circ_{A})$ of $(A,\succ_{A},\prec_{A})$.
 \end{pro}
 \begin{proof}
 	It follows from a direct computation.
 \end{proof}

\begin{ex}
Let $(A,\succ_{A},\prec_{A})$ be an anti-pre-Leibniz algebra. Then
$(\mathcal{L}_{\succ_{A}},\mathcal{R}_{\succ_{A}},
\mathcal{L}_{\prec_{A}},\mathcal{R}_{\prec_{A}},A)$
is a representation of $(A,\succ_{A},\prec_{A})$, which is called
the \textbf{adjoint representation}.
\end{ex}

\begin{pro}\label{pro:anLei dual rep}
Let $(A,\succ_{A},\prec_{A})$ be an anti-pre-Leibniz algebra and $(A,\circ_{A})$ be the sub-adjacent Leibniz algebra. If
$(l_{\succ_{A}},r_{\succ_{A}},l_{\prec_{A}},r_{\prec_{A}},V)$ is a representation of $(A,\succ_{A},\prec_{A})$, then
$ (-l^{*}_{\circ_{A}},-l^{*}_{\prec_{A}}-r^{*}_{\succ_{A}}$,
$
l^{*}_{\prec_{A}},l^{*}_{\circ_{A}}+r^{*}_{\circ_{A}},V^{*})$
is also a representation of $(A,\succ_{A},\prec_{A})$.
In particular, $(-\mathcal{L}^{*}_{\circ_{A}},-\mathcal{L}^{*}_{\prec_{A}}
-\mathcal{R}^{*}_{\succ_{A}},\mathcal{L}^{*}_{\prec_{A}},
\mathcal{L}^{*}_{\circ_{A}}+\mathcal{R}^{*}_{\circ_{A}},
A^{*})$ is a representation of $(A,\succ_{A},\prec_{A})$.%, which is called the {\bf coadjoint representation} of $(A,\succ_{A}$,$\prec_{A})$.
\end{pro}
\begin{proof}
Let $x,y\in A,u^{*}\in V^{*},v\in V$. Then we have
\begin{eqnarray*}
&&\langle\big((l^{*}_{\circ_{A}}+r^{*}_{\circ_{A}})(x)
(l^{*}_{\succ_{A}}+r^{*}_{\prec_{A}})(y)+
(l^{*}_{\prec_{A}}+r^{*}_{\succ_{A}})(y\circ_{A}x)
-l^{*}_{\circ_{A}}(y)(l^{*}_{\succ_{A}}+r^{*}_{\prec_{A}})(x)\big) u^{*},v\rangle\\
&=&\langle u^{*},\big((l_{\succ_{A}}+r_{\prec_{A}})(y)
(l_{\circ_{A}}+r_{\circ_{A}})(x)
-(l_{\prec_{A}}+r_{\succ_{A}})(y\circ_{A}x)
-(l_{\succ_{A}}+r_{\prec_{A}})(x)l_{\circ_{A}}(y)\big)v\rangle\\
&\overset{\eqref{eq:anLei rep1},\eqref{eq:anLei rep8}}{=}&
\langle u^{*},\big(l_{\succ_{A}}(y)l_{\circ_{A}}(x)+
l_{\succ_{A}}(y)r_{\circ_{A}}(x)+r_{\prec_{A}}(y)l_{\circ_{A}}(x)
+r_{\prec_{A}}(y)r_{\circ_{A}}(x)
-l_{\succ_{A}}(y)l_{\circ_{A}}(x)\\
&&+l_{\succ_{A}}(x)l_{\circ_{A}}(y)
-r_{\prec_{A}}(x)r_{\circ_{A}}(y)
-l_{\succ_{A}}(y)r_{\circ_{A}}(x)
-l_{\succ_{A}}(x)l_{\circ_{A}}(y)-r_{\prec_{A}}(x)l_{\circ_{A}}(y)\big)v\rangle\\
&=&\langle u^{*},\big(r_{\prec_{A}}(y)l_{\circ_{A}}(x)
+r_{\prec_{A}}(y)r_{\circ_{A}}(x)-r_{\prec_{A}}(x)r_{\circ_{A}}(y)
-r_{\prec_{A}}(x)l_{\circ_{A}}(y)\big)v\rangle\\
&
\overset{\eqref{eq:anLei rep4},\eqref{eq:anLei rep7}}{=}&0.
\end{eqnarray*}
Thus \eqref{eq:anLei rep1} holds for  $(-l^{*}_{\circ_{A}},-l^{*}_{\prec_{A}}-r^{*}_{\succ_{A}},
l^{*}_{\prec_{A}},l^{*}_{\circ_{A}}+r^{*}_{\circ_{A}},V^{*})$. Similarly, \eqref{eq:anLei rep2}-\eqref{eq:anLei rep11} hold for $(-l^{*}_{\circ_{A}},-l^{*}_{\prec_{A}}$
$-r^{*}_{\succ_{A}},
l^{*}_{\prec_{A}},l^{*}_{\circ_{A}}+r^{*}_{\circ_{A}},V^{*})$.
Thus $(-l^{*}_{\circ_{A}},-l^{*}_{\prec_{A}}-r^{*}_{\succ_{A}},
l^{*}_{\prec_{A}},l^{*}_{\circ_{A}}+r^{*}_{\circ_{A}},V^{*})$
is a representation of $(A,\succ_{A}$,
$\prec_{A})$.
\end{proof}

\begin{defi}
A bilinear form $\omega$ on an anti-pre-Leibniz algebra  $(A,\succ_{A},\prec_{A})$ is called
{\bf invariant} if \eqref{eq:cor3} and \eqref{eq:cor4} hold. A {\bf quadratic anti-pre-Leibniz algebra}
$(A,\succ_{A},\prec_{A},\omega)$ is an anti-pre-Leibniz algebra $(A,\succ_{A},\prec_{A})$ together with a nondegenerate skew-symmetric invariant bilinear form $\omega$.
\end{defi}

Next we establish a one-to-one correspondence between Leibniz algebras with nondegenerate skew-symmetric $2$-cocycles and quadratic anti-pre-Leibniz algebras.

\begin{pro}
Let $(A,\succ_{A},\prec_{A},\omega)$ be a quadratic anti-pre-Leibniz algebra. Then $\omega$ is a nondegenerate skew-symmetric $2$-cocycle on the sub-adjacent Leibniz algebra $(A,\circ_{A})$. Conversely, let $\omega$ be a nondegenerate skew-symmetric $2$-cocycle on a Leibniz algebra $(A,\circ_{A})$. Then $\omega$ is invariant on the compatible anti-pre-Leibniz algebra $(A,\succ_{A},\prec_{A})$ defined by \eqref{eq:cor3} and \eqref{eq:cor4}.
\end{pro}
\begin{proof}
	The first half part follows from a direct checking.
 The second half part follows from Theorem \ref{thm:420} immediately.
\end{proof}

\begin{pro}
Let $(A,\succ_{A},\prec_{A},\omega)$ be a quadratic anti-pre-Leibniz algebra. Then $(\mathcal{L}_{\succ_{A}},\mathcal{R}_{\succ_{A}},$
$\mathcal{L}_{\prec_{A}},\mathcal{R}_{\prec_{A}},A)$
and
$(-\mathcal{L}^{*}_{\circ_{A}},-\mathcal{L}^{*}_{\prec_{A}}
-\mathcal{R}^{*}_{\succ_{A}},\mathcal{L}^{*}_{\prec_{A}},
\mathcal{L}^{*}_{\circ_{A}}+\mathcal{R}^{*}_{\circ_{A}},A^{*})$
are equivalent as representations of $(A,\succ_{A},\prec_{A})$.
\end{pro}
\begin{proof}
Let $x,y,z\in A$. Observing the RHS of \eqref{eq:cor4} is symmetric in $y$ and $z$, we have
\begin{equation}\label{eq:cor3.27}
\omega(x\prec_{A}y,z)=\omega(x\prec_{A}z,y)
=-\omega(y,x\prec_{A}z).
\end{equation}
Then we obtain
\begin{equation}\label{eq:cor3.37}
\omega(x\succ_{A}y,z)\overset{\eqref{eq:bf3}}{=}
-\omega(y,z\circ_{A}x)-\omega(y\prec_{A}x,z)
\overset{\eqref{eq:cor3},\eqref{eq:cor3.27}}{=}
\omega(x,z\succ_{A}y)+\omega(x,y\prec_{A}z).
\end{equation}
Thus by \eqref{eq:cor3}, \eqref{eq:cor3.37}, \eqref{eq:cor3.27} and \eqref{eq:cor4}, we respectively have
\begin{eqnarray*}
&&\omega^{\natural}\big(\mathcal{L}_{\succ_{A}}(x)y\big)
=-\mathcal{L}^{*}_{\circ_{A}}(x)\omega^{\natural}(y),\;
\omega^{\natural}\big(\mathcal{R}_{\succ_{A}}(y)x\big)
=(-\mathcal{L}^{*}_{\prec_{A}}
-\mathcal{R}^{*}_{\succ_{A}})(y)\omega^{\natural}(x),\\
&&\omega^{\natural}\big(\mathcal{L}_{\prec_{A}}(x)y\big)
=\mathcal{L}^{*}_{\prec_{A}}(x)\omega^{\natural}(y),\;
\omega^{\natural}\big(\mathcal{R}_{\prec_{A}}(y)x\big)
=(\mathcal{L}^{*}_{\circ_{A}}+\mathcal{R}^{*}_{\circ_{A}})(y)\omega^{\natural}(x)
.%\;\forall x,y\in A.
\end{eqnarray*}
Hence, the bijection $\omega^{\natural}:A\rightarrow A^*$ gives the equivalence between $(\mathcal{L}_{\succ_{A}},\mathcal{R}_{\succ_{A}},
\mathcal{L}_{\prec_{A}},\mathcal{R}_{\prec_{A}},A)$
and
$(-\mathcal{L}^{*}_{\circ_{A}},-\mathcal{L}^{*}_{\prec_{A}}
-\mathcal{R}^{*}_{\succ_{A}},\mathcal{L}^{*}_{\prec_{A}},
\mathcal{L}^{*}_{\circ_{A}}+\mathcal{R}^{*}_{\circ_{A}},A^{*})$
as representations of $(A,\succ_{A},\prec_{A})$.
\end{proof}

\section{Admissible  Novikov dialgebras as a subclass of anti-pre-Leibniz algebras}\label{sec:3}\

We introduce the notion of admissible Novikov dialgebras as a subclass of anti-pre-Leibniz algebras and give the one-to-one correspondence between them and Novikov dialgebras as a subclass of transformed pre-Leibniz algebras.
We also study the relationship between (admissible) Novikov dialgebras and other types of algebras, such as compatible Leibniz algebras and Gel'fand-Dorfman dialgebras.

%\subsection{Correspondence between some subclasses of pre-Leibniz algebras and anti-pre-Leibniz algebras}\label{sec:3.1}\

\begin{defi}\label{defi:Nov dialgebra}\cite{Kol}
A \textbf{Novikov dialgebra} is a triple $(A, \vdash_{A} ,\dashv_{A})$, where $A$ is a vector space and $ \vdash_{A} ,\dashv_{A}:A\otimes A\rightarrow A$ are multiplications such that the following equations hold:
\begin{eqnarray}
&&x\vdash_{A}(y \vdash_{A} z )=( x\vdash_{A}y )\vdash_{A}z-(y\dashv_{A}x)\vdash_{A}z+y\vdash_{A}(x\vdash_{A}z),\label{eq:nd1}\\
&&(y\vdash_{A}z)\dashv_{A}x=y\vdash_{A}(z\dashv_{A}x)-z \dashv_{A}(y\vdash_{A}x)
+(z\dashv_{A}y)\dashv_{A}x,\label{eq:nd2}\\
&&z\dashv_{A}(x\vdash_{A}y)=z\dashv_{A}(x\dashv_{A}y),\label{eq:nd3}\\
&&(z\dashv_{A}y)\dashv_{A}x=(z\dashv_{A}x)\dashv_{A}y,\label{eq:nd4}\\
&&(y\dashv_{A}x)\vdash_{A}z=(y\vdash_{A}z)\dashv_{A}x=(y\vdash_{A}x)\vdash_{A}z,\;\forall x,y,z\in A.\label{eq:nd5}
\end{eqnarray}
\end{defi}

Recall that a \textbf{pre-Leibniz algebra} \cite{ST} is a triple $(A,\triangleright_{A},\triangleleft_{A})$, where $A$ is a vector space, and $\triangleright_{A},\triangleleft_{A}:A\otimes A\rightarrow A$ are multiplications such that the following equations hold:
\begin{eqnarray}
	&&x\triangleright_{A}(y\triangleright_{A} z)=(x\triangleright_{A}y)\triangleright_{A} z+(x\triangleleft_{A}y)\triangleright_{A} z+y\triangleright_{A}(x\triangleright_{A} z),\label{eq:pre-L1}\\
	&&x\triangleleft_{A}(y\triangleright_{A}z)=y\triangleright_{A}(x\triangleleft_{A}z)-(y\triangleright_{A}x)\triangleleft_{A} z-x\triangleleft_{A}(y\triangleleft_{A} z),\label{eq:pre-L2}\\
	&&(x\triangleright_{A} y)\triangleleft_{A} z=-(y\triangleleft_{A}x)\triangleleft_{A} z,\;\;\forall x,y,z\in A.\label{eq:pre-L3}
\end{eqnarray}
Then we have the following result.

\begin{pro}\label{pro:transformed}
	Let $A$ be a vector space, $\vdash_{A},\dashv_{A}:A\otimes A\rightarrow A$ be multiplications and
	\begin{eqnarray}\label{eq:transformed}
		x\triangleright_{A}y=x\vdash_{A}y,\;
		x\triangleleft_{A}y=-y\dashv_{A}x,\;\forall x,y\in A.
	\end{eqnarray}
	Then
	$(A,\triangleright_{A},\triangleleft_{A})$ is a pre-Leibniz algebra if and only if
	\eqref{eq:nd1}-\eqref{eq:nd3} hold for $(A,\vdash_{A},\dashv_{A})$.
	In this case, we say $(A,\vdash_{A},\dashv_{A})$ is a {\bf transformed pre-Leibniz algebra}.
\end{pro}

\delete{
\begin{pro}
A Novikov dialgebra is a pre-Leibniz algebra.
\end{pro}
\begin{proof}
It is clear that the difference between \eqref{eq:LN4} and \eqref{eq:LN5} gives \eqref{eq:pre-L3}. Thus \eqref{eq:LN3}-\eqref{eq:LN5} hold if and only if \eqref{eq:pre-L1}-\eqref{eq:pre-L3} hold. Hence the conclusion follows.
\end{proof}}

\begin{defi}
Let $A$ be a vector space with multiplications $\succ_{A},\prec_{A}:A\otimes A\rightarrow A$. The triple $(A,\succ_{A},\prec_{A})$ is called
an \textbf{admissible Novikov dialgebra} if \eqref{eq:anLei2}-\eqref{eq:anLei4} and the following equations hold for all $x,y,z\in A$:
\begin{eqnarray}
(x\succ_{A}y)\succ_{A}z&=&-(y\prec_{A}x)\succ_{A}z,\label{eq:adm Nov dia1}\\
x\prec_{A}(y\prec_{A}z)-y\prec_{A}(x\prec_{A}z)
&=&2(x\prec_{A}y)\prec_{A}z-2(y\prec_{A}x)\prec_{A}z,\label{eq:adm Nov dia2}\\
(x\succ_{A}y)\succ_{A}z+y\prec_{A}(x\succ_{A}z)
&=&2x\succ_{A}(y\circ_{A}z).\label{eq:adm Nov dia3}
\end{eqnarray}
\end{defi}

\begin{pro}
An admissible Novikov dialgebra is an anti-pre-Leibniz algebra.
\end{pro}
\begin{proof}
Let $(A,\succ_{A},\prec_{A})$ is an admissible Novikov dialgebra.
For all $x,y,z\in A$, we have
\begin{eqnarray*}
&&2x\prec_{A}(y\circ_{A}z)-2y\prec_{A}(x\circ_{A}z)-2(x\circ_{A}y)\succ_{A}z\\
&\overset{\eqref{eq:anLei2}-\eqref{eq:anLei4}}{=}&
2(y\prec_{A}x)\prec_{A}z+2x\succ_{A}(y\circ_{A}z)
-2(x\prec_{A}y)\prec_{A}z-2y\succ_{A}(x\circ_{A}z)\\
&\overset{\eqref{eq:adm Nov dia2},\eqref{eq:adm Nov dia3}}{=}&
(x\succ_{A}y)\succ_{A}z+y\prec_{A}(x\circ_{A}z)
-(y\succ_{A}x)\succ_{A}z-x\prec_{A}(y\circ_{A}z)\\
&\overset{\eqref{eq:adm Nov dia1}}{=}&
y\prec_{A}(x\circ_{A}z)-x\prec_{A}(y\circ_{A}z)+(x\circ_{A}y)\succ_{A}z.
\end{eqnarray*}
Thus
\begin{equation*}
3(x\circ_{A}y)\succ_{A}z-3x\prec_{A}(y\circ_{A}z)
+3y\prec_{A}(x\circ_{A}z)=0,
\end{equation*}
that is, \eqref{eq:anLei equivalent} holds. By Proposition
\ref{pro:defi} (\ref{it:2}), $(A,\succ_{A},\prec_{A})$ is an anti-pre-Leibniz algebra.
\end{proof}

\delete{
\begin{defi}\cite{GLB}
Let $A$ be a vector space with multiplications $\triangleright_{A},\triangleleft_{A}:A\otimes A\rightarrow A$. Define multiplications $\succ_{A},\prec_{A}:A\otimes A\rightarrow A$ respectively by
\begin{eqnarray}
x\succ_{A}y=x\triangleright_{A}y+qx\triangleleft_{A}y,
\;x\prec_{A}y=x\triangleleft_{A}y+qx\triangleright_{A}y,\;\forall x,y\in A,
\end{eqnarray}
for some $q\in \mathbb {K}$. Then the triple $(A,\succ_{A},\prec_{A})$ is called the \textbf{$q$-algebra} of $(A,\triangleright_{A},\triangleleft_{A})$.
\end{defi}}

\begin{pro}\label{pro:-2-alg}
Let $(A,\vdash_{A},\dashv_{A})$ be a transformed pre-Leibniz algebra, and
\begin{eqnarray}\label{eq:-2algebra}
x\succ_{A}y=x\vdash_{A}y+2y \dashv_{A}x,\;
x\prec_{A}y=-y\dashv_{A}x-2x\vdash_{A}y,\;\forall x,y\in A.
\end{eqnarray}
Then $(A,\succ_{A},\prec_{A})$ is an anti-pre-Leibniz algebra if and only if $(A,\vdash_{A},\dashv_{A})$ is further a Novikov dialgebra. Moreover, in this case, $(A,\succ_{A},\prec_{A})$ is an admissible Novikov dialgebra.
\end{pro}
%\textcolor{blue}{GL:Please revise the proof accordingly}

\begin{proof}
By Proposition \ref{pro:transformed}, $(A,\triangleright_{A},\triangleleft_{A})$ with $\triangleright_{A},\triangleleft_{A}$ given by \eqref{eq:transformed} is a pre-Leibniz algebra.
Let the sub-adjacent Leibniz algebra of $(A,\triangleright_{A},\triangleleft_{A})$ be $(A,\bullet_{A})$. Then by \eqref{eq:transformed} and \eqref{eq:-2algebra}, we have
\begin{equation*}
x\circ_{A}y=x\succ_{A}y+x\prec_{A}y
=-x\vdash_{A}y+y\dashv_{A}x
=-x\triangleright_{A}y-x\triangleleft_{A}y=-x\bullet_{A}y, \;\forall x,y\in A.
\end{equation*}
Thus $(A,\circ_{A})$ is a Leibniz algebra.
For all $x,y,z\in A$, we have
{\small
\begin{eqnarray*}
&&(x\succ_{A} y+x\prec_{A} y)\succ_{A}z-y\succ_{A}(x\succ_{A} z)+x\succ_{A}(y\succ_{A} z)\\
&\overset{\eqref{eq:-2algebra}}{=}&
-(x\vdash_{A}y)\vdash_{A}z+(y\dashv_{A}x)\vdash_{A}z
-2z\dashv_{A}(x\vdash_{A}y)+2z\dashv_{A}(y\dashv_{A}x)
-y\vdash_{A}(x\vdash_{A}z)\\
&&
-2y\vdash_{A}(z\dashv_{A}x)-2(x\vdash_{A}z)\dashv_{A}y
-4(z\dashv_{A}x)\dashv_{A}y
+x\vdash_{A}(y\vdash_{A}z)+2x\vdash_{A}(z\dashv_{A}y)\\
&&
+2(y\vdash_{A}z)\dashv_{A}x+4(z\dashv_{A}y)\dashv_{A}x\\
&\overset{\eqref{eq:nd1}}{=}&
-2z\dashv_{A}(x\vdash_{A}y)+2z\dashv_{A}(y\dashv_{A}x)
-2y\vdash_{A}(z\dashv_{A}x)-2(x\vdash_{A}z)\dashv_{A}y
-4(z\dashv_{A}x)\dashv_{A}y\\
&&
+2x\vdash_{A}(z\dashv_{A}y)+2(y\vdash_{A}z)\dashv_{A}x
+4(z\dashv_{A}y)\dashv_{A}x\\
&\overset{\eqref{eq:nd2},\eqref{eq:nd3}}{=}&
6(z\dashv_{A}y)\dashv_{A}x-6(z\dashv_{A}x)\dashv_{A}y.
\end{eqnarray*}}Thus \eqref{eq:anLei2} holds if and only if \eqref{eq:nd4} holds.
Similarly, we have \eqref{eq:anLei3} and \eqref{eq:anLei4} hold if and only if \eqref{eq:nd5} holds.
\delete{
\begin{eqnarray*}
&&x\prec_{A}(y\succ_{A} z+y\prec_{A} z)-(y\succ_{A}x)\prec_{A}z+y\succ_{A}(x\prec_{A}z)\\
&&\overset{\eqref{eq:-2algebra}}{=}-x\triangleleft_{A}(y\triangleright_{A}z)-x\triangleleft_{A}(y\triangleleft_{A}z)
+2x\triangleright_{A}(y\triangleright_{A}z)+2x\triangleright_{A}(y\triangleleft_{A}z)-(y\triangleright_{A}x)\triangleleft_{A}z\\
&&\ \ \
+2(y\triangleleft_{A}x)\triangleleft_{A}z+2(y\triangleright_{A}x)\triangleright_{A}z-4(y\triangleleft_{A}x)\triangleright_{A}z
+y\triangleright_{A}(x\triangleleft_{A}z)-2y\triangleright_{A}(x\triangleright_{A}z)\\
&&\ \ \
-2y\triangleleft_{A}(x\triangleleft_{A}z)+4y\triangleleft_{A}(x\triangleright_{A}z)\\
&&\overset{\eqref{eq:pre-L2}}{=}2x\triangleright_{A}(y\triangleright_{A}z)+2(y\triangleright_{A}x)\triangleright_{A}z
-4(y\triangleleft_{A}x)\triangleright_{A}z-2y\triangleright_{A}(x\triangleright_{A}z)+4y\triangleleft_{A}(x\triangleright_{A}z)\\
&&\overset{\eqref{eq:pre-L1}}{=}6y\triangleleft_{A}(x\triangleright_{A}z)-6(y\triangleleft_{A}x)\triangleright_{A}z,\\
&&(x\succ_{A}y)\prec_{A}z+(y\prec_{A}x)\prec_{A}z\\
&&\overset{\eqref{eq:-2algebra}}{=}(x\triangleright_{A}y)\triangleleft_{A}z-2(x\triangleleft_{A}y)\triangleleft_{A}z
-2(x\triangleright_{A}y)\triangleright_{A}z+4(x\triangleleft_{A}y)\triangleright_{A}z+(y\triangleleft_{A}x)\triangleleft_{A}z\\
&&\ \ \
-2(y\triangleright_{A}x)\triangleleft_{A}z-2(y\triangleleft_{A}x)\triangleright_{A}z+4(y\triangleright_{A}x)\triangleright_{A}z\\
&&\overset{\eqref{eq:pre-L1}}{=}(x\triangleright_{A}y)\triangleleft_{A}z-2(x\triangleleft_{A}y)\triangleleft_{A}z
+4(x\triangleleft_{A}y)\triangleright_{A}z+(y\triangleleft_{A}x)\triangleleft_{A}z\\
&&\ \ \
-2(y\triangleright_{A}x)\triangleleft_{A}z+4(y\triangleright_{A}x)\triangleright_{A}z\\
&&\overset{\eqref{eq:pre-L3}}{=}6(x\triangleleft_{A}y)\triangleright_{A}z+6(y\triangleright_{A}x)\triangleright_{A}z.
\end{eqnarray*}}Then $(A,\succ_{A},\prec_{A})$ is an anti-pre-Leibniz algebra if and only if $(A,\vdash_{A},\dashv_{A})$ is a Novikov dialgebra. Moreover, if $(A,\vdash_{A},\dashv_{A})$ is a Novikov dialgebra, then we have
\begin{eqnarray*}
&&(x\succ_{A}y)\succ_{A}z+(y\prec_{A}x)\succ_{A}z\\
&\overset{\eqref{eq:-2algebra}}{=}&
(x\vdash_{A}y)\vdash_{A}z+2z\dashv_{A}(x\vdash_{A}y)
+2(y\dashv_{A}x)\vdash_{A}z+4z\dashv_{A}(y\dashv_{A}x)\\
&&-(x\dashv_{A}y)\vdash_{A}z-2z\dashv_{A}(x\dashv_{A}y)
-2(y\vdash_{A}x)\vdash_{A}z-4z\dashv_{A}(y\vdash_{A}x)\\
&\overset{\eqref{eq:nd3},\eqref{eq:nd5}}{=}&0.
\end{eqnarray*}
Thus \eqref{eq:adm Nov dia1} holds. Similarly, \eqref{eq:adm Nov dia2} and \eqref{eq:adm Nov dia3} hold. Therefore, $(A,\succ_{A},\prec_{A})$ is an admissible Novikov dialgebra.
\delete{
\begin{eqnarray*}
&&x\prec_{A}(y\prec_{A}z)-y\prec_{A}(x\prec_{A}z)
-2(x\prec_{A}y)\prec_{A}z+2(y\prec_{A}x)\prec_{A}z\\
&\overset{\eqref{eq:-2algebra}}{=}&x\triangleleft_{A}(y\triangleleft_{A}z)
-2x\triangleright_{A}(y\triangleleft_{A}z)-2x\triangleleft_{A}(y\triangleright_{A}z)
+4x\triangleright_{A}(y\triangleright_{A}z)-y\triangleleft_{A}(x\triangleleft_{A}z)\\
&&+2y\triangleright_{A}(x\triangleleft_{A}z)+2y\triangleleft_{A}(x\triangleright_{A}z)
-4y\triangleright_{A}(x\triangleright_{A}z)-2(x\triangleleft_{A}y)\triangleleft_{A}z
+4(x\triangleleft_{A}y)\triangleright_{A}z\\
&&+4(x\triangleright_{A}y)\triangleleft_{A}z
-8(x\triangleright_{A}y)\triangleright_{A}z+2(y\triangleleft_{A}x)\triangleleft_{A}z
-4(y\triangleleft_{A}x)\triangleright_{A}z-4(y\triangleright_{A}x)\triangleleft_{A}z\\
&&+8(y\triangleright_{A}x)\triangleright_{A}z\\
&\overset{\eqref{eq:pre-L3}-\eqref{eq:LN2}}{=}&
4x\triangleright_{A}(y\triangleright_{A}z)-2x\triangleright_{A}(y\triangleleft_{A}z)
+2y\triangleright_{A}(x\triangleleft_{A}z)-4y\triangleright_{A}(x\triangleright_{A}z)
-2(y\triangleright_{A}x)\triangleleft_{A}z\\
&&+2(x\triangleright_{A}y)\triangleleft_{A}z
-6(x\triangleright_{A}y)\triangleright_{A}z-6(x\triangleleft_{A}y)\triangleright_{A}z\\
&\overset{\eqref{eq:pre-L1}}{=}&2y\triangleright_{A}(x\triangleleft_{A}z)
-2x\triangleright_{A}(y\triangleleft_{A}z)-2(y\triangleright_{A}x)\triangleleft_{A}z
+2(x\triangleright_{A}y)\triangleleft_{A}z-2(x\triangleright_{A}y)\triangleright_{A}z\\
&&-2(x\triangleleft_{A}y)\triangleright_{A}z\\
&\overset{\eqref{eq:pre-L2},\eqref{eq:LN2}}{=}&
2x\triangleleft_{A}(y\triangleleft_{A}z)-2x\triangleright_{A}(y\triangleleft_{A}z)
+2(x\triangleright_{A}y)\triangleleft_{A}z-2(x\triangleright_{A}y)\triangleright_{A}z\\
&\overset{\eqref{eq:pre-L2},\eqref{eq:LN1}}{=}&
-2y\triangleleft_{A}(x\triangleright_{A}z)-2(x\triangleright_{A}y)\triangleright_{A}z\\
&\overset{\eqref{eq:LN2}}{=}&0,\\
&&(x\succ_{A}y)\succ_{A}z+y\prec_{A}(x\succ_{A}z)-2x\succ_{A}(y\succ_{A}z)
-2x\succ_{A}(y\prec_{A}z)\\
&\overset{\eqref{eq:-2algebra}}{=}&(x\triangleright_{A}y)\triangleright_{A}z
-2(x\triangleright_{A}y)\triangleleft_{A}z-2(x\triangleleft_{A}y)\triangleright_{A}z
+4(x\triangleleft_{A}y)\triangleleft_{A}z+y\triangleleft_{A}(x\triangleright_{A}z)\\
&&-2y\triangleright_{A}(x\triangleright_{A}z)-2y\triangleleft_{A}(x\triangleleft_{A}z)
+4y\triangleright_{A}(x\triangleleft_{A}z)-2x\triangleright_{A}(y\triangleright_{A}z)
+4x\triangleleft_{A}(y\triangleright_{A}z)\\
&&+4x\triangleright_{A}(y\triangleleft_{A}z)
-8x\triangleleft_{A}(y\triangleleft_{A}z)-2x\triangleright_{A}(y\triangleleft_{A}z)
+4x\triangleleft_{A}(y\triangleleft_{A}z)+4x\triangleright_{A}(y\triangleright_{A}z)\\
&&-8x\triangleleft_{A}(y\triangleright_{A}z)\\
&\overset{\eqref{eq:LN1},\eqref{eq:LN2}}{=}&-2(x\triangleright_{A}y)\triangleleft_{A}z
-6x\triangleleft_{A}(y\triangleright_{A}z)+4(x\triangleleft_{A}y)\triangleleft_{A}z
-2y\triangleright_{A}(x\triangleright_{A}z)-6x\triangleleft_{A}(y\triangleleft_{A}z)\\
&&+4y\triangleright_{A}(x\triangleleft_{A}z)+2x\triangleright_{A}(y\triangleright_{A}z)
+2x\triangleright_{A}(y\triangleleft_{A}z)\\
&\overset{\eqref{eq:pre-L2},\eqref{eq:pre-L3}}{=}&
2(y\triangleright_{A}x)\triangleleft_{A}z-2y\triangleright_{A}(x\triangleleft_{A}z)
-2(x\triangleright_{A}y)\triangleleft_{A}z-2y\triangleright_{A}(x\triangleright_{A}z)
+2x\triangleright_{A}(y\triangleright_{A}z)\\
&&+2x\triangleright_{A}(y\triangleleft_{A}z)\\
&\overset{\eqref{eq:pre-L2},\eqref{eq:LN1}}{=}&2y\triangleleft_{A}(x\triangleright_{A}z)
-2x\triangleleft_{A}(y\triangleright_{A}z)-2y\triangleright_{A}(x\triangleright_{A}z)
+2x\triangleright_{A}(y\triangleright_{A}z)\\
&\overset{\eqref{eq:pre-L1},\eqref{eq:LN2}}{=}&0.
\end{eqnarray*}}
\end{proof}

By \eqref{eq:-2algebra}, we can also express $ \vdash_{A}$ and $\dashv_{A}$ in terms of $\succ_{A}$ and $\prec_{A}$, that is, we have the following expression:
\begin{eqnarray}
x\vdash_{A}y=-\frac{1}{3}x\succ_{A}y-\frac{2}{3}x\prec_{A}y,\;
x\dashv_{A}y=\frac{1}{3}y\prec_{A}x+\frac{2}{3}y\succ_{A}x,\;\forall x,y\in A.
\end{eqnarray}
For simplicity, we
 adjust them as
\begin{eqnarray}\label{eq:2algebra}
x\vdash_{A}y=x\succ_{A}y+2x\prec_{A}y,\;
x\dashv_{A}y=-y\prec_{A}x-2y\succ_{A}x,\;\forall x,y\in A.
\end{eqnarray}

Now we are motivated to study the converse side of Proposition \ref{pro:-2-alg}.

\begin{pro}\label{pro:2-alg}
Let $(A,\succ_{A},\prec_{A})$ be an anti-pre-Leibniz algebra, and
$\vdash_{A},\dashv_{A}:A\otimes A\rightarrow A$ be multiplications given by
\eqref{eq:2algebra}.
Then $(A,\vdash_{A},\dashv_{A})$ is a transformed pre-Leibniz algebra if and only if $(A,\succ_{A},\prec_{A})$ is further an admissible Novikov dialgebra. Moreover, in this case, $(A,\vdash_{A},\dashv_{A})$ is a Novikov dialgebra.
\end{pro}
%\textcolor{blue}{GL:Please revise the proof accordingly}
\begin{proof}
For all  $x,y,z\in A$, we have
{\small
\begin{eqnarray*}
&&x\vdash_{A}(y\vdash_{A}z)-(x\vdash_{A}y)\vdash_{A}z
+(y\dashv_{A}x)\vdash_{A}z-y\vdash_{A}(x\vdash_{A}z)\\
&\overset{\eqref{eq:2algebra}}{=}&
x\succ_{A}(y\circ_{A}z)
+x\succ_{A}(y\prec_{A}z)+2x\prec_{A}(y\circ_{A}z)
+2x\prec_{A}(y\prec_{A}z)-3(x\circ_{A}y)\succ_{A}z\\
&&-6(x\circ_{A}y)\prec_{A}z
-y\succ_{A}(x\circ_{A}z)-y\succ_{A}(x\prec_{A}z)-2y\prec_{A}(x\circ_{A}z)
-2y\prec_{A}(x\prec_{A}z)\\
&\overset{\eqref{eq:anLei1},\eqref{eq:anLei2}}{=}&
4y\succ_{A}(x\circ_{A}z)-4x\succ_{A}(y\circ_{A}z)+2x\prec_{A}(y\circ_{A}z)
+2x\prec_{A}(y\prec_{A}z)-2y\succ_{A}(x\succ_{A}z)\\
&&+2x\succ_{A}(y\succ_{A}z)
-2y\prec_{A}(x\circ_{A}z)-2y\prec_{A}(x\prec_{A}z)\\
&\overset{\eqref{eq:anLei1}}{=}&-4(x\circ_{A}y)\prec_{A}z
+2x\prec_{A}(y\prec_{A}z)-2y\prec_{A}(x\prec_{A}z)
+2x\prec_{A}(y\circ_{A}z)-2y\succ_{A}(x\succ_{A}z)\\
&&+2x\succ_{A}(y\succ_{A}z)-2y\prec_{A}(x\circ_{A}z)\\
&\overset{\eqref{eq:anLei2},\eqref{eq:anLei equivalent}}{=}&
-4(x\circ_{A}y)\prec_{A}z+2x\prec_{A}(y\prec_{A}z)
-2y\prec_{A}(x\prec_{A}z)\\
&\overset{\eqref{eq:anLei4}}{=}&
4(y\prec_{A}x)\prec_{A}z-4(x\prec_{A}y)\prec_{A}z
+2x\prec_{A}(y\prec_{A}z)-2y\prec_{A}(x\prec_{A}z).
\end{eqnarray*}}Thus \eqref{eq:nd1} holds if and only if \eqref{eq:adm Nov dia2} holds. Similarly, \eqref{eq:nd2} holds if and only if \eqref{eq:adm Nov dia3} and \eqref{eq:adm Nov dia1} hold, and
\eqref{eq:nd3} holds if and only if \eqref{eq:adm Nov dia1} holds.
\delete{
\begin{eqnarray*}
&&x\triangleleft_{A}(y\triangleright_{A}z)-y\triangleright_{A}(x\triangleleft_{A}z)
+(y\triangleright_{A}x)\triangleleft_{A}z+x\triangleleft_{A}(y\triangleleft_{A}z)\\
&\overset{\eqref{eq:2algebra}}{=}&3x\prec_{A}(y\circ_{A}z)
+6x\succ_{A}(y\circ_{A}z)-y\succ_{A}(x\circ_{A}z)-y\succ_{A}(x\succ_{A}z)
-2y\prec_{A}(x\circ_{A}z)\\
&&-2y\prec_{A}(x\succ_{A}z)
+(y\circ_{A}x)\prec_{A}z+(y\prec_{A}x)\prec_{A}z+2(y\circ_{A}x)\succ_{A}z
+2(y\prec_{A}x)\succ_{A}z\\
&\overset{\eqref{eq:anLei1},\eqref{eq:anLei equivalent}}{=}&
x\prec_{A}(y\circ_{A}z)+5x\succ_{A}(y\circ_{A}z)
-y\succ_{A}(x\succ_{A}z)-2y\prec_{A}(x\succ_{A}z)
+(y\prec_{A}x)\prec_{A}z\\
&&+2(y\prec_{A}x)\succ_{A}z\\
&\overset{\eqref{eq:anLei3}}{=}&(y\succ_{A}x)\prec_{A}z
-y\succ_{A}(x\prec_{A}z)+5x\succ_{A}(y\circ_{A}z)
-y\succ_{A}(x\succ_{A}z)+(y\prec_{A}x)\prec_{A}z\\
&&-2y\prec_{A}(x\succ_{A}z)+2(y\prec_{A}x)\succ_{A}z\\
&\overset{\eqref{eq:anLei1}}{=}&4x\succ_{A}(y\circ_{A}z)
-2y\prec_{A}(x\succ_{A}z)+2(y\prec_{A}x)\succ_{A}z,\\
&&(x\triangleright_{A}y)\triangleleft_{A}z+(y\triangleleft_{A}x)\triangleleft_{A}z\\
&\overset{\eqref{eq:2algebra}}{=}&(x\circ_{A}y)\prec_{A}z
+(x\prec_{A}y)\prec_{A}z+2(x\circ_{A}y)\succ_{A}z
+2(x\prec_{A}y)\succ_{A}z+(y\circ_{A}x)\prec_{A}z\\
&&+(y\succ_{A}x)\prec_{A}z+2(y\circ_{A}x)\succ_{A}z
+2(y\succ_{A}x)\succ_{A}z\\
&\overset{\eqref{eq:anLei1},\eqref{eq:anLei2}}{=}&
(x\prec_{A}y)\prec_{A}z+(y\succ_{A}x)\prec_{A}z
+2(x\prec_{A}y)\succ_{A}z+2(y\succ_{A}x)\succ_{A}z\\
&\overset{\eqref{eq:anLei4}}{=}&
2(x\prec_{A}y)\succ_{A}z+2(y\succ_{A}x)\succ_{A}z.
\end{eqnarray*}}Then $(A,\vdash_{A},\dashv_{A})$ is a transformed pre-Leibniz algebra if and only if $(A,\succ_{A},\prec_{A})$ is an admissible Novikov dialgebra. Moreover, if $(A,\succ_{A},\prec_{A})$ is an admissible Novikov dialgebra, then we have
\begin{eqnarray*}
&&(z\dashv_{A}y)\dashv_{A}x-(z\dashv_{A}x)\dashv_{A}y\\
&\overset{\eqref{eq:2algebra}}{=}&
x\prec_{A}(y\circ_{A}z)+x\prec_{A}(y\succ_{A}z)+2x\succ_{A}(y\circ_{A}z)
+2x\succ_{A}(y\succ_{A}z)-y\prec_{A}(x\circ_{A}z)\\
&&-y\prec_{A}(x\succ_{A}z)
-2y\succ_{A}(x\circ_{A}z)-2y\succ_{A}(x\succ_{A}z)\\
&\overset{\eqref{eq:anLei1},\eqref{eq:anLei equivalent}}{=}&
2(x\circ_{A}y)\prec_{A}z+(x\circ_{A}y)\succ_{A}z+x\prec_{A}(y\succ_{A}z)
+2x\succ_{A}(y\succ_{A}z)-y\prec_{A}(x\succ_{A}z)\\
&&-2y\succ_{A}(x\succ_{A}z)\\
&\overset{\eqref{eq:anLei2}}{=}&
2(x\circ_{A}y)\prec_{A}z+x\prec_{A}(y\succ_{A}z)-y\prec_{A}(x\succ_{A}z)
-(x\circ_{A}y)\succ_{A}z\\
&\overset{\eqref{eq:anLei equivalent}}{=}&
2(x\circ_{A}y)\prec_{A}z-x\prec_{A}(y\prec_{A}z)
+y\prec_{A}(x\prec_{A}z)\\
&\overset{\eqref{eq:anLei4},\eqref{eq:adm Nov dia2}}{=}&0.
\end{eqnarray*}
Thus \eqref{eq:nd4} holds. Similarly, \eqref{eq:nd5} holds.
Therefore, $(A,\vdash_{A},\dashv_{A})$ is a Novikov dialgebra.
\delete{
\begin{eqnarray*}
&&(x\triangleleft_{A}y)\triangleright_{A}z-x\triangleleft_{A}(y\triangleright_{A}z)\\
&\overset{\eqref{eq:2algebra}}{=}&
(x\circ_{A}y)\succ_{A}z+(x\succ_{A}y)\succ_{A}z+2(x\circ_{A}y)\prec_{A}z
+2(x\succ_{A}y)\prec_{A}z-x\prec_{A}(y\circ_{A}z)\\
&&-x\prec_{A}(y\prec_{A}z)
-2x\succ_{A}(y\circ_{A}z)-2x\succ_{A}(y\prec_{A}z)\\
&\overset{\eqref{eq:anLei1},\eqref{eq:anLei equivalent}}{=}&
(x\succ_{A}y)\succ_{A}z-y\prec_{A}(x\circ_{A}z)-2y\succ_{A}(x\circ_{A}z)
+2(x\succ_{A}y)\prec_{A}z-x\prec_{A}(y\prec_{A}z)\\
&&-2x\succ_{A}(y\prec_{A}z)\\
&\overset{\eqref{eq:anLei3}}{=}&
y\prec_{A}(x\circ_{A}z)+(x\succ_{A}y)\succ_{A}z-2y\succ_{A}(x\circ_{A}z)
-x\prec_{A}(y\prec_{A}z)\\
&\overset{\eqref{eq:adm Nov dia3}}{=}&
y\prec_{A}(x\circ_{A}z)+(x\succ_{A}y)\succ_{A}z-(y\succ_{A}x)\succ_{A}z
-x\prec_{A}(y\succ_{A}z)-x\prec_{A}(y\prec_{A}z)\\
&\overset{\eqref{eq:anLei equivalent},\eqref{eq:adm Nov dia1}}{=}&0,\\
&&(x\triangleleft_{A}y)\triangleright_{A}z+(y\triangleright_{A}x)\triangleright_{A}z\\
&\overset{\eqref{eq:2algebra}}{=}&
(x\prec_{A}y)\succ_{A}z+2(x\succ_{A}y)\succ_{A}z
+2(x\prec_{A}y)\prec_{A}z+4(x\succ_{A}y)\prec_{A}z
+(y\succ_{A}x)\succ_{A}z\\
&&+2(y\prec_{A}x)\succ_{A}z
+2(y\succ_{A}x)\prec_{A}z+4(y\prec_{A}x)\prec_{A}z\\
&\overset{\eqref{eq:anLei4},\eqref{eq:adm Nov dia1}}{=}&0.
\end{eqnarray*}}
\end{proof}

\begin{defi}\cite{MaSa}
	A {\bf compatible Leibniz algebra} is a triple $(A,\circ_{1_{A}},\circ_{2_{A}})$, such that $(A,\circ_{1_{A}})$
	and $(A,\circ_{2_{A}})$ are Leibniz algebras satisfying the following equation:
	{\small
		\begin{equation}
			x\circ_{2_{A}}(y\circ_{1_{A}}z)+x\circ_{1_{A}}(y\circ_{2_{A}}z)
			-(x\circ_{1_{A}}y)\circ_{2_{A}}z-(x\circ_{2_{A}}y)\circ_{1_{A}}z
			-y\circ_{2_{A}}(x\circ_{1_{A}}z)-y\circ_{1_{A}}(x\circ_{2_{A}}z)=0
			,\forall x,y,z\in A.
	\end{equation}}
\end{defi}

\begin{pro}\label{pro:comLei and anti}
	Let $(A,\succ_{A},\prec_{A})$ be an anti-pre-Leibniz algebra and
	$(A,\circ_{A})$ be the sub-adjacent Leibniz algebra. Then $(A\oplus A^*,\circ_{1_{d}},\circ_{2_{d}})$ with multiplications $\circ_{1_{d}},\circ_{2_{d}}:(A\oplus A^*)\otimes (A\oplus A^*)\rightarrow A\oplus A^*$ respectively defined by \eqref{eq:left inv2} and
	\begin{equation}
		%&&(x+a^{*})\circ_{1_{d}}(y+b^{*})=x\circ_{A}y-\mathcal{L}^{*}_{\succ_{A}}(x)b^{*} +(\mathcal{L}^{*}_{\succ_{A}}+\mathcal{R}^{*}_{\prec_{A}})(y)a^{*},\\
		(x+a^{*})\circ_{2_{d}}(y+b^{*})=x\circ_{A}y+\mathcal{L}^{*}_{\circ_{A}}(x)b^{*}
		-(\mathcal{L}^{*}_{\circ_{A}}
		+\mathcal{R}^{*}_{\circ_{A}})(y)a^{*},\;\forall x,y\in A,a^{*},b^{*}\in A^{*}
	\end{equation}
	is a compatible Leibniz algebra if and only if $(A,\succ_{A},\prec_{A})$ is an admissible Novikov dialgebra.
\end{pro}

\begin{proof}
	By Lemma \ref{ex:rep Leibniz} and Proposition \ref{pro:antipreLeibniz}, $(A\oplus A^*,\circ_{1_{d}})$ and $(A\oplus A^*,\circ_{2_{d}})$ are Leibniz algebras.
	Let $x,y,z\in A,a^{*},b^{*},c^{*}\in A^{*}$, we have
	\begin{eqnarray*}
		&&(x+a^*)\circ_{2_{d}}\big((y+b^*)\circ_{1_{d}}(z+c^*)\big)
		+(x+a^*)\circ_{1_{d}}\big((y+b^*)\circ_{2_{d}}(z+c^*)\big)\\
		&&
		-\big((x+a^*)\circ_{1_{d}}(y+b^*)\big)\circ_{2_{d}}(z+c^*)
		-\big((x+a^*)\circ_{2_{d}}(y+b^*)\big)\circ_{1_{d}}(z+c^*)\\
		&&
		-(y+b^*)\circ_{2_{d}}\big((x+a^*)\circ_{1_{d}}(z+c^*)\big)
		-(y+b^*)\circ_{1_{d}}\big((x+a^*)\circ_{2_{d}}(z+c^*)\big)\\
		%&=&(x+a^*)\circ_{2_{d}}\big(y\circ_{A}z -\mathcal{L}^{*}_{\succ_{A}}(y)c^{*} +(\mathcal{L}^{*}_{\succ_{A}} +\mathcal{R}^{*}_{\prec_{A}})(z)b^{*}\big)\\
		%&&+(x+a^*)\circ_{1_{d}}\big(y\circ_{A}z +\mathcal{L}^{*}_{\circ_{A}}(y)c^{*} -(\mathcal{L}^{*}_{\circ_{A}} +\mathcal{R}^{*}_{\circ_{A}})(z)b^{*}\big)\\
		%&&-\big(x\circ_{A}y-\mathcal{L}^{*}_{\succ_{A}}(x)b^{*} +(\mathcal{L}^{*}_{\succ_{A}} +\mathcal{R}^{*}_{\prec_{A}})(y)a^{*}\big)\circ_{2_{d}}(z+c^*)\\
		%&&-\big(x\circ_{A}y+\mathcal{L}^{*}_{\circ_{A}}(x)b^{*} -(\mathcal{L}^{*}_{\circ_{A}} +\mathcal{R}^{*}_{\circ_{A}})(y)a^{*}\big)\circ_{1_{d}}(z+c^*)\\
		%&&-(y+b^*)\circ_{2_{d}}\big(x\circ_{A}z -\mathcal{L}^{*}_{\succ_{A}}(x)c^{*} +(\mathcal{L}^{*}_{\succ_{A}} +\mathcal{R}^{*}_{\prec_{A}})(z)a^{*}\big)\\
		%&&-(y+b^*)\circ_{1_{d}}\big(x\circ_{A}z +\mathcal{L}^{*}_{\circ_{A}}(x)c^{*} -(\mathcal{L}^{*}_{\circ_{A}} +\mathcal{R}^{*}_{\circ_{A}})(z)a^{*}\big)\\
		&=&2x\circ_{A}(y\circ_{A}z)
		-\mathcal{L}^{*}_{\circ_{A}}(x)\mathcal{L}^{*}_{\succ_{A}}(y)c^{*}
		+\mathcal{L}^{*}_{\circ_{A}}(x)(\mathcal{L}^{*}_{\succ_{A}}
		+\mathcal{R}^{*}_{\prec_{A}})(z)b^{*}
		-(\mathcal{L}^{*}_{\circ_{A}}
		+\mathcal{R}^{*}_{\circ_{A}})(y\circ_{A}z)a^{*}\\
		&&
		-\mathcal{L}^{*}_{\succ_{A}}(x)\mathcal{L}^{*}_{\circ_{A}}(y)c^{*}
		+\mathcal{L}^{*}_{\succ_{A}}(x)(\mathcal{L}^{*}_{\circ_{A}}
		+\mathcal{R}^{*}_{\circ_{A}})(z)b^{*}+(\mathcal{L}^{*}_{\succ_{A}}
		+\mathcal{R}^{*}_{\prec_{A}})(y\circ_{A}z)a^{*}
		-2(x\circ_{A}y)\circ_{A}z\\
		&&
		-\mathcal{L}^{*}_{\circ_{A}}(x\circ_{A}y)c^{*}
		-(\mathcal{L}^{*}_{\circ_{A}}
		+\mathcal{R}^{*}_{\circ_{A}})(z)\mathcal{L}^{*}_{\succ_{A}}(x)b^{*}
		+(\mathcal{L}^{*}_{\circ_{A}}
		+\mathcal{R}^{*}_{\circ_{A}})(z)(\mathcal{L}^{*}_{\succ_{A}}
		+\mathcal{R}^{*}_{\prec_{A}})(y)a^{*}\\
		&&
		+\mathcal{L}^{*}_{\succ_{A}}(x\circ_{A}y)c^{*}
		-(\mathcal{L}^{*}_{\succ_{A}}
		+\mathcal{R}^{*}_{\prec_{A}})(z)\mathcal{L}^{*}_{\circ_{A}}(x)b^{*}
		+(\mathcal{L}^{*}_{\succ_{A}}
		+\mathcal{R}^{*}_{\prec_{A}})(z)(\mathcal{L}^{*}_{\circ_{A}}
		+\mathcal{R}^{*}_{\circ_{A}})(y)a^{*}\\
		&&
		-2y\circ_{A}(x\circ_{A}z)
		+\mathcal{L}^{*}_{\circ_{A}}(y)\mathcal{L}^{*}_{\succ_{A}}(x)c^{*}
		-\mathcal{L}^{*}_{\circ_{A}}(y)(\mathcal{L}^{*}_{\succ_{A}}
		+\mathcal{R}^{*}_{\prec_{A}})(z)a^{*}
		+(\mathcal{L}^{*}_{\circ_{A}}
		+\mathcal{R}^{*}_{\circ_{A}})(x\circ_{A}z)b^{*}\\
		&&
		+\mathcal{L}^{*}_{\succ_{A}}(y)\mathcal{L}^{*}_{\circ_{A}}(x)c^{*}
		-\mathcal{L}^{*}_{\succ_{A}}(y)(\mathcal{L}^{*}_{\circ_{A}}
		+\mathcal{R}^{*}_{\circ_{A}})(z)a^{*}
		-(\mathcal{L}^{*}_{\succ_{A}}
		+\mathcal{R}^{*}_{\prec_{A}})(x\circ_{A}z)b^{*}\\
		&\overset{\eqref{eq:Leibniz}}{=}&
		-\mathcal{L}^{*}_{\circ_{A}}(x)\mathcal{L}^{*}_{\succ_{A}}(y)c^{*}
		+\mathcal{L}^{*}_{\circ_{A}}(x)(\mathcal{L}^{*}_{\succ_{A}}
		+\mathcal{R}^{*}_{\prec_{A}})(z)b^{*}
		-(\mathcal{L}^{*}_{\circ_{A}}
		+\mathcal{R}^{*}_{\circ_{A}})(y\circ_{A}z)a^{*}
		-\mathcal{L}^{*}_{\succ_{A}}(x)\mathcal{L}^{*}_{\circ_{A}}(y)c^{*}\\
		&&
		+\mathcal{L}^{*}_{\succ_{A}}(x)(\mathcal{L}^{*}_{\circ_{A}}
		+\mathcal{R}^{*}_{\circ_{A}})(z)b^{*}+(\mathcal{L}^{*}_{\succ_{A}}
		+\mathcal{R}^{*}_{\prec_{A}})(y\circ_{A}z)a^{*}
		-\mathcal{L}^{*}_{\circ_{A}}(x\circ_{A}y)c^{*}\\
		&&
		-(\mathcal{L}^{*}_{\circ_{A}}
		+\mathcal{R}^{*}_{\circ_{A}})(z)\mathcal{L}^{*}_{\succ_{A}}(x)b^{*}
		+(\mathcal{L}^{*}_{\circ_{A}}
		+\mathcal{R}^{*}_{\circ_{A}})(z)(\mathcal{L}^{*}_{\succ_{A}}
		+\mathcal{R}^{*}_{\prec_{A}})(y)a^{*}
		+\mathcal{L}^{*}_{\succ_{A}}(x\circ_{A}y)c^{*}\\
		&&
		-(\mathcal{L}^{*}_{\succ_{A}}
		+\mathcal{R}^{*}_{\prec_{A}})(z)\mathcal{L}^{*}_{\circ_{A}}(x)b^{*}
		+(\mathcal{L}^{*}_{\succ_{A}}
		+\mathcal{R}^{*}_{\prec_{A}})(z)(\mathcal{L}^{*}_{\circ_{A}}
		+\mathcal{R}^{*}_{\circ_{A}})(y)a^{*}
		+\mathcal{L}^{*}_{\circ_{A}}(y)\mathcal{L}^{*}_{\succ_{A}}(x)c^{*}\\
		&&
		-\mathcal{L}^{*}_{\circ_{A}}(y)(\mathcal{L}^{*}_{\succ_{A}}
		+\mathcal{R}^{*}_{\prec_{A}})(z)a^{*}
		+(\mathcal{L}^{*}_{\circ_{A}}
		+\mathcal{R}^{*}_{\circ_{A}})(x\circ_{A}z)b^{*}
		+\mathcal{L}^{*}_{\succ_{A}}(y)\mathcal{L}^{*}_{\circ_{A}}(x)c^{*}\\
		&&
		-\mathcal{L}^{*}_{\succ_{A}}(y)(\mathcal{L}^{*}_{\circ_{A}}
		+\mathcal{R}^{*}_{\circ_{A}})(z)a^{*}
		-(\mathcal{L}^{*}_{\succ_{A}}
		+\mathcal{R}^{*}_{\prec_{A}})(x\circ_{A}z)b^{*}.
	\end{eqnarray*}
	Suppose that $(A\oplus A^*,\circ_{1_{d}},\circ_{2_{d}})$
	is a compatible Leibniz algebra. Then we have
	{\small
	\begin{eqnarray*}
		0&=&\langle -\mathcal{L}^{*}_{\circ_{A}}(x)\mathcal{L}^{*}_{\succ_{A}}(y)c^{*}
		-\mathcal{L}^{*}_{\succ_{A}}(x)\mathcal{L}^{*}_{\circ_{A}}(y)c^{*}
		-\mathcal{L}^{*}_{\circ_{A}}(x\circ_{A}y)c^{*}
		+\mathcal{L}^{*}_{\succ_{A}}(x\circ_{A}y)c^{*}\\
		&&
		+\mathcal{L}^{*}_{\circ_{A}}(y)\mathcal{L}^{*}_{\succ_{A}}(x)c^{*}
		+\mathcal{L}^{*}_{\succ_{A}}(y)\mathcal{L}^{*}_{\circ_{A}}(x)c^{*},z\rangle\\
		&=&\langle c^{*},-y\succ_{A}(x\circ_{A}z)-y\circ_{A}(x\succ_{A}z)
		+(x\circ_{A}y)\prec_{A}z+x\succ_{A}(y\circ_{A}z)
		+x\circ_{A}(y\succ_{A}z)\rangle,\\
		0&=&\langle \mathcal{L}^{*}_{\circ_{A}}(x)(\mathcal{L}^{*}_{\succ_{A}}
		+\mathcal{R}^{*}_{\prec_{A}})(z)b^{*}
		+\mathcal{L}^{*}_{\succ_{A}}(x)(\mathcal{L}^{*}_{\circ_{A}}
		+\mathcal{R}^{*}_{\circ_{A}})(z)b^{*}
		-(\mathcal{L}^{*}_{\circ_{A}}
		+\mathcal{R}^{*}_{\circ_{A}})(z)\mathcal{L}^{*}_{\succ_{A}}(x)b^{*}\\
		&&
		-(\mathcal{L}^{*}_{\succ_{A}}
		+\mathcal{R}^{*}_{\prec_{A}})(z)\mathcal{L}^{*}_{\circ_{A}}(x)b^{*}
		+(\mathcal{L}^{*}_{\circ_{A}}
		+\mathcal{R}^{*}_{\circ_{A}})(x\circ_{A}z)b^{*}
		-(\mathcal{L}^{*}_{\succ_{A}}
		+\mathcal{R}^{*}_{\prec_{A}})(x\circ_{A}z)b^{*},y\rangle\\
		&=&\langle b^{*},z\succ_{A}(x\circ_{A}y)+(x\circ_{A}y)\prec_{A}z
		+z\circ_{A}(x\succ_{A}y)+(x\succ_{A}y)\circ_{A}z
		-x\succ_{A}(z\circ_{A}y)\\
		&&
		-x\succ_{A}(y\circ_{A}z)
		-x\circ_{A}(z\succ_{A}y)-x\circ_{A}(y\prec_{A}z)
		-(x\circ_{A}z)\prec_{A}y-y\succ_{A}(x\circ_{A}z)\rangle,\\
		0&=&\langle-(\mathcal{L}^{*}_{\circ_{A}}
		+\mathcal{R}^{*}_{\circ_{A}})(y\circ_{A}z)a^{*}
		+(\mathcal{L}^{*}_{\succ_{A}}
		+\mathcal{R}^{*}_{\prec_{A}})(y\circ_{A}z)a^{*}
		+(\mathcal{L}^{*}_{\circ_{A}}
		+\mathcal{R}^{*}_{\circ_{A}})(z)(\mathcal{L}^{*}_{\succ_{A}}
		+\mathcal{R}^{*}_{\prec_{A}})(y)a^{*}\\
		&&
		+(\mathcal{L}^{*}_{\succ_{A}}
		+\mathcal{R}^{*}_{\prec_{A}})(z)(\mathcal{L}^{*}_{\circ_{A}}
		+\mathcal{R}^{*}_{\circ_{A}})(y)a^{*}
		-\mathcal{L}^{*}_{\circ_{A}}(y)(\mathcal{L}^{*}_{\succ_{A}}
		+\mathcal{R}^{*}_{\prec_{A}})(z)a^{*}
		-\mathcal{L}^{*}_{\succ_{A}}(y)(\mathcal{L}^{*}_{\circ_{A}}
		+\mathcal{R}^{*}_{\circ_{A}})(z)a^{*},x\rangle\\
		&=&\langle a^{*},(y\circ_{A}z)\prec_{A}x+x\succ_{A}(y\circ_{A}z)
		+y\succ_{A}(z\circ_{A}x)+y\succ_{A}(x\circ_{A}z)
		+(z\circ_{A}x)\prec_{A}y\\
		&&
		+(x\circ_{A}z)\prec_{A}y
		+y\circ_{A}(z\succ_{A}x)+y\circ_{A}(x\prec_{A}z)
		+(z\succ_{A}x)\circ_{A}y+(x\prec_{A}z)\circ_{A}y\\
		&&
		-z\succ_{A}(y\circ_{A}x)-(y\circ_{A}x)\prec_{A}z
		-z\circ_{A}(y\succ_{A}x)-(y\succ_{A}x)\circ_{A}z\rangle.
	\end{eqnarray*}}Thus we obtain the following equations:
{\small
	\begin{eqnarray}
		0&=&-y\succ_{A}(x\circ_{A}z)-y\circ_{A}(x\succ_{A}z)
		+(x\circ_{A}y)\prec_{A}z+x\succ_{A}(y\circ_{A}z)
		+x\circ_{A}(y\succ_{A}z),\label{eq:1}\\
		0&=&z\succ_{A}(x\circ_{A}y)+(x\circ_{A}y)\prec_{A}z
		+z\circ_{A}(x\succ_{A}y)+(x\succ_{A}y)\circ_{A}z
		-x\succ_{A}(z\circ_{A}y)\nonumber\\
		&&
		-x\succ_{A}(y\circ_{A}z)
		-x\circ_{A}(z\succ_{A}y)-x\circ_{A}(y\prec_{A}z)
		-(x\circ_{A}z)\prec_{A}y-y\succ_{A}(x\circ_{A}z)\nonumber\\
		&\overset{\eqref{eq:1}}{=}&
		(x\circ_{A}y)\prec_{A}z+(x\succ_{A}y)\circ_{A}z
		-x\succ_{A}(y\circ_{A}z)-x\circ_{A}(y\prec_{A}z)
		-y\succ_{A}(x\circ_{A}z),\label{eq:2}\\
		0&=&(y\circ_{A}z)\prec_{A}x+x\succ_{A}(y\circ_{A}z)
		+y\succ_{A}(z\circ_{A}x)+y\succ_{A}(x\circ_{A}z)
		+(z\circ_{A}x)\prec_{A}y\nonumber\\
		&&
		+(x\circ_{A}z)\prec_{A}y
		+y\circ_{A}(z\succ_{A}x)+y\circ_{A}(x\prec_{A}z)
		+(z\succ_{A}x)\circ_{A}y+(x\prec_{A}z)\circ_{A}y\nonumber\\
		&&
		-z\succ_{A}(y\circ_{A}x)-(y\circ_{A}x)\prec_{A}z
		-z\circ_{A}(y\succ_{A}x)-(y\succ_{A}x)\circ_{A}z\nonumber\\
		&\overset{\eqref{eq:1},\eqref{eq:2}}{=}&
		(z\circ_{A}x)\prec_{A}y+(x\circ_{A}z)\prec_{A}y
		+(z\succ_{A}x)\circ_{A}y+(x\prec_{A}z)\circ_{A}y.\label{eq:3}
	\end{eqnarray}}Since $(A,\succ_{A},\prec_{A})$ is an anti-pre-Leibniz algebra,
	we rewrite \eqref{eq:1}-\eqref{eq:3} as
	\begin{eqnarray*}
		0&=&-y\succ_{A}(x\circ_{A}z)-y\circ_{A}(x\succ_{A}z)
		+(x\circ_{A}y)\prec_{A}z+x\succ_{A}(y\circ_{A}z)
		+x\circ_{A}(y\succ_{A}z)\\
		&\overset{\eqref{eq:anLei1}}{=}&
		2(x\circ_{A}y)\prec_{A}z-y\circ_{A}(x\succ_{A}z)
		+x\circ_{A}(y\succ_{A}z)\\
		&\overset{\eqref{eq:anLei2}}{=}&
		2(x\circ_{A}y)\prec_{A}z-y\prec_{A}(x\succ_{A}z)
		+x\prec_{A}(y\succ_{A}z)-(x\circ_{A}y)\succ_{A}z\\
		&\overset{\eqref{eq:anLei equivalent}}{=}&
		2(x\circ_{A}y)\prec_{A}z-x\prec_{A}(y\prec_{A}z)
		+y\prec_{A}(x\prec_{A}z)\\
		&\overset{\eqref{eq:anLei4}}{=}&
		2(x\prec_{A}y)\prec_{A}z-2(y\prec_{A}x)\prec_{A}z
		-x\prec_{A}(y\prec_{A}z)+y\prec_{A}(x\prec_{A}z),\\
		0&=&(x\circ_{A}y)\prec_{A}z+(x\succ_{A}y)\circ_{A}z
		-x\succ_{A}(y\circ_{A}z)-x\circ_{A}(y\prec_{A}z)
		-y\succ_{A}(x\circ_{A}z)\\
		&\overset{\eqref{eq:anLei1}}{=}&
		-2y\succ_{A}(x\circ_{A}z)+(x\succ_{A}y)\circ_{A}z
		-x\circ_{A}(y\prec_{A}z)\\
		&\overset{\eqref{eq:anLei3}}{=}&
		-2y\succ_{A}(x\circ_{A}z)+(x\succ_{A}y)\succ_{A}z
		+y\prec_{A}(x\circ_{A}z)-x\prec_{A}(y\prec_{A}z)\\
		&\overset{\eqref{eq:anLei equivalent}}{=}&
		-2y\succ_{A}(x\circ_{A}z)-(x\prec_{A}y)\succ_{A}z
		+x\prec_{A}(y\succ_{A}z),\\
		0&=&(z\circ_{A}x)\prec_{A}y+(x\circ_{A}z)\prec_{A}y
		+(z\succ_{A}x)\circ_{A}y+(x\prec_{A}z)\circ_{A}y\\
		&\overset{\eqref{eq:anLei4}}{=}&
		(z\succ_{A}x)\succ_{A}y+(x\prec_{A}z)\succ_{A}y,
	\end{eqnarray*}
	that is, \eqref{eq:adm Nov dia1}-\eqref{eq:adm Nov dia3} hold. Therefore $(A,\succ_{A},\prec_{A})$ is an admissible Novikov dialgebra. The converse side is obtained similarly.
\end{proof}

With a similar proof of Proposition \ref{pro:comLei and anti},
we have the following conclusion.

\begin{pro}\label{pro:comLei and preLei}
Let $(A,\triangleright_{A},\triangleleft_{A})$ be a pre-Leibniz algebra and $(A,\bullet_{A})$ be the sub-adjacent Leibniz algebra. Then $(A\oplus A^*,\bullet_{1_{d}},\bullet_{2_{d}})$ with multiplications $\bullet_{1_{d}},\bullet_{2_{d}}:(A\oplus A^*)\otimes (A\oplus A^*)\rightarrow A\oplus A^*$ respectively defined by
	\begin{eqnarray}
		&&(x+a^{*})\bullet_{1_{d}}(y+b^{*})=x\bullet_{A}y
		+\mathcal{L}^{*}_{\triangleright_{A}}(x)b^{*} -(\mathcal{L}^{*}_{\triangleright_{A}}
		+\mathcal{R}^{*}_{\triangleleft_{A}})(y)a^{*},\\
		&&(x+a^{*})\bullet_{2_{d}}(y+b^{*})=x\bullet_{A}y
		+\mathcal{L}^{*}_{\bullet_{A}}(x)b^{*}
		-(\mathcal{L}^{*}_{\bullet_{A}}
		+\mathcal{R}^{*}_{\bullet_{A}})(y)a^{*},\;\forall x,y\in A,a^{*},b^{*}\in A^{*}
	\end{eqnarray}
	is a compatible Leibniz algebra if and only if the transformed pre-Leibniz algebra  $(A,\vdash_{A},\dashv_{A})$ is a Novikov dialgebra.
\end{pro}

Recall that a {\bf Gel'fand-Dorfman algebra} \cite{Gel} (or a {\bf GD algebra} in short) is a triple $(A$,
$[-,-]_{A},$
$\ast_{A})$, such that $(A,[-,-]_{A})$ is a Lie algebra, $(A,\ast_{A})$ is a Novikov algebra, and the following equation holds:
\begin{equation}\label{eq:GD alg}
[x\ast_{A}y,z]_{A}-[x\ast_{A}z,y]_{A}+[x,y]_{A}\ast_{A}z
-[x,z]_{A}\ast_{A}y-x\ast_{A}[y,z]_{A}=0,\;\forall x,y,z\in A.
\end{equation}
Now we introduce the notion of a Gel'fand-Dorfman dialgebra.

\begin{defi}
A {\bf Gel'fand-Dorfman dialgebra} (or a {\bf GD dialgebra} in short) is a quadruple $(A,\circ_{A},\vdash_{A},\dashv_{A})$, such that $(A,\circ_{A})$ is a Leibniz algebra, $(A,\vdash_{A},\dashv_{A})$ is a Novikov dialgebra, and the following equations hold:
\begin{align}
&x\vdash_{A}(y\circ_{A}z)-(x\vdash_{A}y)\circ_{A}z
-(x\circ_{A}y)\vdash_{A}z-y\circ_{A}(x\vdash_{A}z)
+ (x\circ_{A}z)\dashv_{A} y=0,\label{eq:GD dia1}\\
&x\circ_{A}(y\vdash_{A}z)-(y\circ_{A}z)\dashv_{A}x
+( y\dashv_{A}x )\circ_{A}z-(x\circ_{A}y)\vdash_{A}z
-y\vdash_{A}(x\circ_{A}z)=0,\label{eq:GD dia2}\\
&x\circ_{A}(z\dashv_{A}y)+ (y\circ_{A}z)\dashv_{A}x
-z\dashv_{A}(x\circ_{A}y) -y\circ_{A}(z\dashv_{A}x)
- (x\circ_{A}z)\dashv_{A}y=0,
\label{eq:GD dia3}
\end{align}
for all $x,y,z\in A$.
\end{defi}

%\textcolor{blue}{GL:Please use $x\vdash_{A}y$ instead of $x\triangleright_{A} y$, use $x\vdash_{A}y$ instead of $x\triangleright_{A} y$ in the notion of a Novikov dialgebra.}

\begin{ex}
\begin{enumerate}
	\item Let $(A,\vdash_{A},\dashv_{A})$ be a Novikov dialgebra. Then $(A,\vdash_{A},\dashv_{A})$ is also a transformed pre-Leibniz algebra. Hence $(A,\circ_{A})$ is a Leibniz algebra, where
    \begin{eqnarray*}
    x\circ_{A}y=x\vdash_{A}y-y\dashv_{A}x,\;\forall x,y\in A.
    \end{eqnarray*}
Moreover, for all $x,y,z\in A$, we have
\begin{eqnarray*}
&&x\vdash_{A}(y\circ_{A}z)-(x\vdash_{A}y)\circ_{A}z
-(x\circ_{A}y)\vdash_{A}z-y\circ_{A}(x\vdash_{A}z)
+(x\circ_{A}z)\dashv_{A}y\\
&\overset{\eqref{eq:transformed}}{=}&x\vdash_{A}(y\vdash_{A}z)-x\vdash_{A}(z\dashv_{A}y)
-2(x\vdash_{A}y)\vdash_{A}z+z\dashv_{A}(x\vdash_{A}y)
+(y\dashv_{A}x)\vdash_{A}z\\
&&-y\vdash_{A}(x\vdash_{A}z)
+2(x\vdash_{A}z)\dashv_{A}y-(z\dashv_{A}x)\dashv_{A}y\\
&\overset{\eqref{eq:nd1},\eqref{eq:nd2},\eqref{eq:nd5}}{=}&0.
\end{eqnarray*}
Hence \eqref{eq:GD dia1} holds. Similarly, \eqref{eq:GD dia2} and \eqref{eq:GD dia3} hold. Thus $(A,\circ_{A},\vdash_{A},\dashv_{A})$ is a GD dialgebra.
	\item Let $P$ be an averaging operator on a GD algebra
	$(A,[-,-]_{A},\ast_{A})$. Suppose that $\circ_{A},\vdash_{A}$,
	$\dashv_{A}:A\otimes A\rightarrow A$ are multiplications given by
\begin{equation}\label{eq:av} x\circ_{A}y=[P(x),y]_{A},\;x\vdash_{A}y=P(x)\ast_{A}y,\;
x\dashv_{A}y=x\ast_{A}P(y),\;\forall x,y\in A.
\end{equation}
Then for all $x,y,z\in A$, we have
\begin{eqnarray*}
&&x\vdash_{A}(y\circ_{A}z)-(x\vdash_{A}y)\circ_{A}z
-(x\circ_{A}y)\vdash_{A}z-y\circ_{A}(x\vdash_{A}z)
+(x\circ_{A}z)\dashv_{A}y\\
&\overset{\eqref{eq:av}}{=}&P(x)\ast_{A}[P(y),z]_{A}
-[P\big(P(x)\ast_{A}y\big),z]_{A}-P[P(x),y]_{A}\ast_{A}z
-[P(y),P(x)\ast_{A}z]_{A}\\
&&+[P(x),z]_{A}\ast_{A}P(y)\\
&\overset{\eqref{eq:Ao}}{=}&P(x)\ast_{A}[P(y),z]_{A}
-[P(x)\ast_{A}P(y),z]_{A}-[P(x),P(y)]_{A}\ast_{A}z
-[P(y),P(x)\ast_{A}z]_{A}\\
&&+[P(x),z]_{A}\ast_{A}P(y)\\
&\overset{\eqref{eq:GD alg}}{=}&0.
\end{eqnarray*}
Hence \eqref{eq:GD dia1} holds. Similarly, \eqref{eq:GD dia2} and \eqref{eq:GD dia3} hold. It is known that $(A,\circ_{A})$ defined by \eqref{eq:av} is a Leibniz algebra and $(A,\vdash_{A},\dashv_{A})$ defined by \eqref{eq:av} is a Novikov dialgebra.
Therefore $(A,\circ_{A},\vdash_{A},\dashv_{A})$ is a GD dialgebra.
\end{enumerate}
\end{ex}

\begin{pro}\label{pro:LN and Leib}
	Let $(A,\circ_{A},\vdash_{A},\dashv_{A})$ be a GD dialgebra with a derivation $P$. Define a multiplication
	$\cdot_{A}:A\otimes A\rightarrow A$ by
	\begin{equation}\label{eq:derivation LN} x\cdot_{A}y=x\circ_{A}y+P(x)\vdash_{A}y-P(y)\dashv_{A}x,\;\forall x,y\in A.
	\end{equation}
	Then $(A,\cdot_{A})$ is a Leibniz algebra.
\end{pro}

\begin{proof}
	For all $x,y,z\in A$, we have
	{\small
	\begin{eqnarray*}
		x\cdot_{A}(y\cdot_{A}z)
		&=&x\cdot_{A}\big(y\circ_{A}z+P(y)\vdash_{A}z-P(z)\dashv_{A}y\big)\\
		&=&x\circ_{A}\big(y\circ_{A}z+P(y)\vdash_{A}z- P(z)\dashv_{A}y\big)
		+P(x)\vdash_{A}\big(y\circ_{A}z+P(y)\vdash_{A}z
		-P(z)\dashv_{A}y\big)\\
		&& 
		-P\big(y\circ_{A}z+P(y)\vdash_{A}z
		-P(z)\dashv_{A}y\big)\dashv_{A}x\\
		&=&x\circ_{A}(y\circ_{A}z)+x\circ_{A}\big(P(y)\vdash_{A}z\big)-
		x\circ_{A}\big(P(z)\dashv_{A}y\big)
		+P(x)\vdash_{A}(y\circ_{A}z)\\
		&&
		+P(x)\vdash_{A}\big(P(y)\vdash_{A}z\big)
		-P(x)\vdash_{A}\big(P(z)\dashv_{A}y\big)
		-\Big(P\big(P(y)\big)\vdash_{A}z\Big)\dashv_{A}x\\
		&&
		-\big(P(y)\circ_{A}z\big)\dashv_{A}x 
		-\big(y\circ_{A}P(z)\big)\dashv_{A}x
		-\big(P(y)\vdash_{A}P(z)\big)\dashv_{A}x
		+\big(P(z)\dashv_{A}P(y)\big)\dashv_{A}x\\
		&&
		+\Big(P\big(P(z)\big)\dashv_{A}y\Big)\dashv_{A}x,\\
		(x\cdot_{A}y)\cdot_{A}z
		&=&\big(x\circ_{A}y+P(x)\vdash_{A}y
-P(y)\dashv_{A}x\big)\cdot_{A}z\\
		&=&\big(x\circ_{A}y+P(x)\vdash_{A}y
-P(y)\dashv_{A}x\big)\circ_{A}z
		+P\big(x\circ_{A}y+P(x)\vdash_{A}y
-P(y)\dashv_{A}x\big)\vdash_{A}z\\
		&& 
		-P(z)\dashv_{A}\big(x\circ_{A}y+P(x)\vdash_{A}y
-P(y)\dashv_{A}x\big)\\
		&=&(x\circ_{A}y)\circ_{A}z+\big(P(x)\vdash_{A}y\big)\circ_{A}z
		-\big(P(y)\dashv_{A}x\big)\circ_{A}z
		+\big(P(x)\circ_{A}y\big)\vdash_{A}z\\
		&& 
		+\big(x\circ_{A}P(y)\big)\vdash_{A}z
		+\big(P(x)\vdash_{A}P(y)\big)\vdash_{A}z
		+\Big(P\big(P(x)\big)\vdash_{A}y\Big)\vdash_{A}z
		\\
		&&
		-\big(P(y)\dashv_{A}P(x)\big)\vdash_{A}z
		-\Big(P\big(P(y)\big)\dashv_{A}x\Big)\vdash_{A}z
		-P(z)\dashv_{A}(x\circ_{A}y)\\
		&&
		-P(z)\dashv_{A}\big(P(x)\vdash_{A}y\big)
		+P(z)\dashv_{A}\big(P(y)\dashv_{A}x\big),\\
		y\cdot_{A}(x\cdot_{A}z)
		&=&y\circ_{A}(x\circ_{A}z)+y\circ_{A}\big(P(x)\vdash_{A}z\big)
		-y\circ_{A}\big(P(z)\dashv_{A}x\big)
		+P(y)\vdash_{A}(x\circ_{A}z)\\
		&& 
		+P(y)\vdash_{A}\big(P(x)\vdash_{A}z\big)
		-P(y)\vdash_{A}\big(P(z)\dashv_{A}x\big)
		-\Big(P\big(P(x)\big)\vdash_{A}z\Big)\dashv_{A}y\\
		&&
		-\big(P(x)\circ_{A}z\big)\dashv_{A}y 
		-\big(x\circ_{A}P(z)\big)\dashv_{A}y
		-\big(P(x)\vdash_{A}P(z)\big)\dashv_{A}y\\
		&&
		+\big(P(z)\dashv_{A}P(x)\big)\dashv_{A}y
		+\Big(P\big(P(z)\big)\dashv_{A}x\Big)\dashv_{A}y.
	\end{eqnarray*}}
	Thus we have
	{\small
		\begin{eqnarray*}
			&&x\cdot_{A}(y\cdot_{A}z)-(x\cdot_{A}y)\cdot_{A}z-y\cdot_{A}(x\cdot_{A}z)\\
			&=&x\circ_{A}(y\circ_{A}z)+x\circ_{A}\big(P(y)\vdash_{A}z\big)-
			x\circ_{A}\big(P(z)\dashv_{A}y\big)
			+P(x)\vdash_{A}(y\circ_{A}z)\\
			&&+P(x)\vdash_{A}\big(P(y)\vdash_{A}z\big)
			-P(x)\vdash_{A}\big(P(z)\dashv_{A}y\big)
			-\Big(P\big(P(y)\big)\vdash_{A}z\Big)\dashv_{A}x\\
			&&
			-\big(P(y)\circ_{A}z\big)\dashv_{A}x -\big(y\circ_{A}P(z)\big)\dashv_{A}x
			-\big(P(y)\vdash_{A}P(z)\big)\dashv_{A}x\\
			&&
			+\big(P(z)\dashv_{A}P(y)\big)\dashv_{A}x
			+\Big( P\big(P(z)\big)\dashv_{A}y\Big)\dashv_{A}x -(x\circ_{A}y)\circ_{A}z\\
			&&-\big(P(x)\vdash_{A}y\big)\circ_{A}z
			+\big(P(y)\dashv_{A}x\big)\circ_{A}z
			-\big(P(x)\circ_{A}y\big)\vdash_{A}z\\
			&&-\big(x\circ_{A}P(y)\big)\vdash_{A}z
			-\big(P(x)\vdash_{A}P(y)\big)\vdash_{A}z
			-\Big(P\big(P(x)\big)\vdash_{A}y\Big)\vdash_{A}z\\
			&&
			+\big(P(y)\dashv_{A}P(x)\big)\vdash_{A}z +\Big(P\big(P(y)\big)\dashv_{A}x\Big)\vdash_{A}z
			+P(z)\dashv_{A}(x\circ_{A}y)\\
			&&
			+P(z)\dashv_{A}\big(P(x)\vdash_{A}y\big) 
			-P(z)\dashv_{A}\big(P(y)\dashv_{A}x\big) 
			 -y\circ_{A}(x\circ_{A}z)\\
			 &&-y\circ_{A}\big(P(x)\vdash_{A}z\big) 
			+y\circ_{A}\big(P(z)\dashv_{A}x\big)
			-P(y)\vdash_{A}(x\circ_{A}z) \\
			&&-P(y)\vdash_{A}\big(P(x)\vdash_{A}z\big) 
			+P(y)\vdash_{A}\big(P(z)\dashv_{A}x\big)
			+\Big(P\big(P(x)\big)\vdash_{A}z\Big)\dashv_{A}y\\
			&&
			+\big(P(x)\circ_{A}z\big)\dashv_{A}y +\big(x\circ_{A}P(z)\big)\dashv_{A}y
			+\big(P(x)\vdash_{A}P(z)\big)\dashv_{A}y\\
			&&
			-\big(P(z)\dashv_{A}P(x)\big)\dashv_{A}y
			-\Big( P\big(P(z)\big)\dashv_{A}x\Big)\dashv_{A}y\\
			&\overset{\eqref{eq:Leibniz},\eqref{eq:GD dia1}-\eqref{eq:GD dia3}}{=}&
			P(x)\vdash_{A}\big(P(y)\vdash_{A}z\big)
			-P(x)\vdash_{A}\big(P(z)\dashv_{A}y\big)
			-\Big(P\big(P(y)\big)\vdash_{A}z\Big)\dashv_{A}x\\
			&&
			-\big(P(y)\vdash_{A}P(z)\big)\dashv_{A}x +\big(P(z)\dashv_{A}P(y)\big)\dashv_{A}x
			+\Big( P\big(P(z)\big)\dashv_{A}y\Big)\dashv_{A}x\\
			&&
			-\big(P(x)\vdash_{A}P(y)\big)\vdash_{A}z 
			-\Big(P\big(P(x)\big)\vdash_{A}y\Big)\vdash_{A}z +\big(P(y)\dashv_{A}P(x)\big)\vdash_{A}z\\
			&&
			+\Big(P\big(P(y)\big)\dashv_{A}x\Big)\vdash_{A}z 
			+P(z)\dashv_{A}\big(P(x)\vdash_{A}y\big)
			-P(z)\dashv_{A}\big(P(y)\dashv_{A}x\big)\\
			&& -P(y)\vdash_{A}\big(P(x)\vdash_{A}z\big) 
			+P(y)\vdash_{A}\big(P(z)\dashv_{A}x\big)
			+\Big(P\big(P(x)\big)\vdash_{A}z\Big)\dashv_{A}y\\
			&&
			+\big(P(x)\vdash_{A}P(z)\big)\dashv_{A}y -\big(P(z)\dashv_{A}P(x)\big)\dashv_{A}y
			-\Big( P\big(P(z)\big)\dashv_{A}x\Big)\dashv_{A}y\\
			&\overset{\eqref{eq:nd4},\eqref{eq:nd5}}{=}&
			P(x)\vdash_{A}\big(P(y)\vdash_{A}z\big)
			-P(x)\vdash_{A}\big(P(z)\dashv_{A}y\big)
			-\big(P(y)\vdash_{A}P(z)\big)\dashv_{A}x\\
			&&
			+\big(P(z)\dashv_{A}P(y)\big)\dashv_{A}x -\big(P(x)\vdash_{A}P(y)\big)\vdash_{A}z
			+\big(P(y)\dashv_{A}P(x)\big)\vdash_{A}z\\
			&&
			+P(z)\dashv_{A}\big(P(x)\vdash_{A}y\big)
			-P(z)\dashv_{A}\big(P(y)\dashv_{A}x\big) -P(y)\vdash_{A}\big(P(x)\vdash_{A}z\big)\\
			&&
			+P(y)\vdash_{A}\big(P(z)\dashv_{A}x\big)
			+\big(P(x)\vdash_{A}P(z)\big)\dashv_{A}y
			-\big(P(z)\dashv_{A}P(x)\big)\dashv_{A}y\\
			&\overset{\eqref{eq:nd1}-\eqref{eq:nd3}}{=}&0.
	\end{eqnarray*}}Hence $(A,\cdot_{A})$ is a Leibniz algebra.
\end{proof}

\begin{pro}\label{pro:tensor product}
Let $A$ be a vector space with multiplications $\circ_{A},\vdash_{A},\dashv_{A}:A\otimes A\rightarrow A$. Set $\widehat{A}=A\otimes \mathbb {K}[t,t^{-1}]$.
Define a multiplication $\cdot_{\widehat{A}}:\widehat{A}\otimes \widehat{A}\rightarrow \widehat{A}$ by
\begin{eqnarray}\label{eq:1839}
(x\otimes t^{i})\cdot_{\widehat{A}}(y\otimes t^{j})=x\circ_{A}y\otimes t^{i+j}+(ix\vdash_{A}y-jy\dashv_{A}x)\otimes t^{i+j-1},\;\forall x,y\in A,i,j\in \mathbb {Z}.
\end{eqnarray}
Then $(\widehat{A},\cdot_{\widehat{A}})$ is a Leibniz algebra if and only if $(A,\circ_{A},\vdash_{A},\dashv_{A})$ is a GD dialgebra.
\end{pro}

\begin{proof}
Let $x,y,z\in A,i,j,k\in \mathbb {Z}$. Then we have
\begin{align*}
(x\otimes t^{i})\cdot_{\widehat{A}}\big((y\otimes t^{j})\cdot_{\widehat{A}}(z\otimes t^{k})\big)
&=(x\otimes t^{i})\cdot_{\widehat{A}}\big(y\circ_{A}z\otimes t^{j+k}+(jy\vdash_{A}z-kz\dashv_{A}y)\otimes t^{j+k-1}\big)\\
%&=&(x\otimes t^{i})\circ_{\widehat{A}}\big((jy\triangleright_{A} z)\otimes t^{j+k-1}\big)+(x\otimes t^{i})\circ_{\widehat{A}}\big((ky\triangleleft_{A}z)\otimes t^{j+k-1}\big)\\
&=x\circ_{A}(y\circ_{A}z)\otimes t^{i+j+k}+\big(jx\circ_{A}(y\vdash_{A}z)
-kx\circ_{A}(z\dashv_{A}y)\\
&\ \
+ix\vdash_{A}(y\circ_{A}z)
-(j+k)(y\circ_{A}z)\dashv_{A}x\big)\otimes t^{i+j+k-1}\\
&\ \
+\big(ijx\vdash_{A}(y\vdash_{A} z)-j(j+k-1)(y\vdash_{A}z)\dashv_{A}x\\
&\ \
-ikx\vdash_{A}(z\dashv_{A} y)+k(j+k-1)(z\dashv_{A}y)\dashv_{A}x\big)\otimes t^{i+j+k-2},\\
\big((x\otimes t^{i})\cdot_{\widehat{A}}(y\otimes t^{j})\big)\cdot_{\widehat{A}}(z\otimes t^{k})
&=\big(x\circ_{A}y\otimes t^{i+j}+(ix\vdash_{A} y-jy\dashv_{A}x)\otimes t^{i+j-1}\big)\cdot_{\widehat{A}}(z\otimes t^{k})\\
%&=&\big((ix\triangleright_{A} y)\otimes t^{i+j-1}\big)\circ_{\widehat{A}}(z\otimes t^{k})+\big((jx\triangleleft_{A}y)\otimes t^{i+j-1}\big)\circ_{\widehat{A}}(z\otimes t^{k})\\
&=(x\circ_{A}y)\circ_{A}z\otimes t^{i+j+k}
+\big(i(x\vdash_{A}y)\circ_{A}z-j(y\dashv_{A}x)\circ_{A}z\\
&\ \
+(i+j)(x\circ_{A}y)\vdash_{A}z-kz\dashv_{A}(x\circ_{A}y)\big)
\otimes t^{i+j+k-1}\\
&\ \
+\big(i(i+j-1)(x\vdash_{A}y)\vdash_{A} z-ik z\dashv_{A}(x\vdash_{A}y)\\
&\ \
-j(i+j-1)(y\dashv_{A}x)\vdash_{A}z+jkz\dashv_{A}(y\dashv_{A} x)\big)\otimes t^{i+j+k-2},\\
(y\otimes t^{j})\cdot_{\widehat{A}}\big((x\otimes t^{i})\cdot_{\widehat{A}}(z\otimes t^{k})\big)
%&=&\{y\otimes t^{j},(ix\triangleright_{A} z+kx\triangleleft_{A}z)\otimes t^{i+k-1}\}_{\widehat{A}}\\
%&=&\{y\otimes t^{j},(ix\triangleright_{A} z)\otimes t^{i+k-1}\}_{\widehat{A}}+\{y\otimes t^{j},(kx\triangleleft_{A}z)\otimes t^{i+k-1}\}_{\widehat{A}}\\
&=y\circ_{A}(x\circ_{A}z)\otimes t^{i+j+k}+\big(iy\circ_{A}(x\vdash_{A}z)
-ky\circ_{A}(z\dashv_{A}x)\\
&\ \
+jy\vdash_{A}(x\circ_{A}z)
-(i+k)(x\circ_{A}z)\dashv_{A}y\big)\otimes t^{i+j+k-1}\\
&\ \
+\big(ijy\vdash_{A}(x\vdash_{A} z)-i(i+k-1)(x\vdash_{A}z)\dashv_{A}y\\
&\ \
-jky\vdash_{A}(z\dashv_{A} x)+k(i+k-1)(z\dashv_{A}x)\dashv_{A}y\big)\otimes t^{i+j+k-2}.
\end{align*}
Since $(\widehat{A},\cdot_{\widehat{A}})$ is a Leibniz algebra, we obtain the following equations:
\begin{align}
&x\circ_{A}(y\circ_{A}z)=(x\circ_{A}y)\circ_{A}z+y\circ_{A}(x\circ_{A}z),\label{eq:tensor1}\\
&jx\circ_{A}(y\vdash_{A}z)-kx\circ_{A}(z\dashv_{A}y)
+ix\vdash_{A}(y\circ_{A}z)-(j+k)(y\circ_{A}z)\dashv_{A}x\nonumber\\
&=i(x\vdash_{A}y)\circ_{A}z-j(y\dashv_{A}x)\circ_{A}z
+(i+j)(x\circ_{A}y)\vdash_{A}z-kz\dashv_{A}(x\circ_{A}y)\nonumber\\
&\ \
+iy\circ_{A}(x\vdash_{A}z)-ky\circ_{A}(z\dashv_{A}x)
+jy\vdash_{A}(x\circ_{A}z)-(i+k)(x\circ_{A}z)\dashv_{A}y
,\label{eq:tensor2}\\
&ijx\vdash_{A}(y\vdash_{A} z)-j(j+k-1)(y\vdash_{A} z)\dashv_{A}x-ikx\vdash_{A}(z\dashv_{A} y)+k(j+k-1)(z\dashv_{A}y)\dashv_{A}x\nonumber\\
&=i(i+j-1)(x\vdash_{A}y)\vdash_{A}z-ikz\dashv_{A}(x\vdash_{A}y)
-j(i+j-1)(y\dashv_{A}x)\vdash_{A}z+jkz\dashv_{A}(y\dashv_{A} x)\nonumber\\
&\ \
+ijy\vdash_{A}(x\vdash_{A} z)-i(i+k-1)(x\vdash_{A}z)\dashv_{A}y
-jky\vdash_{A}(z\dashv_{A} x)+k(i+k-1)(z\dashv_{A} x)\dashv_{A}y.\label{eq:tensor3}
\end{align}
Hence we have the following conclusions.
\begin{enumerate}
\item
By \eqref{eq:tensor1}, $(A,\circ_{A})$ is a Leibniz algebra.
\item
Taking $i=j=1,k=0$, \eqref{eq:tensor3} reads \eqref{eq:nd1}.
\item
Taking $i=k=1,j=0$, \eqref{eq:tensor3} reads \eqref{eq:nd2}.
\item
Taking $j=k=1,i=0$, \eqref{eq:tensor3} reads
\begin{equation}\label{eq:L2}
(y\vdash_{A}z)\dashv_{A}x=y\vdash_{A}(z\dashv_{A}x)
+(z\dashv_{A}y)\dashv_{A}x-z\dashv_{A}(y\dashv_{A}x).
\end{equation}
\item
Taking $i=-1,j=k=0$, \eqref{eq:tensor2} reads \eqref{eq:GD dia1}; \eqref{eq:tensor3} reads the second equality in \eqref{eq:nd5}.
\item
Taking $j=-1,i=k=0$, \eqref{eq:tensor2} reads \eqref{eq:GD dia2}; \eqref{eq:tensor3} reads the first equality in \eqref{eq:nd5}.
\item
Taking $k=-1,i=j=0$, \eqref{eq:tensor2} reads \eqref{eq:GD dia3}; \eqref{eq:tensor3} reads \eqref{eq:nd4}.
\end{enumerate}
It is clear that the difference between \eqref{eq:nd2} and \eqref{eq:L2} gives \eqref{eq:nd3}. Thus $(A,\vdash_{A},\dashv_{A})$ is a Novikov dialgebra. Therefore $(A,\circ_{A},\vdash_{A},\dashv_{A})$ is a GD dialgebra.

 Conversely, suppose that $(A,\circ_{A},\vdash_{A},\dashv_{A})$ is a GD dialgebra.
 Then there is a natural extended GD dialgebra structure
 on $\widehat{A}$ given by
 \begin{eqnarray*}
 &&(x\otimes t^{i})\circ_{\widehat{A}}(y\otimes t^{j})=x\circ_{A}y\otimes t^{i+j},\;
 (x\otimes t^{i})\vdash_{\widehat{A}}(y\otimes t^{j})=x\vdash_{A}y\otimes t^{i+j},\\
 &&(x\otimes t^{i})\dashv_{\widehat{A}}(y\otimes t^{j})=x\dashv_{A}y\otimes t^{i+j},\;\forall x,y\in A,i,j\in \mathbb {Z}.
 \end{eqnarray*}
Moreover, the natural derivation $P$ on $\mathbb{K}[t,t^{-1}]$ given by $P(t^{i})=it^{i-1}$ can be extended to $\widehat{A}$, that is, we have
\begin{eqnarray*}
P(x\otimes t^{i})=x\otimes P(t^{i})=ix\otimes t^{i-1},\;\forall x\in A,i\in \mathbb {Z}.
\end{eqnarray*}
Hence we can rewrite \eqref{eq:1839} by
\begin{eqnarray*}
	(x\otimes t^{i})\cdot_{\widehat{A}} (y\otimes t^{j})=(x\otimes t^{i})\circ_{\widehat{A}}(y\otimes t^{j})+P(x\otimes t^{i})\vdash_{\widehat{A}}(y\otimes t^{j})
	-P(y\otimes t^{j})\dashv_{\widehat{A}}(x\otimes t^{i}).
\end{eqnarray*}
Thus by Proposition \ref{pro:LN and Leib}, $(\widehat{A},\cdot_{\widehat{A}})$ is a Leibniz algebra.
\end{proof}

\delete{
Combining Proposition \ref{pro:2-alg} and Proposition \ref{pro:tensor product} together, we have the following result.

\begin{cor}
Let $A$ be a vector space with multiplications $\succ_{A},\prec_{A}:A\otimes A\rightarrow A$. Set $\widehat{A}=A\otimes \mathbb {K}[t,t^{-1}]$.
Define a multiplication $\circ_{\widehat{A}}:\widehat{A}\otimes \widehat{A}\rightarrow \widehat{A}$ by
\begin{eqnarray}
(x\otimes t^{i})\circ_{\widehat{A}}(y\otimes t^{j})=\big((i+2j)x\succ_{A}y+(2i+j)x\prec_{A}y\big)\otimes t^{i+j-1},\;\forall x,y\in A,i,j\in \mathbb {Z}.
\end{eqnarray}
Then $(\widehat{A},\circ_{\widehat{A}})$ is a Leibniz algebra if and only if $(A,\succ_{A},\prec_{A})$ is an admissible Novikov dialgebra.
\end{cor}}

\delete{
\begin{proof}
It is clear that $(A\oplus A^*,\bullet_{1_{d}})$ and $(A\oplus A^*,\bullet_{2_{d}})$ are Leibniz algebras.
Let $x,y,z\in A,a^{*},b^{*},c^{*}\in A^{*}$, we have
\begin{eqnarray*}
&&(x+a^*)\bullet_{2_{d}}\big((y+b^*)\bullet_{1_{d}}(z+c^*)\big)
+(x+a^*)\bullet_{1_{d}}\big((y+b^*)\bullet_{2_{d}}(z+c^*)\big)\\
&&
-\big((x+a^*)\bullet_{1_{d}}(y+b^*)\big)\bullet_{2_{d}}(z+c^*)
-\big((x+a^*)\bullet_{2_{d}}(y+b^*)\big)\bullet_{1_{d}}(z+c^*)\\
&&
-(y+b^*)\bullet_{2_{d}}\big((x+a^*)\bullet_{1_{d}}(z+c^*)\big)
-(y+b^*)\bullet_{1_{d}}\big((x+a^*)\bullet_{2_{d}}(z+c^*)\big)\\
%&=&(x+a^*)\bullet_{2_{d}}\big(y\bullet_{A}z +\mathcal{L}^{*}_{\triangleright_{A}}(y)c^{*} -(\mathcal{L}^{*}_{\triangleright_{A}} +\mathcal{R}^{*}_{\triangleleft_{A}})(z)b^{*}\big)\\
%&&+(x+a^*)\bullet_{1_{d}}\big(y\bullet_{A}z +\mathcal{L}^{*}_{\bullet_{A}}(y)c^{*} -(\mathcal{L}^{*}_{\bullet_{A}} +\mathcal{R}^{*}_{\bullet_{A}})(z)b^{*}\big)\\
%&&-\big(x\bullet_{A}y+\mathcal{L}^{*}_{\triangleright_{A}}(x)b^{*} -(\mathcal{L}^{*}_{\triangleright_{A}} +\mathcal{R}^{*}_{\triangleleft_{A}})(y)a^{*}\big)\bullet_{2_{d}}(z+c^*)\\
%&&-\big(x\bullet_{A}y+\mathcal{L}^{*}_{\bullet_{A}}(x)b^{*} -(\mathcal{L}^{*}_{\bullet_{A}} +\mathcal{R}^{*}_{\bullet_{A}})(y)a^{*}\big)\bullet_{1_{d}}(z+c^*)\\
%&&-(y+b^*)\bullet_{2_{d}}\big(x\bullet_{A}z +\mathcal{L}^{*}_{\triangleright_{A}}(x)c^{*} -(\mathcal{L}^{*}_{\triangleright_{A}} +\mathcal{R}^{*}_{\triangleleft_{A}})(z)a^{*}\big)\\
%&&-(y+b^*)\bullet_{1_{d}}\big(x\bullet_{A}z +\mathcal{L}^{*}_{\bullet_{A}}(x)c^{*} -(\mathcal{L}^{*}_{\bullet_{A}} +\mathcal{R}^{*}_{\bullet_{A}})(z)a^{*}\big)\\
&=&2x\bullet_{A}(y\bullet_{A}z)
+\mathcal{L}^{*}_{\bullet_{A}}(x)\mathcal{L}^{*}_{\triangleright_{A}}(y)c^{*}
-\mathcal{L}^{*}_{\bullet_{A}}(x)(\mathcal{L}^{*}_{\triangleright_{A}}
+\mathcal{R}^{*}_{\triangleleft_{A}})(z)b^{*}
-(\mathcal{L}^{*}_{\bullet_{A}}
+\mathcal{R}^{*}_{\bullet_{A}})(y\bullet_{A}z)a^{*}\\
&&
+\mathcal{L}^{*}_{\triangleright_{A}}(x)\mathcal{L}^{*}_{\bullet_{A}}(y)c^{*}
-\mathcal{L}^{*}_{\triangleright_{A}}(x)(\mathcal{L}^{*}_{\bullet_{A}}
+\mathcal{R}^{*}_{\bullet_{A}})(z)b^{*}-(\mathcal{L}^{*}_{\triangleright_{A}}
+\mathcal{R}^{*}_{\triangleleft_{A}})(y\bullet_{A}z)a^{*}
-2(x\bullet_{A}y)\bullet_{A}z\\
&&
-\mathcal{L}^{*}_{\bullet_{A}}(x\bullet_{A}y)c^{*}
+(\mathcal{L}^{*}_{\bullet_{A}}
+\mathcal{R}^{*}_{\bullet_{A}})(z)\mathcal{L}^{*}_{\triangleright_{A}}(x)b^{*}
-(\mathcal{L}^{*}_{\bullet_{A}}
+\mathcal{R}^{*}_{\bullet_{A}})(z)(\mathcal{L}^{*}_{\triangleright_{A}}
+\mathcal{R}^{*}_{\triangleleft_{A}})(y)a^{*}\\
&&
-\mathcal{L}^{*}_{\triangleright_{A}}(x\bullet_{A}y)c^{*}
+(\mathcal{L}^{*}_{\triangleright_{A}}
+\mathcal{R}^{*}_{\triangleleft_{A}})(z)\mathcal{L}^{*}_{\bullet_{A}}(x)b^{*}
-(\mathcal{L}^{*}_{\triangleright_{A}}
+\mathcal{R}^{*}_{\triangleleft_{A}})(z)(\mathcal{L}^{*}_{\bullet_{A}}
+\mathcal{R}^{*}_{\bullet_{A}})(y)a^{*}\\
&&
-2y\bullet_{A}(x\bullet_{A}z)
-\mathcal{L}^{*}_{\bullet_{A}}(y)\mathcal{L}^{*}_{\triangleright_{A}}(x)c^{*}
+\mathcal{L}^{*}_{\bullet_{A}}(y)(\mathcal{L}^{*}_{\triangleright_{A}}
+\mathcal{R}^{*}_{\triangleleft_{A}})(z)a^{*}
+(\mathcal{L}^{*}_{\bullet_{A}}
+\mathcal{R}^{*}_{\bullet_{A}})(x\bullet_{A}z)b^{*}\\
&&
-\mathcal{L}^{*}_{\triangleright_{A}}(y)\mathcal{L}^{*}_{\bullet_{A}}(x)c^{*}
+\mathcal{L}^{*}_{\triangleright_{A}}(y)(\mathcal{L}^{*}_{\bullet_{A}}
+\mathcal{R}^{*}_{\bullet_{A}})(z)a^{*}
+(\mathcal{L}^{*}_{\triangleright_{A}}
+\mathcal{R}^{*}_{\triangleleft_{A}})(x\bullet_{A}z)b^{*}\\
&\overset{\eqref{eq:Leibniz}}{=}&
\mathcal{L}^{*}_{\bullet_{A}}(x)\mathcal{L}^{*}_{\triangleright_{A}}(y)c^{*}
-\mathcal{L}^{*}_{\bullet_{A}}(x)(\mathcal{L}^{*}_{\triangleright_{A}}
+\mathcal{R}^{*}_{\triangleleft_{A}})(z)b^{*}
-(\mathcal{L}^{*}_{\bullet_{A}}
+\mathcal{R}^{*}_{\bullet_{A}})(y\bullet_{A}z)a^{*}
+\mathcal{L}^{*}_{\triangleright_{A}}(x)\mathcal{L}^{*}_{\bullet_{A}}(y)c^{*}\\
&&
-\mathcal{L}^{*}_{\triangleright_{A}}(x)(\mathcal{L}^{*}_{\bullet_{A}}
+\mathcal{R}^{*}_{\bullet_{A}})(z)b^{*}-(\mathcal{L}^{*}_{\triangleright_{A}}
+\mathcal{R}^{*}_{\triangleleft_{A}})(y\bullet_{A}z)a^{*}
-\mathcal{L}^{*}_{\bullet_{A}}(x\bullet_{A}y)c^{*}\\
&&
+(\mathcal{L}^{*}_{\bullet_{A}}
+\mathcal{R}^{*}_{\bullet_{A}})(z)\mathcal{L}^{*}_{\triangleright_{A}}(x)b^{*}
-(\mathcal{L}^{*}_{\bullet_{A}}
+\mathcal{R}^{*}_{\bullet_{A}})(z)(\mathcal{L}^{*}_{\triangleright_{A}}
+\mathcal{R}^{*}_{\triangleleft_{A}})(y)a^{*}
-\mathcal{L}^{*}_{\triangleright_{A}}(x\bullet_{A}y)c^{*}\\
&&
+(\mathcal{L}^{*}_{\triangleright_{A}}
+\mathcal{R}^{*}_{\triangleleft_{A}})(z)\mathcal{L}^{*}_{\bullet_{A}}(x)b^{*}
-(\mathcal{L}^{*}_{\triangleright_{A}}
+\mathcal{R}^{*}_{\triangleleft_{A}})(z)(\mathcal{L}^{*}_{\bullet_{A}}
+\mathcal{R}^{*}_{\bullet_{A}})(y)a^{*}
-\mathcal{L}^{*}_{\bullet_{A}}(y)\mathcal{L}^{*}_{\triangleright_{A}}(x)c^{*}\\
&&
+\mathcal{L}^{*}_{\bullet_{A}}(y)(\mathcal{L}^{*}_{\triangleright_{A}}
+\mathcal{R}^{*}_{\triangleleft_{A}})(z)a^{*}
+(\mathcal{L}^{*}_{\bullet_{A}}
+\mathcal{R}^{*}_{\bullet_{A}})(x\bullet_{A}z)b^{*}
-\mathcal{L}^{*}_{\triangleright_{A}}(y)\mathcal{L}^{*}_{\bullet_{A}}(x)c^{*}\\
&&
+\mathcal{L}^{*}_{\triangleright_{A}}(y)(\mathcal{L}^{*}_{\bullet_{A}}
+\mathcal{R}^{*}_{\bullet_{A}})(z)a^{*}
+(\mathcal{L}^{*}_{\triangleright_{A}}
+\mathcal{R}^{*}_{\triangleleft_{A}})(x\bullet_{A}z)b^{*}.
\end{eqnarray*}
Suppose that $(A\oplus A^*,\bullet_{1_{d}},\bullet_{2_{d}})$
is a compatible Leibniz algebra. Then we have
\begin{eqnarray*}
0&=&\langle \mathcal{L}^{*}_{\bullet_{A}}(x)\mathcal{L}^{*}_{\triangleright_{A}}(y)c^{*}
+\mathcal{L}^{*}_{\triangleright_{A}}(x)\mathcal{L}^{*}_{\bullet_{A}}(y)c^{*}
-\mathcal{L}^{*}_{\bullet_{A}}(x\bullet_{A}y)c^{*}
-\mathcal{L}^{*}_{\triangleright_{A}}(x\bullet_{A}y)c^{*}\\
&&
-\mathcal{L}^{*}_{\bullet_{A}}(y)\mathcal{L}^{*}_{\triangleright_{A}}(x)c^{*}
-\mathcal{L}^{*}_{\triangleright_{A}}(y)\mathcal{L}^{*}_{\bullet_{A}}(x)c^{*},z\rangle\\
&=&\langle c^{*},y\triangleright_{A}(x\bullet_{A}z)+y\bullet_{A}(x\triangleright_{A}z)
+(x\bullet_{A}y)\bullet_{A}z+(x\bullet_{A}y)\triangleright_{A}z
-x\triangleright_{A}(y\bullet_{A}z)-x\bullet_{A}(y\triangleright_{A}z)\rangle,\\
0&=&\langle -\mathcal{L}^{*}_{\bullet_{A}}(x)(\mathcal{L}^{*}_{\triangleright_{A}}
+\mathcal{R}^{*}_{\triangleleft_{A}})(z)b^{*}
-\mathcal{L}^{*}_{\triangleright_{A}}(x)(\mathcal{L}^{*}_{\bullet_{A}}
+\mathcal{R}^{*}_{\bullet_{A}})(z)b^{*}
+(\mathcal{L}^{*}_{\bullet_{A}}
+\mathcal{R}^{*}_{\bullet_{A}})(z)\mathcal{L}^{*}_{\triangleright_{A}}(x)b^{*}\\
&&
+(\mathcal{L}^{*}_{\triangleright_{A}}
+\mathcal{R}^{*}_{\triangleleft_{A}})(z)\mathcal{L}^{*}_{\bullet_{A}}(x)b^{*}
+(\mathcal{L}^{*}_{\bullet_{A}}
+\mathcal{R}^{*}_{\bullet_{A}})(x\bullet_{A}z)b^{*}
+(\mathcal{L}^{*}_{\triangleright_{A}}
+\mathcal{R}^{*}_{\triangleleft_{A}})(x\bullet_{A}z)b^{*},y\rangle\\
&=&\langle b^{*},-z\triangleright_{A}(x\bullet_{A}y)-(x\bullet_{A}y)\triangleleft_{A}z
-z\bullet_{A}(x\triangleright_{A}y)-(x\triangleright_{A}y)\bullet_{A}z
+x\triangleright_{A}(z\bullet_{A}y)\\
&&
+x\triangleright_{A}(y\bullet_{A}z)
+x\bullet_{A}(z\triangleright_{A}y)+x\bullet_{A}(y\triangleleft_{A}z)
-(x\bullet_{A}z)\bullet_{A}y-(x\bullet_{A}z)\triangleright_{A}y
-y\bullet_{A}(x\bullet_{A}z)-y\triangleleft_{A}(x\bullet_{A}z)\rangle,\\
0&=&\langle-(\mathcal{L}^{*}_{\bullet_{A}}
+\mathcal{R}^{*}_{\bullet_{A}})(y\bullet_{A}z)a^{*}
-(\mathcal{L}^{*}_{\triangleright_{A}}
+\mathcal{R}^{*}_{\triangleleft_{A}})(y\bullet_{A}z)a^{*}
-(\mathcal{L}^{*}_{\bullet_{A}}
+\mathcal{R}^{*}_{\bullet_{A}})(z)(\mathcal{L}^{*}_{\triangleright_{A}}
+\mathcal{R}^{*}_{\triangleleft_{A}})(y)a^{*}\\
&&
-(\mathcal{L}^{*}_{\triangleright_{A}}
+\mathcal{R}^{*}_{\triangleleft_{A}})(z)(\mathcal{L}^{*}_{\bullet_{A}}
+\mathcal{R}^{*}_{\bullet_{A}})(y)a^{*}
+\mathcal{L}^{*}_{\bullet_{A}}(y)(\mathcal{L}^{*}_{\triangleright_{A}}
+\mathcal{R}^{*}_{\triangleleft_{A}})(z)a^{*}
+\mathcal{L}^{*}_{\triangleright_{A}}(y)(\mathcal{L}^{*}_{\bullet_{A}}
+\mathcal{R}^{*}_{\bullet_{A}})(z)a^{*},x\rangle\\
&=&\langle a^{*},(y\bullet_{A}z)\bullet_{A}x+
x\bullet_{A}(y\bullet_{A}z)+(y\bullet_{A}z)\triangleright_{A}x
+x\triangleleft_{A}(y\bullet_{A}z)
-y\triangleright_{A}(z\bullet_{A}x)-y\triangleright_{A}(x\bullet_{A}z)
-(z\bullet_{A}x)\triangleleft_{A}y\\
&&
-(x\bullet_{A}z)\triangleleft_{A}y
-y\bullet_{A}(z\triangleright_{A}x)-y\bullet_{A}(x\triangleleft_{A}z)
-(z\triangleright_{A}x)\bullet_{A}y-(x\triangleleft_{A}z)\bullet_{A}y\\
&&
+z\triangleright_{A}(y\bullet_{A}x)+(y\bullet_{A}x)\triangleleft_{A}z
+z\bullet_{A}(y\triangleright_{A}x)+(y\triangleright_{A}x)\bullet_{A}z\rangle.
\end{eqnarray*}
Thus we obtain the following equations:
\begin{eqnarray}
0&=&y\triangleright_{A}(x\bullet_{A}z)+y\bullet_{A}(x\triangleright_{A}z)
+(x\bullet_{A}y)\bullet_{A}z+(x\bullet_{A}y)\triangleright_{A}z
-x\triangleright_{A}(y\bullet_{A}z)-x\bullet_{A}(y\triangleright_{A}z),\label{eq:1}\\
0&=&-z\triangleright_{A}(x\bullet_{A}y)-(x\bullet_{A}y)\triangleleft_{A}z
-z\bullet_{A}(x\triangleright_{A}y)-(x\triangleright_{A}y)\bullet_{A}z
+x\triangleright_{A}(z\bullet_{A}y)\nonumber\\
&&
+x\triangleright_{A}(y\bullet_{A}z)
+x\bullet_{A}(z\triangleright_{A}y)+x\bullet_{A}(y\triangleleft_{A}z)
-(x\bullet_{A}z)\bullet_{A}y-(x\bullet_{A}z)\triangleright_{A}y
-y\bullet_{A}(x\bullet_{A}z)-y\triangleleft_{A}(x\bullet_{A}z)\nonumber\\
&\overset{\eqref{eq:1}}{=}&
-(x\bullet_{A}y)\triangleleft_{A}z-(x\triangleright_{A}y)\bullet_{A}z
+x\triangleright_{A}(y\bullet_{A}z)+x\bullet_{A}(y\triangleleft_{A}z)
-y\bullet_{A}(x\bullet_{A}z)-y\triangleleft_{A}(x\bullet_{A}z),\label{eq:2}\\
0&=&(y\bullet_{A}z)\bullet_{A}x+
x\bullet_{A}(y\bullet_{A}z)+(y\bullet_{A}z)\triangleright_{A}x
+x\triangleleft_{A}(y\bullet_{A}z)
-y\triangleright_{A}(z\bullet_{A}x)-y\triangleright_{A}(x\bullet_{A}z)
-(z\bullet_{A}x)\triangleleft_{A}y\nonumber\\
&&
-(x\bullet_{A}z)\triangleleft_{A}y
-y\bullet_{A}(z\triangleright_{A}x)-y\bullet_{A}(x\triangleleft_{A}z)
-(z\triangleright_{A}x)\bullet_{A}y-(x\triangleleft_{A}z)\bullet_{A}y\nonumber\\
&&
+z\triangleright_{A}(y\bullet_{A}x)+(y\bullet_{A}x)\triangleleft_{A}z
+z\bullet_{A}(y\triangleright_{A}x)+(y\triangleright_{A}x)\bullet_{A}z\nonumber\\
&\overset{\eqref{eq:1},\eqref{eq:2}}{=}&
-(z\bullet_{A}x)\triangleleft_{A}y-(x\bullet_{A}z)\triangleleft_{A}y
-(z\triangleright_{A}x)\bullet_{A}y-(x\triangleleft_{A}z)\bullet_{A}y.\label{eq:3}
\end{eqnarray}
Since $(A,\triangleright_{A},\triangleleft_{A})$ is a pre-Leibniz algebra, we have
\begin{eqnarray*}
0&=&y\triangleright_{A}(x\bullet_{A}z)+y\bullet_{A}(x\triangleright_{A}z)
+(x\bullet_{A}y)\bullet_{A}z+(x\bullet_{A}y)\triangleright_{A}z
-x\triangleright_{A}(y\bullet_{A}z)-x\bullet_{A}(y\triangleright_{A}z)\\
&\overset{\eqref{eq:pre-L1}}{=}&
-x\triangleleft_{A}(y\triangleright_{A}z)
+y\triangleright_{A}(x\triangleleft_{A}z)
+y\bullet_{A}(x\triangleright_{A}z)+(x\bullet_{A}y)\bullet_{A}z
-x\triangleright_{A}(y\bullet_{A}z)\\
&\overset{\eqref{eq:Leibniz}}{=}&
x\triangleleft_{A}(y\triangleleft_{A}z)
-y\triangleleft_{A}(x\triangleleft_{A}z),\\
0&=&-(x\bullet_{A}y)\triangleleft_{A}z-(x\triangleright_{A}y)\bullet_{A}z
+x\triangleright_{A}(y\bullet_{A}z)+x\bullet_{A}(y\triangleleft_{A}z)
-y\bullet_{A}(x\bullet_{A}z)-y\triangleleft_{A}(x\bullet_{A}z)\\
&\overset{\eqref{eq:pre-L2}}{=}&
-(x\triangleleft_{A}y)\triangleleft_{A}z
-(x\triangleright_{A}y)\bullet_{A}z
+x\triangleright_{A}(y\triangleright_{A}z)
+x\bullet_{A}(y\triangleleft_{A}z)-y\bullet_{A}(x\bullet_{A}z)\\
&\overset{\eqref{eq:Leibniz}}{=}&
(x\triangleleft_{A}y)\triangleright_{A}z
-x\triangleleft_{A}(y\triangleright_{A}z),\\
0&=&-(z\bullet_{A}x)\triangleleft_{A}y-(x\bullet_{A}z)\triangleleft_{A}y
-(z\triangleright_{A}x)\bullet_{A}y-(x\triangleleft_{A}z)\bullet_{A}y\\
&\overset{\eqref{eq:pre-L3}}{=}&
-(z\triangleright_{A}x)\triangleright_{A}y
-(x\triangleleft_{A}z)\triangleright_{A}y.
\end{eqnarray*}
Hence \eqref{eq:LN1}-\eqref{eq:LN2} hold. Therefore $(A,\triangleright_{A},\triangleleft_{A})$ is a Novikov dialgebra. The converse side is obtained similarly.
\end{proof}}

\delete{
\begin{pro}
Let $A$ be a vector space with multiplications $\triangleright_{A},\triangleleft_{A}:A\otimes A\rightarrow A$. Set $\widehat{A}=A\otimes \mathbb {K}[t,t^{-1}]$.
Define a multiplication $\circ_{\widehat{A}}:\widehat{A}\otimes \widehat{A}\rightarrow \widehat{A}$ by
\begin{eqnarray}
(x\otimes t^{i})\circ_{\widehat{A}}(y\otimes t^{j})=(ix\triangleright_{A}y+jx\triangleleft_{A}y)\otimes t^{i+j-1},\;\forall x,y\in A,i,j\in \mathbb {Z}.
\end{eqnarray}
Then $(\widehat{A},\circ_{\widehat{A}})$ is a Leibniz algebra if and only if $(A,\triangleright_{A},\triangleleft_{A})$ is a Novikov dialgebra.
\end{pro}

\begin{proof}
Let $x,y,z\in A$. Then we have
\begin{align*}
(x\otimes t^{i})\circ_{\widehat{A}}\big((y\otimes t^{j})\circ_{\widehat{A}}(z\otimes t^{k})\big)
&=(x\otimes t^{i})\circ_{\widehat{A}}\big((jy\triangleright_{A} z+ky\triangleleft_{A}z)\otimes t^{j+k-1}\big)\\
%&=&(x\otimes t^{i})\circ_{\widehat{A}}\big((jy\triangleright_{A} z)\otimes t^{j+k-1}\big)+(x\otimes t^{i})\circ_{\widehat{A}}\big((ky\triangleleft_{A}z)\otimes t^{j+k-1}\big)\\
&=\big(ijx\triangleright_{A}(y\triangleright_{A} z)+j(j+k-1)x\triangleleft_{A}(y\triangleright_{A} z)\\
&\ \
+ikx\triangleright_{A}(y\triangleleft_{A} z)+k(j+k-1)x\triangleleft_{A}(y\triangleleft_{A} z)\big)\otimes t^{i+j+k-2},\\
\big((x\otimes t^{i})\circ_{\widehat{A}}(y\otimes t^{j})\big)\circ_{\widehat{A}}(z\otimes t^{k})
&=\big((ix\triangleright_{A} y+jx\triangleleft_{A}y)\otimes t^{i+j-1}\circ_{\widehat{A}}(z\otimes t^{k})\big)\\
%&=&\big((ix\triangleright_{A} y)\otimes t^{i+j-1}\big)\circ_{\widehat{A}}(z\otimes t^{k})+\big((jx\triangleleft_{A}y)\otimes t^{i+j-1}\big)\circ_{\widehat{A}}(z\otimes t^{k})\\
&=\big(i(i+j-1)(x\triangleright_{A} y)\triangleright_{A} z+ik(x\triangleright_{A} y)\triangleleft_{A} z\\
&\ \
+j(i+j-1)(x\triangleleft_{A} y)\triangleright_{A} z+jk(x\triangleleft_{A} y)\triangleleft_{A} z\big)\otimes t^{i+j+k-2},\\
(y\otimes t^{j})\circ_{\widehat{A}}\big((x\otimes t^{i})\circ_{\widehat{A}}(z\otimes t^{k})\big)
%&=&\{y\otimes t^{j},(ix\triangleright_{A} z+kx\triangleleft_{A}z)\otimes t^{i+k-1}\}_{\widehat{A}}\\
%&=&\{y\otimes t^{j},(ix\triangleright_{A} z)\otimes t^{i+k-1}\}_{\widehat{A}}+\{y\otimes t^{j},(kx\triangleleft_{A}z)\otimes t^{i+k-1}\}_{\widehat{A}}\\
&=\big(ijy\triangleright_{A}(x\triangleright_{A} z)+i(i+k-1)y\triangleleft_{A}(x\triangleright_{A} z)\\
&\ \
+jky\triangleright_{A}(x\triangleleft_{A} z)+k(i+k-1)y\triangleleft_{A}(x\triangleleft_{A} z)\big)\otimes t^{i+j+k-2}.
\end{align*}
Since $(\widehat{A},\circ_{\widehat{A}})$ is a Leibniz algebra, we obtain
\begin{eqnarray}
&&ijx\triangleright_{A}(y\triangleright_{A} z)+j(j+k-1)x\triangleleft_{A}(y\triangleright_{A} z)+ikx\triangleright_{A}(y\triangleleft_{A} z)+k(j+k-1)x\triangleleft_{A}(y\triangleleft_{A} z)\nonumber\\
&&=i(i+j-1)(x\triangleright_{A} y)\triangleright_{A} z+ik(x\triangleright_{A} y)\triangleleft_{A} z
+j(i+j-1)(x\triangleleft_{A} y)\triangleright_{A} z+jk(x\triangleleft_{A} y)\triangleleft_{A} z\nonumber\\
&&\ \
+ijy\triangleright_{A}(x\triangleright_{A} z)+i(i+k-1)y\triangleleft_{A}(x\triangleright_{A} z)
+jky\triangleright_{A}(x\triangleleft_{A} z)+k(i+k-1)y\triangleleft_{A}(x\triangleleft_{A} z).\label{eq:tensor}
\end{eqnarray}
Hence we have the following conclusions.
\begin{enumerate}
\item
Taking $i=j=1,k=0$, \eqref{eq:tensor} reads \eqref{eq:pre-L1}.
\item
Taking $i=k=1,j=0$, \eqref{eq:tensor} reads
\begin{equation}\label{eq:L1}
x\triangleright_{A}(y\triangleleft_{A} z)=(x\triangleright_{A} y)\triangleleft_{A} z+y\triangleleft_{A}(x\triangleright_{A} z)+y\triangleleft_{A}(x\triangleleft_{A} z).
\end{equation}
\item
Taking $j=k=1,i=0$, \eqref{eq:tensor} reads
\begin{equation}\label{eq:L2}
x\triangleleft_{A}(y\triangleright_{A} z)=(x\triangleleft_{A} y)\triangleleft_{A} z+y\triangleright_{A}(x\triangleleft_{A} z)-x\triangleleft_{A}(y\triangleleft_{A} z).
\end{equation}
\item
Taking $i=-1,j=k=0$ and $j=-1,i=k=0$ respectively,
\eqref{eq:tensor} reads \eqref{eq:LN2}.
\item
Taking $k=-1,i=j=0$, \eqref{eq:tensor} reads
\eqref{eq:LN1}.
\end{enumerate}
It is clear that the difference between \eqref{eq:L1} and \eqref{eq:L2} gives \eqref{eq:pre-L3}. Thus \eqref{eq:L1} and \eqref{eq:L2} hold if and only if \eqref{eq:pre-L2} and \eqref{eq:pre-L3} hold. Therefore, $(A,\triangleright_{A},\triangleleft_{A})$ is a Novikov dialgebra. Conversely, a similar argument shows that, if $(A,\triangleright_{A},\triangleleft_{A})$ is a Novikov dialgebra, then $(\widehat{A},\circ_{\widehat{A}})$ is a Leibniz algebra.
\end{proof}

Combining Proposition \ref{pro:2-alg} and Proposition \ref{pro:tensor product} together, we have the following result.

\begin{cor}
Let $A$ be a vector space with multiplications $\succ_{A},\prec_{A}:A\otimes A\rightarrow A$. Set $\widehat{A}=A\otimes \mathbb {K}[t,t^{-1}]$.
Define a multiplication $\circ_{\widehat{A}}:\widehat{A}\otimes \widehat{A}\rightarrow \widehat{A}$ by
\begin{eqnarray}
(x\otimes t^{i})\circ_{\widehat{A}}(y\otimes t^{j})=\big((i+2j)x\succ_{A}y+(2i+j)x\prec_{A}y\big)\otimes t^{i+j-1},\;\forall x,y\in A,i,j\in \mathbb {Z}.
\end{eqnarray}
Then $(\widehat{A},\circ_{\widehat{A}})$ is a Leibniz algebra if and only if $(A,\succ_{A},\prec_{A})$ is an admissible Novikov dialgebra.
\end{cor}

\begin{pro}
Let $(A,\triangleright_{A},\triangleleft_{A})$ be a Novikov dialgebra with a derivation $P$. Define a multiplication
$\circ_{A}:A\otimes A\rightarrow A$ by
\begin{equation}
x\circ_{A}y=P(x)\triangleright_{A}y+x\triangleleft_{A}P(y),\;\forall x,y\in A.
\end{equation}
Then $(A,\circ_{A})$ is a Leibniz algebra.
\end{pro}

\begin{proof}
Let $x,y,z\in A$, then we have
\begin{align*}
x\circ_{A}(y\circ_{A}z)
&=x\circ_{A}\big(P(y)\triangleright_{A}z+y\triangleleft_{A}P(z)\big)\\
&=P(x)\triangleright_{A}\big(P(y)\triangleright_{A}z
+y\triangleleft_{A}P(z)\big)
+x\triangleleft_{A}P\big(P(y)\triangleright_{A}z
+y\triangleleft_{A}P(z)\big)\\
&=P(x)\triangleright_{A}\big(P(y)\triangleright_{A}z\big)+P(x)\triangleright_{A}\big(y\triangleleft_{A}P(z)\big)
+x\triangleleft_{A}\Big(P\big(P(y)\big)\triangleright_{A}z\Big)\\
&\ \
+x\triangleleft_{A}\big(P(y)\triangleright_{A}P(z)\big)
+x\triangleleft_{A}\big(P(y)\triangleleft_{A}P(z)\big)+x\triangleleft_{A}\Big(y\triangleleft_{A}P\big(P(z)\big)\Big),\\
(x\circ_{A}y)\circ_{A}z
&=\big(P(x)\triangleright_{A}y+x\triangleleft_{A}P(y)\big)\circ_{A}z\\
&=P\big(P(x)\triangleright_{A}y+x\triangleleft_{A}P(y)\big)\triangleright_{A}z
+\big(P(x)\triangleright_{A}y+x\triangleleft_{A}P(y)\big)\triangleleft_{A}P(z)\\
&=\big(P(x)\triangleright_{A}P(y)\big)\triangleright_{A}z+\Big(P\big(P(x)\big)\triangleright_{A}y\Big)\triangleright_{A}z
+\big(P(x)\triangleleft_{A}P(y)\big)\triangleright_{A}z\\
&\ \
+\Big(x\triangleleft_{A}P\big(P(y)\big)\Big)\triangleright_{A}z
+\big(P(x)\triangleright_{A}y\big)\triangleleft_{A}P(z)+\big(x\triangleleft_{A}P(y)\big)\triangleleft_{A}P(z),\\
y\circ_{A}(x\circ_{A}z)&=
P(y)\triangleright_{A}\big(P(x)\triangleright_{A}z\big)
+P(y)\triangleright_{A}\big(x\triangleleft_{A}P(z)\big)
+y\triangleleft_{A}\Big(P\big(P(x)\big)\triangleright_{A}z\Big)\\
&\ \
+y\triangleleft_{A}\big(P(x)\triangleright_{A}P(z)\big)
+y\triangleleft_{A}\big(P(x)\triangleleft_{A}P(z)\big)
+y\triangleleft_{A}\Big(x\triangleleft_{A}P\big(P(z)\big)\Big).
\end{align*}
Thus we have
{\small
\begin{eqnarray*}
&&x\circ_{A}(y\circ_{A}z)-(x\circ_{A}y)\circ_{A}z-y\circ_{A}(x\circ_{A}z)\\
&=&P(x)\triangleright_{A}\big(P(y)\triangleright_{A}z\big)
+P(x)\triangleright_{A}\big(y\triangleleft_{A}P(z)\big)
+x\triangleleft_{A}\Big(P\big(P(y)\big)\triangleright_{A}z\Big)
+x\triangleleft_{A}\big(P(y)\triangleright_{A}P(z)\big)\\
&&
+x\triangleleft_{A}\big(P(y)\triangleleft_{A}P(z)\big)
+x\triangleleft_{A}\Big(y\triangleleft_{A}P\big(P(z)\big)\Big)
-\big(P(x)\triangleright_{A}P(y)\big)\triangleright_{A}z
-\Big(P\big(P(x)\big)\triangleright_{A}y\Big)\triangleright_{A}z\\
&&
-\big(P(x)\triangleleft_{A}P(y)\big)\triangleright_{A}z
-\Big(x\triangleleft_{A}P\big(P(y)\big)\Big)\triangleright_{A}z
-\big(P(x)\triangleright_{A}y\big)\triangleleft_{A}P(z)
-\big(x\triangleleft_{A}P(y)\big)\triangleleft_{A}P(z)\\
&&
-P(y)\triangleright_{A}\big(P(x)\triangleright_{A}z\big)
-P(y)\triangleright_{A}\big(x\triangleleft_{A}P(z)\big)
-y\triangleleft_{A}\Big(P\big(P(x)\big)\triangleright_{A}z\Big)
-y\triangleleft_{A}\big(P(x)\triangleright_{A}P(z)\big)\\
&&
-y\triangleleft_{A}\big(P(x)\triangleleft_{A}P(z)\big)
-y\triangleleft_{A}\Big(x\triangleleft_{A}P\big(P(z)\big)\Big)\\
&\overset{\eqref{eq:LN1},\eqref{eq:LN2}}{=}&
P(x)\triangleright_{A}\big(P(y)\triangleright_{A}z\big)
+P(x)\triangleright_{A}\big(y\triangleleft_{A}P(z)\big)
+x\triangleleft_{A}\big(P(y)\triangleright_{A}P(z)\big)
+x\triangleleft_{A}\big(P(y)\triangleleft_{A}P(z)\big)\\
&&
-\big(P(x)\triangleright_{A}P(y)\big)\triangleright_{A}z
-\big(P(x)\triangleleft_{A}P(y)\big)\triangleright_{A}z
-\big(P(x)\triangleright_{A}y\big)\triangleleft_{A}P(z)
-\big(x\triangleleft_{A}P(y)\big)\triangleleft_{A}P(z)\\
&&
-P(y)\triangleright_{A}\big(P(x)\triangleright_{A}z\big)
-P(y)\triangleright_{A}\big(x\triangleleft_{A}P(z)\big)
-y\triangleleft_{A}\big(P(x)\triangleright_{A}P(z)\big)
-y\triangleleft_{A}\big(P(x)\triangleleft_{A}P(z)\big)\\
&\overset{\eqref{eq:pre-L1}-\eqref{eq:pre-L3}}{=}&0.
\end{eqnarray*}}Hence $(A,\circ_{A})$ is a Leibniz algebra.
\end{proof}}

\delete{
\begin{defi}
Let $(A,\succ_{A},\prec_{A})$ be an admissible Novikov dialgebra and $(A,\circ_{A})$ be the sub-adjacent Leibniz algebra.
Suppose that $V$ is a vector space and $l_{\succ_{A}},r_{\succ_{A}},l_{\prec_{A}},r_{\prec_{A}}:A\rightarrow\mathrm{End}_{\mathbb
K}(V)$ are linear maps.
\delete{
\eqref{eq:sum linear} holds.
Set
\begin{equation}\label{eq:sum linear}
l_{\circ_{A}}=l_{\succ_{A}}+l_{\prec_{A}},\;\;
r_{\circ_{A}}=r_{\succ_{A}}+r_{\prec_{A}}.
\end{equation}}
If \eqref{eq:anLei rep3}-\eqref{eq:anLei rep7}, \eqref{eq:anLei rep9}-\eqref{eq:anLei rep11} and the following equations hold:
\begin{eqnarray}
&&-r_{\succ_{A}}(x)l_{\prec_{A}}(y)v=r_{\succ_{A}}(x)r_{\succ_{A}}(y)v,\label{eq:adm Nov diarep1}\\
&&2l_{\prec_{A}}(x\prec_{A}y)v-2l_{\prec_{A}}(y\prec_{A}x)v=
l_{\prec_{A}}(x)l_{\prec_{A}}(y)v-l_{\prec_{A}}(y)l_{\prec_{A}}(x)v
,\label{eq:adm Nov diarep2}\\
&&2r_{\succ_{A}}(x\circ_{A}y)v=r_{\succ_{A}}(y)r_{\succ_{A}}(x)v
+l_{\prec_{A}}(x)r_{\succ_{A}}(y)v,\label{eq:adm Nov diarep3}\\
&&-r_{\succ_{A}}(x)r_{\prec_{A}}(y)v=r_{\succ_{A}}(x)l_{\succ_{A}}(y)v
,\label{eq:adm Nov diarep4}\\
&&2l_{\succ_{A}}(x)r_{\circ_{A}}(y)v=r_{\succ_{A}}(y)l_{\succ_{A}}(x)v
+r_{\prec_{A}}(x\succ_{A}y)v,\label{eq:adm Nov diarep5}\\
&&-l_{\succ_{A}}(y\prec_{A}x)v=l_{\succ_{A}}(x\succ_{A}y)v,\label{eq:adm Nov diarep6}\\
&&2r_{\prec_{A}}(x)r_{\prec_{A}}(y)v-2r_{\prec_{A}}(x)l_{\prec_{A}}(y)v
=r_{\prec_{A}}(y\prec_{A}x)v-l_{\prec_{A}}(y)r_{\prec_{A}}(x)v,\label{eq:adm Nov diarep7}\\
&&2l_{\succ_{A}}(x)l_{\circ_{A}}(y)v=
l_{\succ_{A}}(x\succ_{A}y)v+l_{\prec_{A}}(y)l_{\succ_{A}}(x)v
,\;\forall x,y\in A, v\in V,\label{eq:adm Nov diarep8}
\end{eqnarray}
then we say $(l_{\succ_{A}},r_{\succ_{A}},l_{\prec_{A}},r_{\prec_{A}},V)$ is a {\bf representation} of $(A,\succ_{A},\prec_{A})$.}
\delete{
Two representations $(l_{\succ_{A}},r_{\succ_{A}},$
$l_{\prec_{A}},r_{\prec_{A}},V)$ and $(l'_{\succ_{A}},r'_{\succ_{A}},l'_{\prec_{A}},r'_{\prec_{A}},V')$ of $(A,\succ_{A},\prec_{A})$ are called
\textbf{equivalent} if there exists a linear isomorphism $\phi:V \rightarrow V'$ such that the following equations hold:
\begin{eqnarray*}
\phi l_{\succ_{A}}(x)=l'_{\succ_{A}}(x)\phi,\;\phi r_{\succ_{A}}(x)=r'_{\succ_{A}}(x)\phi,\;\phi l_{\prec_{A}} (x)=l'_{\prec_{A}}(x)\phi,\;\phi r_{\prec_{A}} (x)=r'_{\prec_{A}}(x)\phi,\;\forall x\in A.
\end{eqnarray*}
\end{defi}}

\delete{
\begin{ex}
Let $(A,\succ_{A},\prec_{A})$ be an admissible Novikov dialgebra. Then
$(\mathcal{L}_{\succ_{A}},\mathcal{R}_{\succ_{A}},
\mathcal{L}_{\prec_{A}},\mathcal{R}_{\prec_{A}},A)$
is a representation of $(A,\succ_{A},\prec_{A})$, which is called
the \textbf{adjoint representation}.
\end{ex}}
\delete{
\begin{pro}\label{pro:adNov dual rep}
Let $(A,\succ_{A},\prec_{A})$ be an admissible Novikov dialgebra and $(A,\circ_{A})$ be the sub-adjacent Leibniz algebra. If
$(l_{\succ_{A}},r_{\succ_{A}},l_{\prec_{A}},r_{\prec_{A}},V)$ is a representation of $(A,\succ_{A},\prec_{A})$, then
$$(-l^{*}_{\circ_{A}},-l^{*}_{\prec_{A}}-r^{*}_{\succ_{A}},
l^{*}_{\prec_{A}},l^{*}_{\circ_{A}}+r^{*}_{\circ_{A}},V^{*})$$
is also a representation of $(A,\succ_{A},\prec_{A})$.
In particular, $(-\mathcal{L}^{*}_{\circ_{A}},-\mathcal{L}^{*}_{\prec_{A}}
-\mathcal{R}^{*}_{\succ_{A}},\mathcal{L}^{*}_{\prec_{A}},
\mathcal{L}^{*}_{\circ_{A}}+\mathcal{R}^{*}_{\circ_{A}},
A^{*})$ is a representation of $(A,\succ_{A},\prec_{A})$.
\end{pro}
\begin{proof}
The proof is similar to that of Proposition \ref{pro:anLei dual rep}.
\end{proof}

\begin{defi}
Let $(A,\triangleright_{A},\triangleleft_{A})$ be a Novikov dialgebra and $l_{\triangleright_{A}},r_{\triangleright_{A}},l_{\triangleleft_{A}},r_{\triangleleft_{A}}
:A\rightarrow\mathrm{End}_{\mathbb
K}(V)$ be linear maps.}
\delete{
and $(A,\circ_{A})$ be the sub-adjacent Leibniz algebra.
Let $V$ be a vector space and
Set
\begin{equation}\label{eq:sum linear 2}
l_{\circ_{A}}=l_{\triangleright_{A}}+l_{\triangleleft_{A}},\;\;
r_{\circ_{A}}=r_{\triangleright_{A}}+r_{\triangleleft_{A}}.
\end{equation}
$(l_{\circ_{A}},r_{\circ_{A}},V)$ is a representation of $(A,\circ_{A})$}
\delete{
If the following equations hold:
\begin{eqnarray}
&&r_{\triangleright_{A}}(x\triangleright_{A}y)v
=r_{\triangleright_{A}}(y)r_{\triangleright_{A}}(x)v
+r_{\triangleright_{A}}(y)r_{\triangleleft_{A}}(x)v
+l_{\triangleright_{A}}(x)r_{\triangleright_{A}}(y)v,\label{eq:Nov dialg rep1}\\
&&r_{\triangleleft_{A}}(x\triangleright_{A}y)v
=l_{\triangleright_{A}}(x)r_{\triangleleft_{A}}(y)v
-r_{\triangleleft_{A}}(y)l_{\triangleright_{A}}(x)v
-r_{\triangleleft_{A}}(x\triangleleft_{A}y)v,\label{eq:Nov dialg rep2}\\
&&r_{\triangleleft_{A}}(x)r_{\triangleright_{A}}(y)v
=-r_{\triangleleft_{A}}(x)l_{\triangleleft_{A}}(y)v,\label{eq:Nov dialg rep3}\\
&&r_{\triangleleft_{A}}(x\triangleleft_{A}y)v
=l_{\triangleleft_{A}}(x)r_{\triangleleft_{A}}(y)v,\label{eq:Nov dialg rep4}\\
&&r_{\triangleright_{A}}(x)r_{\triangleleft_{A}}(y)v
=r_{\triangleleft_{A}}(y\triangleright_{A}x)v
=-r_{\triangleright_{A}}(x)l_{\triangleright_{A}}(y)v,\label{eq:Nov dialg rep5}\\
&&l_{\triangleright_{A}}(x)r_{\triangleright_{A}}(y)v
=r_{\triangleright_{A}}(y)l_{\triangleright_{A}}(x)v
+r_{\triangleright_{A}}(y)l_{\triangleleft_{A}}(x)v
+r_{\triangleright_{A}}(x\triangleright_{A}y)v,\label{eq:Nov dialg rep6}\\
&&l_{\triangleleft_{A}}(x)r_{\triangleright_{A}}(y)v
=r_{\triangleright_{A}}(x\triangleleft_{A}y)v
-r_{\triangleleft_{A}}(y)r_{\triangleright_{A}}(x)v
-l_{\triangleleft_{A}}(x)r_{\triangleleft_{A}}(y)v
,\label{eq:Nov dialg rep7}\\
&&r_{\triangleleft_{A}}(x)l_{\triangleright_{A}}(y)v
=-r_{\triangleleft_{A}}(x)r_{\triangleleft_{A}}(y)v
,\label{eq:Nov dialg rep8}\\
&&r_{\triangleright_{A}}(x)l_{\triangleleft_{A}}(y)v
=l_{\triangleleft_{A}}(y)r_{\triangleright_{A}}(x)v
=-r_{\triangleright_{A}}(x)r_{\triangleright_{A}}(y)v
,\label{eq:Nov dialg rep9}\\
&&l_{\triangleright_{A}}(x)l_{\triangleright_{A}}(y)v
=l_{\triangleright_{A}}(x\triangleright_{A}y)v
+l_{\triangleright_{A}}(x\triangleleft_{A}y)v
+l_{\triangleright_{A}}(y)l_{\triangleright_{A}}(x)v,\label{eq:Nov dialg rep10}\\
&&l_{\triangleleft_{A}}(x)l_{\triangleright_{A}}(y)v
=l_{\triangleright_{A}}(y)l_{\triangleleft_{A}}(x)v
-l_{\triangleleft_{A}}(y\triangleright_{A}x)v
-l_{\triangleleft_{A}}(x)l_{\triangleleft_{A}}(y)v,\label{eq:Nov dialg rep11}\\
&&l_{\triangleleft_{A}}(x\triangleright_{A}y)v
=-l_{\triangleleft_{A}}(y\triangleleft_{A}x)v,\label{eq:Nov dialg rep12}\\
&&l_{\triangleleft_{A}}(x)l_{\triangleleft_{A}}(y)v
=l_{\triangleleft_{A}}(y)l_{\triangleleft_{A}}(x)v,\label{eq:Nov dialg rep13}\\
&&l_{\triangleright_{A}}(x\triangleleft_{A}y)v
=l_{\triangleleft_{A}}(x)l_{\triangleright_{A}}(y)v
=-l_{\triangleright_{A}}(y\triangleright_{A}x)v
,\;\forall x,y\in A, v\in V,
\end{eqnarray}
then we say $(l_{\triangleright_{A}},r_{\triangleright_{A}},l_{\triangleleft_{A}},
r_{\triangleleft_{A}},V)$ is a {\bf representation} of $(A,\triangleright_{A},\triangleleft_{A})$.}
\delete{
Two representations $(l _{\triangleright_{A}},r _{\triangleright_{A}},l _{\triangleleft_{A}},$
$r_{\triangleleft_{A}},V)$ and $(l' _{\triangleright_{A}},r' _{\triangleright_{A}},l' _{\triangleleft_{A}},r'_{\triangleleft_{A}},V')$ of $(A,\triangleright_{A},\triangleleft_{A})$ are called
\textbf{equivalent} if there exists a linear isomorphism $\phi:V \rightarrow V'$ such that the following equations hold:
\begin{eqnarray*}
\phi l_{\triangleright_{A}}(x)=l'_{\triangleright_{A}}(x)\phi,\;\phi r_{\triangleright_{A}}(x)=r'_{\triangleright_{A}}(x)\phi,\;\phi l_{\triangleleft_{A}}(x)=l'_{\triangleleft_{A}}(x)\phi,\;\phi
r_{\triangleleft_{A}}(x)=r'_{\triangleleft_{A}}(x)\phi,\;
\forall x\in A.
\end{eqnarray*}
\end{defi}}

\delete{
In fact, for a vector space $V$ and linear maps
$l_{\triangleright_{A}},r_{\triangleright_{A}},l_{\triangleleft_{A}},r_{\triangleleft_{A}}:A\rightarrow \mathrm{End}(V)$, the quadruple
$(l_{\triangleright_{A}},r_{\triangleright_{A}},l_{\triangleleft_{A}},r_{\triangleleft_{A}},V)$ is a representation of the Novikov dialgebra $(A,\triangleright_{A},\triangleleft_{A})$ if and only
if there is a \sdpp structure on $A\oplus V$ given by
\begin{eqnarray}
&&(x+u)\triangleright_{d} (y+v)=x\triangleright_{A}y+l_{\triangleright_{A}}(x)v+r_{\triangleright_{A}}(y)u,\label{eq:sd SDPP1}\\
&&(x+u)\triangleleft_{d} (y+v)=x\triangleleft_{A}y+l_{\triangleleft_{A}}(x)v+r_{\triangleleft_{A}}(y)u,\;\forall x,y\in A,u,v\in V.
\label{eq:sd SDPP2}
\end{eqnarray}
We denote the \sdpp structure on $A\oplus V$ by $A\ltimes
_{l_{\triangleright_{A}},r_{\triangleright_{A}},
l_{\triangleleft_{A}},r_{\triangleleft_{A}}}V$.

\begin{ex}
Let $(A,\triangleright_{A},\triangleleft_{A})$ be a Novikov dialgebra. Then
$(\mathcal{L}_{\triangleright_{A}},\mathcal{R}_{\triangleright_{A}},
\mathcal{L}_{\triangleleft_{A}},\mathcal{R}_{\triangleleft_{A}},A)$
is a representation of $(A,\triangleright_{A},\triangleleft_{A} )$, which is called the \textbf{adjoint representation}.
\end{ex}}
\delete{
\begin{pro}
Let $(A,\triangleright_{A},\triangleleft_{A})$ be a Novikov dialgebra. If
$(l_{\triangleright_{A}},r_{\triangleright_{A}},l_{\triangleleft_{A}},r_{\triangleleft_{A}},V)$ is a
representation of $(A,\triangleright_{A},\triangleleft_{A})$, then
$$(l^{*}_{\triangleright_{A}}-l^{*}_{\triangleleft_{A}},
l^{*}_{\triangleleft_{A}}+r^{*}_{\triangleright_{A}}
,-l^{*}_{\triangleleft_{A}},r^{*}_{\triangleright_{A}}
+l^{*}_{\triangleleft_{A}}-l^{*}_{\triangleright_{A}}
-r^{*}_{\triangleleft_{A}},V^{*})$$
is also a representation of $(A,\triangleright_{A},\triangleleft_{A})$.
In particular, $(\mathcal{L}^{*}_{\triangleright_{A}}-\mathcal{L}^{*}_{\triangleleft_{A}},
\mathcal{L}^{*}_{\triangleleft_{A}}+\mathcal{R}^{*}_{\triangleright_{A}},
-\mathcal{L}^{*}_{\triangleleft_{A}},\mathcal{R}^{*}_{\triangleright_{A}}
+\mathcal{L}^{*}_{\triangleleft_{A}}-\mathcal{L}^{*}_{\triangleright_{A}}
-\mathcal{R}^{*}_{\triangleleft_{A}},A^{*})$ is a representation of $(A,\triangleright_{A},\triangleleft_{A})$, which is called the {\bf coadjoint representation} of $(A,\triangleright_{A},\triangleleft_{A})$.
\end{pro}
\begin{proof}
By the assumptions, $A\ltimes
_{l_{\triangleright_{A}},r_{\triangleright_{A}},l_{\triangleleft_{A}},r_{\triangleleft_{A}}}V$  is a Novikov dialgebra. By Proposition \ref{pro:-2-alg}, the $-2$-algebra of $A\ltimes
_{l_{\triangleright_{A}},r_{\triangleright_{A}},l_{\triangleleft_{A}},
r_{\triangleleft_{A}}}V$
is an admissible Novikov dialgebra
$A\ltimes
_{l_{\succ_{A}},r_{\succ_{A}},
l_{\prec_{A}},r_{\prec_{A}}}V$
for $l_{\succ_{A}}=l_{\triangleright_{A}}-2l_{\triangleleft_{A}},
r_{\succ_{A}}=r_{\triangleright_{A}}-2r_{\triangleleft_{A}},
l_{\prec_{A}}=l_{\triangleleft_{A}}-2l_{\triangleright_{A}},
r_{\prec_{A}}=r_{\triangleleft_{A}}-2r_{\triangleright_{A}}$.
Then $(l_{\succ_{A}},r_{\succ_{A}},
l_{\prec_{A}},r_{\prec_{A}},V)$ is a representation of the admissible Novikov dialgebra $(A,\succ_{A}=\triangleright_{A}-2\triangleleft_{A},
\prec_{A}=\triangleleft_{A}-2\triangleright_{A})$. By Proposition \ref{pro:adNov dual rep}, $(-l^{*}_{\circ_{A}}=l^{*}_{\triangleright_{A}}+l^{*}_{\triangleleft_{A}},
-l^{*}_{\prec_{A}}-r^{*}_{\succ_{A}}=-l^{*}_{\triangleleft_{A}}+2l^{*}_{\triangleright_{A}}
-r^{*}_{\triangleright_{A}}+2r^{*}_{\triangleleft_{A}},
l^{*}_{\prec_{A}}=l^{*}_{\triangleleft_{A}}-2l^{*}_{\triangleright_{A}},
l^{*}_{\circ_{A}}+r^{*}_{\circ_{A}}=-l^{*}_{\triangleleft_{A}}
-l^{*}_{\triangleright_{A}}-r^{*}_{\triangleright_{A}}-r^{*}_{\triangleleft_{A}},V^{*})$
is also a representation of $(A,\succ_{A},\prec_{A})$. Again by
Proposition \ref{pro:2-alg}, $2$-algebra of the admissible Novikov dialgebra $A\ltimes
_{-l^{*}_{\circ_{A}},-l^{*}_{\prec_{A}}-r^{*}_{\succ_{A}},
l^{*}_{\prec_{A}},l^{*}_{\circ_{A}}+r^{*}_{\circ_{A}}}V^*$
is a Novikov dialgebra
$A\ltimes
_{l^{*}_{\triangleright_{A}}-l^{*}_{\triangleleft_{A}},
l^{*}_{\triangleleft_{A}}+r^{*}_{\triangleright_{A}}
,-l^{*}_{\triangleleft_{A}},r^{*}_{\triangleright_{A}}
+l^{*}_{\triangleleft_{A}}-l^{*}_{\triangleright_{A}}
-r^{*}_{\triangleleft_{A}}}V^*$. Thus $(l^{*}_{\triangleright_{A}}-l^{*}_{\triangleleft_{A}},
l^{*}_{\triangleleft_{A}}+r^{*}_{\triangleright_{A}}
,-l^{*}_{\triangleleft_{A}},r^{*}_{\triangleright_{A}}
+l^{*}_{\triangleleft_{A}}-l^{*}_{\triangleright_{A}}
-r^{*}_{\triangleleft_{A}},V^{*})$
is a representation of $(A,\triangleright_{A},\triangleleft_{A})$.
\end{proof}

\begin{defi}
A \textbf{quadratic Novikov dialgebra} $(A,\triangleright_{A},\triangleleft_{A},\omega)$ is a Novikov dialgebra $(A,\triangleright_{A},\triangleleft_{A})$ equipped with a nondegenerate skew-symmetric bilinear form $\omega$ satisfying the following \textbf{invariant} condition:
\begin{eqnarray}
\omega(x\triangleright_{A}y,z)&=&
-\omega(x,z\triangleright_{A}y+y\triangleleft_{A}z),\label{eq:LNbf1}\\
\omega(x\triangleleft_{A}y,z)&=&
\omega(x,y\triangleright_{A}z-z\triangleright_{A}y+
z\triangleleft_{A}y-y\triangleleft_{A}z),\;\forall x,y,z\in A.\label{eq:LNbf2}
\end{eqnarray}
\end{defi}

\begin{pro}
Let $(A,\triangleright_{A},\triangleleft_{A},\omega)$ be a quadratic Novikov dialgebra. Then $(\mathcal{L}_{\triangleright_{A}},\mathcal{R}_{\triangleright_{A}},
\mathcal{L}_{\triangleleft_{A}},\mathcal{R}_{\triangleleft_{A}},A)$
and
$(\mathcal{L}^{*}_{\triangleright_{A}}-\mathcal{L}^{*}_{\triangleleft_{A}},
\mathcal{L}^{*}_{\triangleleft_{A}}+\mathcal{R}^{*}_{\triangleright_{A}},
-\mathcal{L}^{*}_{\triangleleft_{A}},\mathcal{R}^{*}_{\triangleright_{A}}
+\mathcal{L}^{*}_{\triangleleft_{A}}-\mathcal{L}^{*}_{\triangleright_{A}}
-\mathcal{R}^{*}_{\triangleleft_{A}},A^{*})$
are equivalent as representations of $(A,\triangleright_{A},\triangleleft_{A})$.
\end{pro}
\begin{proof}
Let $x,y,z\in A$. Then we have
\begin{eqnarray}
\omega(x\triangleright_{A}y,z)\overset{\eqref{eq:LNbf1},\eqref{eq:LNbf2}}{=}
-\omega(x\triangleleft_{A}z,y)
-\omega(x,y\triangleright_{A}z+z\triangleleft_{A}y)
\overset{\eqref{eq:LNbf1}}{=}
\omega(y,x\triangleleft_{A}z)-\omega(y,x\triangleright_{A}z).\label{eq:cor3.27}
\end{eqnarray}
Observing the RHS of \eqref{eq:LNbf2} is skew-symmetric in $y$ and $z$, we have
\begin{equation}\label{eq:cor3.37}
\omega(x\triangleleft_{A}y,z)=-\omega(x\triangleleft_{A}z,y)
=\omega(y,x\triangleleft_{A}z).
\end{equation}
Thus by \eqref{eq:cor3.27}, \eqref{eq:LNbf1}, \eqref{eq:cor3.37} and \eqref{eq:LNbf2}, we respectively have
\begin{eqnarray*}
&&\omega^{\natural}\big(\mathcal{L}_{\triangleright_{A}}(x)y\big)
=(\mathcal{L}^{*}_{\triangleright_{A}}-\mathcal{L}^{*}_{\triangleleft_{A}})(x)\omega^{\natural}(y),\;
\omega^{\natural}\big(\mathcal{R}_{\triangleright_{A}}(y)x\big)
=(\mathcal{L}^{*}_{\triangleleft_{A}}+\mathcal{R}^{*}_{\triangleright_{A}})(y)\omega^{\natural}(x),\\
&&\omega^{\natural}\big(\mathcal{L}_{\triangleleft_{A}}(x)y\big)
=-\mathcal{L}^{*}_{\triangleleft_{A}}(x)\omega^{\natural}(y),\;
\omega^{\natural}\big(\mathcal{R}_{\triangleleft_{A}}(y)x\big)
=(\mathcal{R}^{*}_{\triangleright_{A}}
+\mathcal{L}^{*}_{\triangleleft_{A}}-\mathcal{L}^{*}_{\triangleright_{A}}
-\mathcal{R}^{*}_{\triangleleft_{A}})(y)\omega^{\natural}(x),\;\forall x,y\in A.
\end{eqnarray*}
Hence, the bijection $\omega^{\natural}:A\rightarrow A^*$ gives the equivalence between $(\mathcal{L}_{\triangleright_{A}},\mathcal{R}_{\triangleright_{A}},
\mathcal{L}_{\triangleleft_{A}},\mathcal{R}_{\triangleleft_{A}},A)$
and
$(\mathcal{L}^{*}_{\triangleright_{A}}-\mathcal{L}^{*}_{\triangleleft_{A}},
\mathcal{L}^{*}_{\triangleleft_{A}}+\mathcal{R}^{*}_{\triangleright_{A}},
-\mathcal{L}^{*}_{\triangleleft_{A}},\mathcal{R}^{*}_{\triangleright_{A}}
+\mathcal{L}^{*}_{\triangleleft_{A}}-\mathcal{L}^{*}_{\triangleright_{A}}
-\mathcal{R}^{*}_{\triangleleft_{A}},A^{*})$
as representations of $(A,\triangleright_{A},\triangleleft_{A})$.
\end{proof}

\begin{pro}
Let $(A,\triangleright_{A},\triangleleft_{A},\omega)$ be a quadratic Novikov dialgebra.
Suppose that $Q$ is a derivation on $(A,\triangleright_{A},\triangleleft_{A})$ such that
\begin{equation}\label{eq:adjoint map}
\omega\big(Q(x),y\big)=-\omega\big(x,Q(y)\big),\;\forall x,y\in A.
\end{equation}
Then $(A,\circ_{A},\omega)$ with a multiplication $\circ_{A}:A\otimes A\rightarrow A$ given by \eqref{eq:derivation LN}
is a quadratic Leibniz algebra.
\end{pro}

\begin{proof}
By proposition \ref{pro:LN and Leib}, $(A,\circ_{A})$ is a Leibniz algebra.
Let $x,y,z\in A$, then we have
\begin{eqnarray*}
\omega(x,y\circ_{A}z)&\overset{\eqref{eq:derivation LN}}{=}&\omega\big(x,Q(y)\triangleright_{A}z+y\triangleleft_{A}Q(z)\big)\\
&\overset{\eqref{eq:LNbf1},\eqref{eq:LNbf2}}{=}&\omega\big(Q(y),x\triangleright_{A}z\big)+\omega\big(Q(y),z\triangleleft_{A}x\big)
-\omega\big(x\triangleright_{A}Q(z),y\big)+\omega\big(x\triangleleft_{A}Q(z),y\big)\\
&&
+\omega\big(Q(z)\triangleright_{A}x,y\big)-\omega\big(Q(z)\triangleleft_{A}x,y\big)\\
&\overset{\eqref{eq:adjoint map}}{=}&\omega\big(Q(x\triangleright_{A}z),y\big)+\omega\big(Q(z\triangleleft_{A}x),y\big)
-\omega\big(x\triangleright_{A}Q(z),y\big)+\omega\big(x\triangleleft_{A}Q(z),y\big)\\
&&
+\omega\big(Q(z)\triangleright_{A}x,y\big)-\omega\big(Q(z)\triangleleft_{A}x,y\big)\\
&\overset{\eqref{eq:Ao}}{=}&\omega\big(Q(x)\triangleright_{A}z,y\big)+\omega\big(z\triangleleft_{A}Q(x),y\big)
+\omega\big(x\triangleleft_{A}Q(z),y\big)+\omega\big(Q(z)\triangleright_{A}x,y\big)\\
&=&\omega(x\circ_{A}z+z\circ_{A}x,y).
\end{eqnarray*}
Hence the conclusion follows.
\end{proof}

\begin{pro}
Let $(A,\triangleright_{A},\triangleleft_{A},\omega)$ be a quadratic Novikov dialgebra and the sub-adjacent Leibniz algebra be $(A,\circ_{A})$. Then $\omega$ is a nondegenerate skew-symmetric 2-cocyle on $(A,\circ_{A})$.
\end{pro}

\begin{proof}
Let $x,y,z\in A$, then we have
\begin{eqnarray*}
&&\omega(x,y\circ_{A}z)+\omega(z,y\circ_{A}x)-\omega(y,z\circ_{A}x+x\circ_{A}z)\\
&=&\omega(x,y\triangleright_{A}z)+\omega(x,y\triangleleft_{A}z)+\omega(z,y\triangleright_{A}x)+\omega(z,y\triangleleft_{A}x)
-\omega(y,z\triangleright_{A}x)-\omega(y,z\triangleleft_{A}x)\\
&&
-\omega(y,x\triangleright_{A}z)-\omega(y,x\triangleleft_{A}z)\\
&\overset{\eqref{eq:LNbf2}}{=}&\omega(x,y\triangleright_{A}z)+\omega(z,y\triangleright_{A}x)-\omega(y,z\triangleright_{A}x)
-\omega(y,z\triangleleft_{A}x)-\omega(y,x\triangleright_{A}z)-\omega(y,x\triangleleft_{A}z)
\overset{\eqref{eq:LNbf1}}{=}0.
\end{eqnarray*}
Hence the conclusion follows.
\end{proof}

\begin{pro}\label{pro:bf}
Let $A$ be a vector space with multiplications $\triangleright_{A},\triangleleft_{A}:A\otimes A\rightarrow A$
and $\omega$ be a nondegenerate skew-symmetric bilinear form on $A$. Suppose that $\widehat{A}=A\otimes \mathbb {K}[t,t^{-1}]$ with a multiplication $\circ_{\widehat{A}}:\widehat{A}\otimes \widehat{A}\rightarrow \widehat{A}$ given by \eqref{eq:tensor product}. Let $\widetilde{\omega}$ be a nondegenerate skew-symmetric bilinear form on $\widehat{A}$ satisfying
\begin{equation}
\widetilde{\omega}(x\otimes t^i,y\otimes t^j)=\omega(x,y)\mathcal{B}(t^i,t^j),\;\forall x,y\in A,i,j\in \mathbb{Z},
\end{equation}
where $\mathcal{B}$ is a nondegenerate symmetric bilinear form on $\mathbb {K}[t,t^{-1}]$ given by
\begin{equation}
\mathcal{B}(t^i,t^j)=\delta_{i+j,-1},\;\forall i,j\in \mathbb{Z}.
\end{equation}
Then $(\widehat{A},\circ_{\widehat{A}},\widetilde{\omega})$ is a quadratic Leibniz algebra if and only if $(A,\triangleright_{A},\triangleleft_{A},\omega)$ is a quadratic Novikov dialgebra.
\end{pro}

\begin{proof}
Let $x,y,z\in A$. Then we have
\begin{eqnarray*}
\widetilde{\omega}\big(x\otimes t^i,(y\otimes t^j)\circ_{\widehat{A}}(z\otimes t^k)\big)&=&\widetilde{\omega}\big(x\otimes t^i,(jy\triangleright_{A} z+ky\triangleleft_{A}z)\otimes t^{j+k-1}\big)\\
&=&j\omega(x,y\triangleright_{A} z)\delta_{i+j+k-1,-1}+k\omega(x,y\triangleleft_{A}z)\delta_{i+j+k-1,-1}\\
&=&j\omega(x,y\triangleright_{A} z)\delta_{i,-j-k}+k\omega(x,y\triangleleft_{A}z)\delta_{i,-j-k},\\
\widetilde{\omega}\big((x\otimes t^i)\circ_{\widehat{A}}(z\otimes t^k),y\otimes t^j\big)&=&\widetilde{\omega}\big((ix\triangleright_{A} z+kx\triangleleft_{A}z)\otimes t^{i+k-1},y\otimes t^{j}\big)\\
&=&i\omega(x\triangleright_{A} z,y)\delta_{i,-j-k}+k\omega(x\triangleleft_{A}z,y)\delta_{i,-j-k},\\
\widetilde{\omega}\big((z\otimes t^k)\circ_{\widehat{A}}(x\otimes t^i),y\otimes t^j\big)&=&
k\omega(z\triangleright_{A}x,y)\delta_{i,-j-k}
+i\omega(z\triangleleft_{A}x,y)\delta_{i,-j-k}.
\end{eqnarray*}
Suppose that $(\widehat{A},\circ_{\widehat{A}},\widetilde{\omega})$ is a quadratic Leibniz algebra. Then we obtain
\begin{eqnarray}
&&j\omega(x,y\triangleright_{A} z)\delta_{i,-j-k}+k\omega(x,y\triangleleft_{A}z)\delta_{i,-j-k}\nonumber\\
&&=i\omega(x\triangleright_{A} z,y)\delta_{i,-j-k}+k\omega(x\triangleleft_{A}z,y)\delta_{i,-j-k}
+k\omega(z\triangleright_{A}x,y)\delta_{i,-j-k}+i\omega(z\triangleleft_{A}x,y)\delta_{i,-j-k}.\ \ \label{eq:bf}
\end{eqnarray}
Taking $i=1,j=-1,k=0$ and $i=-1,j=0,k=1$ in \eqref{eq:bf}, then we have \eqref{eq:LNbf1} and \eqref{eq:LNbf2} respectively.
Conversely, let $(A,\triangleright_{A},\triangleleft_{A},\omega)$ be a quadratic Novikov dialgebra,
then we have
\begin{eqnarray*}
&&\widetilde{\omega}\big(x\otimes t^i,(y\otimes t^j)\circ_{\widehat{A}}(z\otimes t^k)\big)\\
&=&j\omega(x,y\triangleright_{A} z)\delta_{i,-j-k}+k\omega(x,y\triangleleft_{A}z)\delta_{i,-j-k}\\
&=&-j\omega(x\triangleright_{A} z,y)\delta_{i,-j-k}-j\omega(z\triangleleft_{A}x,y)\delta_{i,-j-k}
-k\omega(x\triangleright_{A}z,y)\delta_{i,-j-k}\\
&&
+k\omega(x\triangleleft_{A}z,y)\delta_{i,-j-k}
+k\omega(z\triangleright_{A}x,y)\delta_{i,-j-k}
-k\omega(z\triangleleft_{A}x,y)\delta_{i,-j-k}\\
&=&i\omega(x\triangleright_{A}z,y)\delta_{i,-j-k}
+k\omega(x\triangleleft_{A}z,y)\delta_{i,-j-k}
+k\omega(z\triangleright_{A}x,y)\delta_{i,-j-k}
+i\omega(z\triangleleft_{A}x,y)\delta_{i,-j-k}\\
&=&\widetilde{\omega}\big((x\otimes t^i)\circ_{\widehat{A}}(z\otimes t^k),y\otimes t^j\big)+\widetilde{\omega}\big((z\otimes t^k)\circ_{\widehat{A}}(x\otimes t^i),y\otimes t^j\big).
\end{eqnarray*}
Hence $\widetilde{\omega}$ is invariant on $(\widehat{A},\circ_{\widehat{A}})$ if and only if $\omega$ is invariant on $(A,\triangleright_{A},\triangleleft_{A})$.
Joint with Proposition \ref{pro:tensor product}, the conclusion follows.
\end{proof}}

\delete{
and Proposition \ref{pro:bf}, we get the following conclusion.
\begin{cor}
$(\widehat{A},\circ_{\widehat{A}},\widetilde{\omega})$ is a quadratic Leibniz algebra if and only if $(A,\triangleright_{A},\triangleleft_{A},\omega)$ is a quadratic Novikov dialgebra.
\end{cor}}

\bigskip

\noindent{\bf Acknowledgements.}
This work is supported by NSFC
(12401031, W2412041),
the Postdoctoral Fellowship Program of
CPSF (GZC20240755, 2024T005TJ, 2024M761507) and Nankai Zhide Foundation.

\noindent{\bf Declaration of interests.} 
The authors have no conflicts of interest to disclose.

\noindent{\bf Data availability.} 
Data sharing is not applicable to this article as no new data were created or analyzed.

\end{document}